%% file: main.tex
\documentclass[a4paper,twoside,10pt]{article}
\usepackage[a4paper,left=2.5cm,right=2.5cm, top=2.5cm, bottom=2.5cm]{geometry}
\usepackage[latin1]{inputenc}
\usepackage{cuted}
\usepackage{orcidlink}
\definecolor{darkblue}{rgb}{0.0, 0.0, 0.55}
\definecolor{deepaqua}{RGB}{0, 197, 189}
\definecolor{edgepurple}{RGB}{160, 43, 167}
\usepackage{amsthm}
\usepackage{amssymb}
\usepackage{mathrsfs}
\usepackage{mathtools}
\usepackage[shortlabels]{enumitem}
\usepackage{accents}
\usepackage{multirow}
\usepackage{diagbox}
\usepackage{graphicx}
\usepackage[ruled, linesnumbered]{algorithm2e} 
\usepackage{amscd,amsmath,amssymb,mathrsfs,bbm,listings}
\usepackage{cancel}
\graphicspath{{figs/}}
\usepackage{amsmath}
\usepackage{orcidlink}
\usepackage{footmisc}
\usepackage{amsthm}
\usepackage{amssymb}
\usepackage{stmaryrd}
\SetSymbolFont{stmry}{bold}{U}{stmry}{m}{n}
\usepackage{float}
\usepackage{bigints}
\usepackage{cite}
\usepackage{color}
\usepackage[abs]{overpic}
\usepackage[font=footnotesize,labelfont=bf]{caption}
\usepackage{cases}
\usepackage{tikz}
\usepackage{rotating}
\usepackage{blkarray}
\usepackage{filecontents}
\usepackage{pgfplots}
\usetikzlibrary{pgfplots.groupplots}
\usepackage{subcaption}
\usetikzlibrary{matrix,calc,arrows,cd}
\usepackage{soul,xcolor}
\usepackage{enumitem}
\usepackage{verbatim}
\usepackage{graphicx}
\usepackage{subcaption}
\usepackage{mwe}
\usepackage{tikz}
\usepackage{hyperref}
\usetikzlibrary{external}
\newtheorem{theorem}{Theorem}[section]
\newtheorem{assumption}[theorem]{Assumption}
\newtheorem{lemma}[theorem]{Lemma}

\newtheorem{remark}[theorem]{Remark}
\newtheorem{definition}[theorem]{Definition}

\newcommand{\eremk}{\hbox{}\hfill\rule{0.8ex}{0.8ex}}

\newcolumntype{C}{>{\centering\arraybackslash}p{1.50cm}}

\usepackage{varwidth}

\newcommand{\R}{\mathbb{R}}

\DeclareMathOperator*{\diam}{diam}
\DeclareMathOperator*{\card}{card}
\newcommand{\hK}{h_K}
\newcommand{\hcalK}{h_{\calK}}
\newcommand{\hKF}{h_{K_F}}
\newcommand{\KF}{K_F}
\newcommand{\calK}{\mathcal{K}}
\newcommand{\pK}{p_K}
\newcommand{\pcalK}{p_{\mathcal{K}}}
\newcommand{\pKF}{p_{K_F}}
\newcommand{\pKr}{p_{K_r}}
\newcommand{\lK}{l_{\calK}}
\newcommand{\ellK}{\ell_{\calK}}
\newcommand{\lKr}{\Lambda_{K_r}}
\newcommand{\p}{\boldsymbol{p}}
\newcommand{\mvl}[1]{\{\!\!\{#1\}\!\!\}}
\newcommand{\jump}[1]{\llbracket #1 \rrbracket_{\mathsf{N}}}
\newcommand{\Omegah}{\Omega_h}
\newcommand{\Tchar}{\mathcal{T}^{\#}}
\newcommand{\FK}{\mathcal{F}_K}
\newcommand{\FKd}{\mathcal{F}_K^{\small{\diamond}}}
\newcommand{\sKF}{s_K^F}
\newcommand{\NK}{\mathcal{N}_K}
\newcommand{\NKr}{\mathcal{N}_{K_r}}
\newcommand{\FKs}{\mathcal{F}_{K_s}}
\newcommand{\FKr}{\mathcal{F}_{K_r}}
\newcommand{\FKout}{\mathcal{F}_K^{\mathrm{out}}}
\newcommand{\FKFout}{\mathcal{F}_{K_F}^{\mathrm{out}}}
\newcommand{\FKrout}{\mathcal{F}_{K_r}^{\mathrm{out}}}
\newcommand{\Fh}{\mathcal{F}_h}
\newcommand{\Fhs}{\mathcal{F}_h^{\star}}
\newcommand{\FhI}{\mathcal{F}_h^{\mathcal{I}}}
\newcommand{\FhD}{\mathcal{F}_h^{\mathcal{D}}}
\newcommand{\FhN}{\mathcal{F}_h^{\mathcal{N}}}
\newcommand{\NF}{\nu_F^{\mathrm{out}}}
\newcommand{\NKout}{\mathcal{N}_K^{\mathrm{out}}}
\newcommand{\NKrout}{\mathcal{N}_{K_r}^{\mathrm{out}}}
\newcommand{\Nh}{\mathcal{N}_h^{\mathrm{ele}}}
\newcommand{\LDGf}{\mathrm{LDG}_{\mathsf{f}}}
\newcommand{\LDGw}{\mathrm{LDG}_{\mathsf{w}}}
\newcommand{\xiF}{\chi_F}
\newcommand{\br}{\boldsymbol{r}}
\newcommand{\uh}{u_h}
\newcommand{\UK}{\mathcal{U}_{\calK}}
\newcommand{\vphi}{\boldsymbol{\phi}}
\newcommand{\qh}{\q_h}
\newcommand{\sigmah}{\boldsymbol{\sigma}_h}
\newcommand{\sh}{\boldsymbol{s}_h}
\newcommand{\rh}{\boldsymbol{r}_h}
\newcommand{\vh}{v_h}
\newcommand{\wh}{w_h}
\newcommand{\Vhp}{\mathcal{V}_{h}^{\boldsymbol{p}}}
\newcommand{\Mhp}{\boldsymbol{\mathcal{M}}_{h}^{\boldsymbol{p}}}
\newcommand{\Pp}[2]{\mathcal{P}^{#1}(#2)}
\newcommand{\bs}{\boldsymbol{\psi}}
\newcommand{\bS}{\boldsymbol{\Psi}}
\newcommand{\uflux}{\widehat{u}_h}
\newlength{\dhatheight}
\newcommand\sflux{
    \settoheight{\dhatheight}{\ensuremath{\widehat{q_h}}}
    \addtolength{\dhatheight}{-0.40ex}
    \widehat{\vphantom{\rule{1pt}{\dhatheight}}
    \smash{\widehat{\bsigma}}}_h
}
\newcommand{\nK}{\boldsymbol{n}_{K}}
\newcommand{\nKr}{\boldsymbol{n}_{K_r}}
\newcommand{\nF}{\boldsymbol{n}_F}
\newcommand{\nKo}{\boldsymbol{n}_{K_1}}
\newcommand{\Lh}{\mathcal{L}_h}
\newcommand{\LhD}{\mathcal{L}_h^{\mathcal{D}}}
\newcommand{\LF}{\mathcal{L}^F}

\newcommand{\LD}{\mathcal{L}_{\mathcal{D}}^F}

\newcommand{\Ah}{\mathcal{A}_h}
\newcommand{\calA}{\mathcal{A}}

\newcommand{\lh}{\ell_h}
\newcommand{\Seminorm}[2]{|#1|_{#2}}
\newcommand{\Norm}[2]{\|#1\|_{#2}}
\newcommand{\Tnorm}[2]{|\!|\!|#1|\!|\!|_{#2}}
\newcommand{\CDG}{_{\mathrm{CDG}}}
\newcommand{\CA}{C_{\mathcal{A}}}
\newcommand{\mh}{\boldsymbol{m}_h}
\renewcommand{\dh}{\boldsymbol{d}_h}
\newcommand{\bgradh}{\boldsymbol{b}_h^{\nabla}}
\newcommand{\bav}{\boldsymbol{b}_h^{\mathrm{av}}}

\newcommand{\gh}{g_h}
\DeclareMathOperator*{\essinf}{ess\,inf}
\newcommand{\nablah}{\nabla_h}
\newcommand{\Id}{\boldsymbol{\mathrm{I}}}
\newcommand{\Pih}{\boldsymbol{\Pi}_h}
\newcommand{\M}{\boldsymbol{M}}
\newcommand{\D}{\boldsymbol{D}}
\newcommand{\B}{\boldsymbol{B}}
\newcommand{\Bgrad}{\boldsymbol{B}^{\nabla}}
\newcommand{\Bav}{\boldsymbol{B}^{\mathrm{av}}}
\newcommand{\Bavl}{\boldsymbol{B}^{\mathrm{av}(\ell)}}
\newcommand{\calD}{\boldsymbol{\cal{D}}}
\newcommand{\Pihp}{\widetilde{\Pi}_{hp}}
\newcommand{\PPihp}{\widetilde{\boldsymbol{\Pi}}_{hp}}

\newcommand{\gD}{g_{\mathrm{D}}}
\newcommand{\GD}{\Gamma_{\mathrm{D}}}
\newcommand{\gN}{g_{\mathrm{N}}}
\newcommand{\GN}{\Gamma_{\mathrm{N}}}
\newcommand{\nOmega}{\boldsymbol{n}_{\Omega}}
\newcommand{\bx}{\boldsymbol{x}}

\newcommand{\q}{\boldsymbol{q}}
\newcommand{\bsigma}{\boldsymbol{\sigma}}
\newcommand{\dx}{\,\mathrm{d}\bx}
\newcommand{\dS}{\,\mathrm{d}S}
\newcommand{\bk}{\boldsymbol{\kappa}}
\renewcommand{\v}{\boldsymbol{v}}

\newcommand{\Ctr}{C_{\mathrm{tr}}}

\newcolumntype{C}{>{\centering\arraybackslash}p{1.50cm}}

\numberwithin{equation}{section}
\hypersetup{
    colorlinks,
    linktoc=section,
    citecolor=red,
    linkcolor=darkblue,
    urlcolor=magenta,
    pdfborder={0 0 0}
} 

\title{On the Compact Discontinuous Galerkin method \\
for polytopal meshes}
\author{Mattia Corti\thanks{Faculty of Mathematics, University of Vienna, Oskar-Morgenstern-Platz 1, Vienna, 1090, Austria} 
\thanks{MOX-Dipartimento di Matematica, Politecnico di Milano, Piazza Leonardo da Vinci 32, 20133 Milan, Italy (\href{mailto:mattia.corti@polimi.it}{mattia.corti@polimi.it})}\ \orcidlink{0000-0002-7014-972X} \and Sergio G\'omez\thanks{Department of Mathematics and Applications, University of Milano-Bicocca, Via Cozzi 55, 20125 Milan, Italy (\href{mailto:sergio.gomezmacias@unimib.it}{sergio.gomezmacias@unimib.it})}
\thanks{IMATI-CNR ``E. Magenes", Via Ferrata 5, 27100 Pavia, Italy}\ \orcidlink{0000-0001-9156-5135}}
\date{}

\begin{document}

\maketitle

\begin{abstract}
\noindent 
The Compact Discontinuous Galerkin method was introduced by Peraire and Persson in (SIAM J. Sci. Comput., 30, 1806--1824, 2008).
In this work, we present the stability and convergence analysis for the~$hp$-version of this method applied to  elliptic problems on polytopal meshes. 
Moreover, we introduce fast and practical algorithms that allow the CDG, LDG, and BR2 methods to be implemented within a unified framework.
Our numerical experiments show that the CDG method yields a compact stencil for the stiffness matrix, with faster assembly and solving times compared to the LDG and BR2 methods. 
We numerically study how coercivity depends on the method parameters for various mesh types, with particular focus on the number of facets per mesh element. Finally, we demonstrate the importance of choosing the correct directions for the numerical fluxes when using variable polynomial degrees.
\end{abstract}



\section{Introduction}
In this work, we study some theoretical and computational aspects of the Compact Discontinuous Galerkin (CDG) method for elliptic problems on polytopal meshes. 
The CDG method was first introduced by Peraire and Persson in~\cite{Periaire_Persson:2008} and 
combines features of the Local Discontinuous Galerkin (LDG) method proposed by Cockburn and Shu~\cite{Cockburn_Shu:1998}, and of the modified version 
of the Bassi--Rebay method (BR2) in~\cite{Bassi_Rebay:1997}. The CDG method has also been analyzed for nonlinear convection--diffusion problems~\cite{Brdar_Dedner_Klofkorn:2012} and nearly incompressible linear elasticity models~\cite{Huang_Huang:2013}. In addition, geometric multigrid solvers for the CDG method were investigated in~\cite{Pan_Persson:2022}.

The most attractive feature of the CDG method is that, due to its \emph{built-in stabilization}, it shares the outstanding stability properties of the LDG method while recovering a \emph{compact stencil} similar to that of the BR2 method, i.e., only the degrees of freedom (DoFs) belonging to neighboring elements are connected in the discretization. 
Such a compact stencil is typical of methods formulated directly for the primal variable, e.g., the interior-penalty discontinuous Galerkin (IPDG) \cite{Arnold:1982} and the direct discontinuous Galerkin (DDG) \cite{Liu_Yan:2008} methods. 
In contrast, in the LDG framework, an arbitrary choice of the weighted-average parameters in the definition of the numerical fluxes may lead to a \emph{full stencil}, involving not only immediate neighbors but also their neighbors~(see~\cite[\S3.3]{Sherwin_etal:2006} and~\cite[\S4.1]{Castillo:2002}). 
In~\cite{Castillo:2010}, some heuristic algorithms were studied to determine suitable choices of the numerical fluxes that reduce the stencil of the LDG method for simplicial meshes. However, the discussion in~\cite{Castillo:2010} makes it clear that a compact stencil cannot be achieved for arbitrary unstructured meshes. Finally, the CDG method 
shares the characteristic feature of the LDG method of using weighted averages in the definition of the numerical fluxes, which allows for a reduction 
of the interface integrals to compute, saving %
computational cost of assembling the stiffness matrix compared to the BR2 method. This is of particular interest 
when agglomerated meshes with a large number of facets for each element are considered (see~\cite{Bassi_etal:2012,antonietti_polytopal_2026}).

The stencil of a method directly impacts both computational efficiency and memory requirements. %
A reduction of the stencil is particularly relevant for polytopal meshes, where each element may have several neighbors, %
which can lead to a dramatic increase in the density of the associated matrices. In recent decades, there has been a steadily growing interest in numerical methods designed to handle polytopal meshes, as they allow for a more efficient approximation of complex geometries or interfaces. 
Such a class includes the virtual element method (VEM) \cite{Beirao_etal:2013}, the Hybrid High-Order (HHO) method~\cite{DiPietro_Ern_Lemaire:2014}, and discontinuous Galerkin (DG) methods \cite{Wihler:2002,Cangiani_Georgoulis_Houston:2014,Cangiani_Dong_Georgoulis_Houston:2017}.

In this work, we focus on DG methods, which have the advantage that the number of DoFs is independent of the geometry of the elements.

\paragraph{Model problem} Let~$\Omega \subset \R^d$ ($d \in \{2, 3\}$) be an open, bounded domain with Lipschitz boundary~$\partial \Omega = \GD \cup \GN$, where~$\GD$ and~$\GN$ are disjoint sets ($\GD \cap \GN = \emptyset$) and denote the parts of the boundary of~$\Omega$ where Dirichlet and Neumann boundary conditions are imposed, respectively. We further assume that~$\GD$ has strictly positive~$(d-1)$-dimensional measure.

Given a symmetric diffusion tensor~$\bk \in L^{\infty}(\Omega)^{d\times d}$ with
\begin{equation}
\label{eq:nondegeneracy}
k^{\star} := \Norm{\bk}{L^{\infty}(\Omega)^{d\times d}} \quad \text{and} \quad k_{\star} := \essinf_{\bx \in \Omega} \lambda_{\min}(\bk(\bx)) > 0,
\end{equation}
a source term~$f \in L^2(\Omega)$, a Dirichlet datum~$\gD \in H^{1/2}(\GD)$, and a Neumann datum~$\gN \in L^2(\GN)$, we consider the following boundary value problem (BVP): find~$u: \Omega \to \R$ such that
\begin{subequations}
\label{eq:model}
\begin{alignat}{3}
\label{eq:model-1}
-\nabla \cdot (\bk \nabla u) & = f & & \quad \text{ in } \Omega, \\
u & = \gD & & \quad \text{ on } \GD, \\
\bk \nabla u \cdot \nOmega & = \gN & & \quad \text{ on } \GN,
\end{alignat}
\end{subequations}
where~$\nOmega$ is the unit normal vector pointing outward~$\Omega$. 
We recall that, since~$\gD \in H^{1/2}(\GD)$, there exists a unique continuous weak solution~$u$ to~\eqref{eq:model} that belongs to the convex space~$H_{\GD} := \{\phi \in H^1(\Omega) \ : \ \phi = \gD \text{ on } \GD\}$.

\paragraph{Novelty}
To the best of our knowledge, this is the first work in which the following aspects of the CDG method have been studied.
\begin{itemize}
    \item We describe the CDG method for variable degrees of approximation and fairly general polytopal meshes.
    
    \item We show that the method is well posed without requiring a further stabilization term provided that the weighted averages in the definition of the numerical fluxes satisfy Assumption~\ref{asm:alpha_F} below.
    
    \item We derive~$hp$-\emph{a priori} error estimates in the energy norm, where we employ the recent approximation results from~\cite{Hewett:2025}.

    \item We discuss how the CDG, LDG, and BR2 methods can be implemented within a unified framework, where the stiffness matrix is assembled without relying on large auxiliary matrices.
\end{itemize}

\paragraph{Notations} 
We shall use the standard notation for Sobolev spaces. 
Given an open set~$\Upsilon \subset \R^d$ ($d = 2, 3$) and~$s>0$, we denote by~$H^s(\Upsilon)$ the corresponding Sobolev space with seminorm~$\Seminorm{\cdot}{H^s(\Upsilon)}$ and norm~$\Norm{\cdot}{H^s(\Upsilon)}$. 
We also denote by~$L^2(\Upsilon)$ the space of Lebesgue square integrable functions in~$\Upsilon$ with norm~$\Norm{\cdot}{L^2(\Upsilon)}$. A superscript~$d$ will be used to identify the spaces of~$d$-vector-valued functions.

\paragraph{Outline}
The remainder of the paper is organized as follows: Section~\ref{sec:method} presents the standard notation on polytopal meshes and the CDG method. Moreover, the existence of a discrete solution is proven in Section~\ref{sec:existence}. In Section~\ref{sec:convergence}, we derive an \textit{a-priori} error estimate, 
while Section~\ref{sec:comp_aspects} 
discusses computational aspects and describes fast algorithms for the CDG, LDG, and BR2 methods. In Section~\ref{sec:results}, we carry out some numerical experiments aimed at validating our theoretical results and comparing the CDG method with the LDG and BR2 methods in terms of computational efficiency and stiffness matrix properties. Finally, in Section~\ref{sec:conclusion}, we draw some conclusions and discuss future developments.
\section{Description of the method}
\label{sec:method}
In Section~\ref{sec:meshes-and-DG}, we introduce some notation for polytopal meshes and some standard DG operators. The CDG method is described in Section~\ref{sec:CDG}, and the reduced formulation used in the analysis is derived in Section~\ref{sec:reduced-formulation}. Finally, the existence of a discrete solution is proven in Section~\ref{sec:existence}.
\subsection{Polytopal meshes and DG notations}
\label{sec:meshes-and-DG}
Let~$\{\Omegah\}_{h > 0}$ be a family of polytopal partitions of~$\Omega$, where the subscript~$h$ stands for the meshsize of~$\Omegah$. More precisely, we set~$h := \max_{K \in \Omegah} \hK$, where~$\hK$ denotes the diameter of the element~$K \in \Omegah$. 
We call ``mesh facets" the~$(d-1)$-dimensional facets with positive measure lying on a~$(d-1)$-dimensional hyperplane that compose either the nonempty intersection~$K_1 \cap K_2$ for two elements~$K_1, K_2 \in \Omegah$, or the nonempty intersection~$K \cap \partial \Omega$ for~$K \in \Omegah$.
The set of mesh facets of~$\Omegah$ is denoted by~$\Fh = \FhI \cup \FhD \cup \FhN$, where~$\FhI$ is the set of interior facets of~$\Omegah$, and~$\FhD$ and~$\FhN$ are the sets of facets lying on~$\GD$ and~$\GN$, respectively. For each~$F \in \FhI$, let~$\nF$ denote one of the two~$d$-dimensional unit normal vectors to~$F$. 
Moreover, for each~$F \in \FhD \cup \FhN$, we set~$\nF := \nOmega$. Finally, we denote by~$\FK$ and~$\NK$ the sets of facets and neighbors of~$K$, respectively.

Let~$\p = (\pK)_{K \in \Omegah}$ be a vector that assigns a degree of approximation~$\pK \geq 1$ to each element~$K \in \Omegah$. 
We define the following piecewise polynomial spaces:
\begin{equation*}
\Vhp := \prod_{K \in \Omegah} \Pp{\pK}{K} \quad \text{ and } \quad \Mhp := \prod_{K \in \Omegah} \Pp{\pK}{K}^d,
\end{equation*}
where~$\Pp{\pK}{K}$ denotes the space of polynomials defined on~$K$ of degree at most~$\pK$. We further define the following broken Sobolev spaces: given~$s > 0$,
\begin{equation*}
H^s(\Omegah) := \big\{\phi \in L^2(\Omega) \ : \ \phi_{|_K} \in H^s(K) \text{ for all~$K \in \Omegah$} \big\}.
\end{equation*}

We employ standard notation for weighted averages~($\mvl{\cdot}_{\alpha_F}$) and normal jumps ($\jump{\cdot}$) of piecewise smooth scalar-valued ($v$) and~$d$-vector-valued~($\br$) functions: given~a facet~$F \in \FhI$ shared by two elements~$K_1$ and~$K_2$ in~$\Omegah$ with~$\nF = \nKo$, and a prescribed weighted-average parameter~$\alpha_F \in [0, 1]$, we define 
\begin{alignat*}{3}
\mvl{v}_{\alpha_F} := (1 - \alpha_F) v_{|_{K_1}} + \alpha_F v_{|_{K_2}}, & & \qquad \jump{v} & := (v_{|_{K_1}} - v_{|_{K_2}})\nF, \\
\mvl{\br}_{1-\alpha_F} := 
\alpha_F \br_{|_{K_1}} + 
(1 - \alpha_F) \br_{|_{K_2}}, & & \qquad 
\jump{\br} & := (\br_{|_{K_1}} - \br_{|_{K_2}}) \cdot \nF.
\end{alignat*}

We define the local lifting operator~$\LF : L^2(F)^d \to \Mhp$ as follows:
\begin{alignat}{3}
\label{eq:local-lifting-internal}
\int_{\Omega} \LF(\vphi) \cdot \rh \dx & = \int_F \vphi \cdot \mvl{\rh}_{1 - \alpha_F} \dS & & \qquad \forall \rh \in \Mhp.
\end{alignat}
Similarly, for each boundary facet~$F \in \FhD$, we define the local lifting operator~$\LD : L^2(F) \to \Mhp$ as
\begin{equation}
\label{eq:local-lifting-bndry}
\int_{\Omega} \LD(\phi) \cdot \rh \dx = \int_F \phi \rh \cdot \nOmega \dS  \qquad \forall \rh \in \Mhp.
\end{equation}
Moreover, we define the global lifting operators~$\LhD : H^{1/2 + \varepsilon}(\Omegah) \to \Mhp$ and~$\Lh : H^{1/2 + \varepsilon}(\Omegah) \to \Mhp$ with~$\varepsilon > 0$ by
\begin{equation}
\label{eq:global-liftings}
     \LhD \vh := \sum_{F \in \FhD} \LD (\vh{}_{|_{F}}) \quad \text{ and } \quad \Lh \vh := \sum_{F \in \FhI} \LF (\jump{\vh}) + \LhD \vh .
\end{equation}

We further denote by~$\Pih : L^2(\Omega)^d \to \Mhp$ the~$L^2(\Omega)^d$-orthogonal projection operator onto~$\Mhp$, and by~$\nablah : H^1(\Omegah) \to L^2(\Omega)^d$ the piecewise gradient operator defined on~$\Omegah$. Finally, we shall use the following short-hand notation for sums of integrals over the facets of~$\Omegah$: for~$\star \in \{\mathcal{I}, \mathcal{D}, \mathcal{N}\}$,
\begin{equation*}
    \int_{\Fhs} \phi \dS := \sum_{F \in \Fhs} \int_F \phi \dS.
\end{equation*}

\subsection{Compact discontinuous Galerkin method}
\label{sec:CDG}
In the spirit of~\cite{Perugia_Schotzau:2002} (see also~\cite[\S2.2]{Periaire_Persson:2008}), we introduce the auxiliary variables~$\q = -\nabla u$ and~$\bsigma = \bk \q$, and rewrite model~\eqref{eq:model} as follows:
\begin{alignat*}{3}
\q & = - \nabla u & & \quad \text{ in } \Omega, \\
\bsigma & = \bk \q & & \quad \text{ in } \Omega, \\
\nabla \cdot \bsigma & = f & & \quad \text{ in } \Omega, \\
u & = \gD & & \quad \text{ on } \GD, \\
\bsigma \cdot \nOmega & = -\gN & & \quad \text{ on } \GN.
\end{alignat*}

As in the unified framework of DG methods in~\cite{Arnold_Brezzi_Cockburn_Marini:2001}, we consider the following discrete variational formulation: find~$(\uh,\, \qh,\, \sigmah) \in \Vhp \times \Mhp \times \Mhp$ such that, for each element~$K \in \Omegah$, there hold
\begin{subequations}
\label{eq:flux-formulation}
\begin{alignat}{3}
\label{eq:qh-equation}
\int_K \qh \cdot \rh \dx & = - \int_{\partial K} \uflux \rh \cdot \nK \dS + \int_{K} \uh \nabla \cdot \rh \dx  & & \qquad \forall \rh \in \Mhp,\\
\label{eq:sigmah-equation}
\int_K \sigmah \cdot \sh \dx & = \int_K \bk \qh \cdot \sh \dx & & \qquad \forall \sh \in \Mhp,\\
\label{eq:uh-equation}
\int_{\partial K} \vh \sflux \cdot \nK \dS - \int_{K} \sigmah \cdot \nabla \vh \dx & = \int_{K} f \vh \dx & & \qquad \forall \vh \in \Vhp,
\end{alignat}
\end{subequations}
where the \emph{numerical fluxes} $\uflux$ and~$\sflux$ are approximations of the traces of~$u$ and~$\bsigma$ on~$\Fh$, and they characterize the DG method.

For the CDG method (see~\cite[\S3]{Periaire_Persson:2008}), we define the numerical fluxes %
on each facet~$F \in \Fh$ as
\begin{alignat}{3}
\label{eq:CDG-fluxes}
\uflux & := 
\begin{cases}
\mvl{\uh}_{\alpha_F} & \text{ if~$F \in \FhI$}, \\
\gD & \text{ if~$F \in \FhD$}, \\
\uh & \text{ if~$F \in \FhN$},
\end{cases}
\qquad \sflux & := 
\begin{cases}
\mvl{\sigmah^F}_{1-\alpha_F}
& \text{ if~$F \in \FhI$},\\
\sigmah^F 
& \text{ if~$F \in \FhD$}, \\
-\gN \nOmega & \text{ if~$F \in \FhN$},
\end{cases}
\end{alignat}
where~$\sigmah^F \in \Mhp$ is an auxiliary function defined below and, %
for each~$F \in \FhI$, the weighted-average parameter~$\alpha_F$ is either~$0$ or~$1$. 

For variable degrees of approximation, the choice~$\alpha_F \in \{0, 1\}$, which corresponds to choosing a one-sided trace, leads to the following assumption.

\begin{assumption}[Weighted-average parameters]
\label{asm:alpha_F}
For any interior facet~$F \in \FhI$ shared by two elements~$K_1$ and~$K_2$ in~$\Omegah$ with~$p_{K_1} > p_{K_2}$, the weight parameter~$\alpha_F$ is chosen such that
\begin{equation*}
    \mvl{\sigmah^F}_{1 - \alpha_F} = \sigmah^F{}_{|_{K_1}}.
\end{equation*}
More precisely, for neighboring elements with different degrees of approximation, we set~$\alpha_F$ so that the trace of~$\sigmah^F$ is taken from the element with higher degree.
\end{assumption}

The previous assumption is made to guarantee the coercivity of the bilinear form of the CDG method (see Lemma~\ref{lemma:coercivity-Ah} and Section~\ref{sec:choice-alpha-p-variable} below).
Consequently, for each~$K \in \Omegah$, we define
\begin{alignat}{3}
\label{def:FKout}
\FKout & := \big\{F \in \FK  \ : \ F \in  \FhI \text{ and } \mvl{\cdot}_{1 - \alpha_F} = (\cdot)_{|_{K}}\big\} \cup \big\{F \in \FK \ : \ F \in \FhD \big\}, \\
\label{def:NKout}
\NKout & := \big\{\widetilde{K} \in \Omegah \ : \ %
\FKout \cap \mathcal{F}_{\widetilde{K}} \neq \emptyset \big\},
\end{alignat}
In other words, $\FKout$ is the set of facets of~$K$ such that the support of~$\LF(\cdot)$ or~$\LD(\cdot)$ is~$K$. In addition, for each~$F \in \FhI \cup \FhD$, we denote by~$K_F \in \Omegah$ the only element such that~$F \in \FKFout$, and set~
\begin{equation}
\label{def:NF}
\NF := \card(\FKFout).
\end{equation}
These definitions are illustrated in Figure~\ref{fig:FKout}.
\begin{figure}[t!]
	\centering
	{\includegraphics[width=\textwidth]{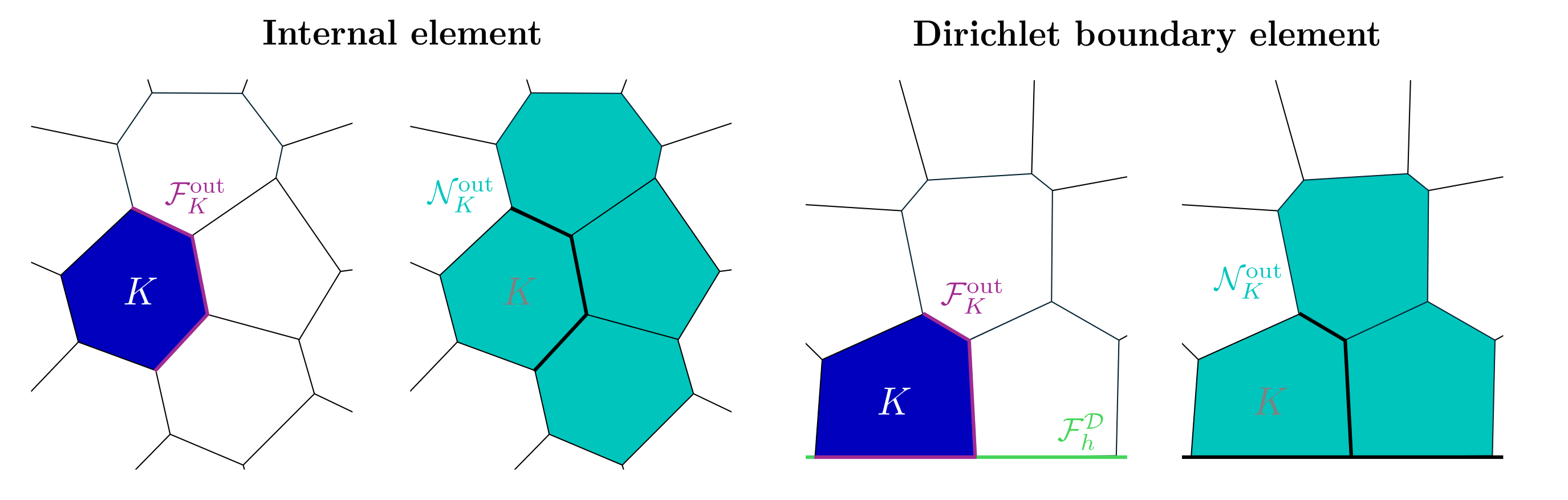}}
	\caption{Schematical representation of the sets $\mathcal{F}_K^\mathrm{out}$ (\textcolor{edgepurple}{\textbf{purple}} on the left) and $\mathcal{N}_K^\mathrm{out}$ (\textcolor{deepaqua}{\textbf{light blue}} on the right) assuming $\alpha_F=1$ for $\beta_F = \boldsymbol{n}_F\cdot[1,0]^{\top} \geq 1$. Internal element case (left panel) and Dirichlet boundary element case (right panel).} 
	\label{fig:FKout}
\end{figure}

As for the auxiliary local functions~$\sigmah^F \in \Mhp$, we first assign a parameter~$\xiF > 0$ to each facet~$F \in \FhI \cup \FhD$. Then, for each interior facet~$F \in \FhI$, we define~$\sigmah^F$ as the solution to the following local problem:
\begin{alignat}{3}
\nonumber
\int_{\Omega} \sigmah^F \cdot \rh \dx & = -\int_{\Omega} \bk \nablah \uh \cdot \rh \dx 
+ \xiF \int_F \jump{\uh} \cdot \mvl{\Pi_h (\bk \rh)}_{1-\alpha_F} \dS \\
\label{eq:sigmaF-interior}
& = -\int_{\Omega} \bk \big(\nablah \uh - \xiF \LF(\jump{\uh}) \big) \cdot \rh \dx & & \qquad \forall \rh \in \Mhp.
\end{alignat}
Analogously, for each boundary facet~$F \in \FhD$, $\sigmah^F \in \Mhp$ is defined as the solution to 
\begin{alignat}{3}
\nonumber
\int_{\Omega} \sigmah^F \cdot \rh \dx & = -\int_{\Omega}  \bk \nablah \uh \cdot \rh \dx + \xiF \int_F (\uh - \gD) \Pih(\bk \rh) \cdot \nOmega \dS \\
\label{eq:sigmaF-bndary}
& = -\int_{\Omega} \bk (\nablah \uh - \xiF\LD(\uh{}_{|_F} - \gD)) \cdot \rh \dx \qquad\qquad \qquad \forall \rh \in \Mhp.
\end{alignat}
Therefore, from~\eqref{eq:sigmaF-interior} and~\eqref{eq:sigmaF-bndary}, we can deduce the following %
expressions for~$\sigmah^F$:
\begin{equation}
\label{eq:def-sigmahF}
\sigmah^F = \begin{cases}
-\Pih(\bk (\nablah \uh - \xiF \LF(\jump{\uh}))) & \text{ if~$F \in \FhI$},\\
-\Pih(\bk (\nablah \uh - \xiF \LD(\uh{}_{|_F} - \gD))) & \text{ if~$F \in \FhD$}.
\end{cases}
\end{equation}

We make the following additional assumption on 
the parameters~$\xiF$, which we use in Lemma~\ref{lemma:CDG-norm} to show that the seminorm induced by the bilinear form of the CDG method is a norm in~$\Vhp$. 

\begin{assumption}[Choice of~$\xiF$]
\label{asm:xiF}
We assume that there is a positive constant~$\gamma$ independent of~$h$ such that
$$0 < \frac{\NF}{\xiF} \le \gamma < 1 \qquad  \forall F \in \FhI \cup \FhD,$$ 
where~$\NF$ are the constants defined in~\eqref{def:NF}, which we assume to be uniformly bounded.
\end{assumption}

\begin{remark}[Relation with the LDG method]
The LDG method is obtained by modifying the definition of the numerical flux~$\sflux$ as follows:
\begin{equation*}
\sflux := \begin{cases}
\mvl{\sigmah}_{1-\alpha_F} + \eta_F \jump{\uh}    & \text{if~$F \in \FhI$},\\
\sigmah + \eta_F (\uh - \gD) & \text{if~$F \in \FhD$}, \\
-\gN \nOmega & \text{if~$F \in \FhN$},
\end{cases}
\end{equation*}
where, for each~$F \in \FhI$, $\alpha_F \in [0, 1]$, and the piecewise-constant stabilization function~$\eta_F \in L^{\infty}(\FhI \cup \FhD)$ satisfies~$\eta_{\star} := \essinf_{\FhI \cup \FhD} \eta_F > 0$.

Analogously to as in~\eqref{eq:def-sigmahF}, for the LDG method, the function~$\sigmah$ can be explicitly expressed as (see also equation~\eqref{eq:qh-sigmah-identity} below) 
$$\sigmah = -\Pih\big(\bk (\nablah \uh - \Lh \uh + \LhD (\gD))\big).$$ 
The implicit dependence in the LDG method of the numerical flux~$\sflux$ on  the global lifting~$\Lh \uh$ leads to a larger stencil. 

The analysis of the LDG method on polytopal meshes was carried out in~\cite{Ye_Zhang_Zhu:2022} (for the~$h$ version) and in~\cite{Gomez-Perinati-Stocker:2026} (for the~$hp$ version in a space--time setting). 
\eremk
\end{remark}

\begin{remark}[Relation with the BR2 method]
The BR2 method is obtained by setting~$\alpha_F = 1/2$ in~\eqref{eq:CDG-fluxes} for all the internal facets~$F \in \FhI$ (see~\cite[\S3]{Brezzi_etal:2000} and~\cite[\S3.2]{Bassi_etal:2012}). For this choice of the weighted-average parameters, the support of~$\sigmah^F$ is the union of both elements sharing the facet~$F$, which leads to an increase in the computational cost compared to the CDG method. Namely, for each internal face $F$, we need to compute two interface integrals instead of one.

The analysis in~\cite[\S IV]{Brezzi_etal:2000} of the BR2 method was extended to agglomerated meshes in two-dimensional domains in~\cite[\S3.3]{Bassi_etal:2012}. Assumption~\ref{asm:xiF} on the
choice of the parameters~$\xiF$ is inspired by the one made in~\cite[Thm.~1]{Bassi_etal:2012}, which allows for a milder dependence of the constants in the a priori error estimates on the maximum number of facets of the elements in~$\Omegah$. 
\eremk
\end{remark}

\begin{remark}[The CDG2 version]
An alternative version (CDG2) of the CDG method was introduced in~\cite{Brdar_Dedner_Klofkorn:2012} in the context of nonlinear convection--diffusion--reaction equations. For the model problem~\eqref{eq:model}, it corresponds to the following choice of the numerical fluxes in the abstract variational formulation~\eqref{eq:flux-formulation}:
\begin{alignat*}{3}
\uflux & := 
\begin{cases}
\mvl{\uh}_{1/2} & \text{ if~$F \in \FhI$},\\ %
\gD & \text{ if~$F \in \FhD$}, \\
\uh & \text{ if~$F \in \FhN$},
\end{cases}\\
\medskip
\sflux & := 
\begin{cases}
-\mvl{\Pi_h(\bk \nablah \uh)}_{1/2} + \xiF \mvl{\Pi_h(\bk \LF(\jump{\uh})}_{1- \alpha_F} %
& \text{ if~$F \in \FhI$,}\\ %
-\Pih(\bk (\nablah \uh - \xiF \LD(\uh{}_{|_F} - \gD)))_{|_{F}} %
& \text{ if~$F \in \FhD$}, \\
-\gN \nOmega & \text{ if~$F \in \FhN$},
\end{cases}
\end{alignat*}
where, for each~$F \in \FhI$, $\alpha_F \in \{0, 1\}$. 
Thus, this version considers standard (arithmetic) averages except for the term in the definition in~\eqref{eq:def-sigmahF} of~$\sigmah^F$ involving the local lifting~$\LF(\jump{\uh})$.
\eremk
\end{remark}

\begin{remark}[Absence of a stabilization term]
\label{rem:stabilization}
Motivated by~\cite[Thm.~2]{Brdar_Dedner_Klofkorn:2012} for standard meshes, which states that it is not necessary to include a stability term in the CDG method if the parameters~$\xiF$ are chosen appropriately, we have not introduced any stability term in the definition of~$\sflux$ in~\eqref{eq:CDG-fluxes} (cf.~\cite[\S3]{Periaire_Persson:2008} and~\cite[\S2.1]{Brdar_Dedner_Klofkorn:2012}). 
This allows us to avoid computing an additional discrete operator in the method. Moreover, the inclusion of a stability term would not solve the issue of the %
dependence of the error estimates on the maximum %
value of~$\NF$, unless the stability parameter were taken ``large enough", as in the IPDG method. 
\eremk
\end{remark}

\subsection{Reduced formulation}
\label{sec:reduced-formulation}
For the analysis, we proceed as in~\cite{Periaire_Persson:2008} and~\cite{Brdar_Dedner_Klofkorn:2012}, by first rewriting the method as a variational problem involving only the scalar unknown~$\uh$. 

Summing~\eqref{eq:qh-equation} over all the elements of the mesh~$\Omegah$, integrating by parts, and using the average--jump identity
\begin{equation*}
    \mvl{\vh}_{\alpha_F} \jump{\rh} + \jump{\vh} \cdot \mvl{\rh}_{1 - \alpha_F} = \jump{\vh \rh} \qquad \forall (\vh, \rh) \in \Vhp \times \Mhp,
\end{equation*}
and the definition of the global lifting operators, we get
\begin{alignat*}{3}
\nonumber
\int_{\Omega} \qh \cdot \rh \dx & = \sum_{K \in \Omegah} \int_K \uh \nabla \cdot \rh \dx 
-\int_{\FhI} \mvl{\uh}_{\alpha_F} \jump{\rh} \dS - \int_{\FhD} \gD \rh \cdot \nOmega \dS - \int_{\FhN} \uh \rh \cdot \nOmega \dS \\
\nonumber
& = - \int_{\Omega} \nablah \uh \cdot \rh \dx + \int_{\FhI} \big( \jump{\uh \rh} - \mvl{\uh}_{\alpha_F} \jump{\rh}\big) \dS + \int_{\FhD}  \big(\uh - \gD \big)\rh  \cdot \nOmega \dS \\
\nonumber
& = - \int_{\Omega} \nablah \uh \cdot \rh \dx + \int_{\FhI} \jump{\uh} \cdot \mvl{\rh}_{1 - \alpha_F} \dS + \int_{\FhD} (\uh - \gD) \rh \cdot \nOmega \dS \\
\nonumber
& = - \int_{\Omega} \nablah \uh \cdot \rh \dx + \sum_{F \in \FhI} \int_{\Omega} \LF(\jump{\uh}) \cdot \rh \dx + \sum_{F \in \FhD} \int_{\Omega} \LD(\uh{}_{|_F}  - \gD) \cdot \rh \dx \\
& = - \int_{\Omega} \big(\nablah \uh  - \Lh \uh + \LhD(\gD) \big) \cdot \rh \dx,
\end{alignat*}
which leads to the following expressions for~$\qh$ and~$\sigmah$:
\begin{equation}
\label{eq:qh-sigmah-identity}
\qh = - (\nablah \uh - \Lh \uh + \LhD(\gD)) \quad \text{ and } \quad \sigmah = \Pih(\bk \qh).
\end{equation}

Summing the left-hand side of~\eqref{eq:uh-equation} over all the elements in~$\Omegah$, substituting the definition of the numerical flux~$\sflux$ in~\eqref{eq:CDG-fluxes}, and using the identities in~\eqref{eq:qh-sigmah-identity}, the expressions in~\eqref{eq:def-sigmahF} for~$\sigmah^F$, and the definition of the local and global lifting operators in~\eqref{eq:local-lifting-internal}, \eqref{eq:local-lifting-bndry}, and~\eqref{eq:global-liftings}, we obtain
\begin{align*}
\sum_{K \in \Omegah} \Big( - & \int_K \sigmah \cdot \nabla \vh \dx +  \int_{\partial K} \vh \sflux \cdot \nK \dS \Big) \\
& = - \int_{\Omega} \sigmah \cdot \nablah \vh \dx + \int_{\FhI} \mvl{\sigmah^F}_{1 - \alpha_F} \cdot \jump{\vh}\dS + \int_{\FhD} \vh \sigmah^F \cdot \nOmega  \dS  - \int_{\FhN} \gN \vh \dS \\
& =  \int_{\Omega} \bk \big(\nablah \uh - \Lh \uh \big) \cdot\nablah \vh \dx + \sum_{F \in \FhI} \int_{\Omega} \sigmah^F \cdot \LF(\jump{\vh}) \dx \\
& \quad + \sum_{F \in \FhD} \int_{\Omega} \sigmah^F \cdot \LD(\vh{}_{|_F}) \dx - \int_{\FhN} \gN \vh \dS + \int_{\FhD} \gD( \bk \nablah \vh \cdot \nOmega) \dS \\
& = \int_{\Omega} \bk \big(\nablah \uh - \Lh \uh \big)\cdot \nablah \vh \dx - \sum_{F \in \FhI} \int_{\Omega} \bk \big(\nablah \uh - \xiF \LF(\jump{\uh}) \big) \cdot \LF (\jump{\vh}) \dx \\
& \quad - \sum_{F \in \FhD} \int_{\Omega} \bk( \nablah \uh - \xiF \LD(\uh{}_{|_F} - \gD)) \cdot \LD(\vh{}_{|_F}) \dx \\
& \quad - \int_{\FhN} \gN \vh \dS + \int_{\FhD} \gD(\bk \nablah \vh \cdot \nOmega) \dS \\
& = \int_{\Omega} \bk(\nablah \uh - \Lh \uh) \cdot \nablah \vh \dx - \int_{\Omega} \bk \nablah \uh \cdot \Lh \vh \dx \\
& \quad + \sum_{F \in \FhI} \xiF \int_{\Omega}  \bk \LF(\jump{\uh}) \cdot \LF(\jump{\vh}) \dx 
+ \sum_{F \in \FhD} \xiF \int_{\Omega} \bk \LD(\uh{}_{|_F}) \cdot \LD (\vh {}_{|_F}) \dx  \\
& \quad - \int_{\FhN} \gN \vh \dS + \int_{\FhD} \gD(\bk (\nablah \vh - \xiF \LhD \vh) \cdot \nOmega) \dS.
\end{align*}

Therefore, we get the following reduced formulation in terms of the primal variable only: find~$\uh \in \Vhp$ such that
\begin{equation}
\label{eq:CDG-variational}
\Ah(\uh, \vh) = \lh(\vh) \qquad \forall \vh \in \Vhp,
\end{equation}
where the bilinear form~$\Ah : \Vhp \times \Vhp \to \R$ and the linear functional~$\lh : \Vhp \to \R$ are given by
\begin{subequations}
\begin{alignat}{3}
\nonumber
\Ah(\uh, \vh) & := \int_{\Omega} \bk(\nablah \uh - \Lh \uh) \cdot \nablah \vh \dx - \int_{\Omega} \bk \nablah \uh \cdot \Lh \vh \dx  \\
\label{eq:def-Ah}
& + \sum_{F \in \FhI} \xiF \int_{\Omega} \bk \LF(\jump{\uh}) \cdot \LF(\jump{\vh}) \dx 
+ \sum_{F \in \FhD} \xiF \int_{\Omega} \bk\LD(\uh{}_{|_F}) \cdot \LD (\vh {}_{|_F}) \dx, \\
\label{eq:def-lh}
\lh(\vh) & := \int_{\Omega} f \vh \dx + \int_{\FhN} \gN \vh \dS - \int_{\FhD} \gD(\bk (\nablah \vh - \xiF \LhD \vh) \cdot \nOmega) \dS.
\end{alignat}
\end{subequations}

\subsection{Well-posedness}
\label{sec:existence}
We define the following mesh-dependent seminorm in $\Vhp + H^1(\Omegah)$:
\begin{equation*}
\Tnorm{v}{\CDG}^2 := \Norm{\sqrt{\bk} \nablah v}{L^2(\Omega)^d}^2 + \sum_{F \in \FhI} \xiF \Norm{\sqrt{\bk} \LF(\jump{v})}{L^2(\Omega)^d}^2 + \sum_{F \in \FhD} \xiF \Norm{\sqrt{\bk} \LD(v_{|_F})}{L^2(\Omega)^d}^2,
\end{equation*}
where~$\sqrt{\bk}$ is the symmetric positive definite matrix such that~$\sqrt{\bk} \cdot \sqrt{\bk} = \bk$.

The existence and uniqueness of a discrete solution to~\eqref{eq:CDG-variational} is a consequence of the following lemmas, where we show that~$\Tnorm{\cdot}{\CDG}$ is a norm in~$\Vhp$ provided that Assumption ~\ref{asm:alpha_F} holds, and that the bilinear form~$\Ah(\cdot, \cdot)$ is coercive with respect to~$\Tnorm{\cdot}{\CDG}$.
\begin{lemma}
\label{lemma:CDG-norm}
Under Assumption~\ref{asm:alpha_F}, the seminorm~$\Tnorm{\cdot}{\CDG}$ is a norm in the discrete space~$\Vhp$.
\end{lemma}
\begin{proof}
Let~$\vh \in \Vhp$ with~$\Tnorm{\vh}{\CDG} = 0$. Due to the nondegeneracy in~\eqref{eq:nondegeneracy} of~$\bk$, we deduce that~$\LF(\jump{\uh}) = 0$ for all~$F \in \FhI$, and~$\LD( \vh{}_{|_{F}}) = 0$ for all~$F \in \FhD$. 
Moreover, from the definition in~\eqref{eq:local-lifting-internal} of~$\LF(\cdot)$ and Assumption~\ref{asm:alpha_F}, for each~$F \in \FhI$ shared by two elements~$K_1,\, K_2 \in \Omegah$ with~$p_{K_1} > p_{K_2}$, we get
\begin{equation}
\label{eq:vanishing-jumps-interior}
\int_F \jump{\vh} \cdot \mvl{\rh}_{1 - \alpha_F} \dS = \int_F \jump{\vh} \cdot \rh{}_{|_{K_1}} \dS = 0 \quad \forall \rh \in \Mhp. 
\end{equation}
In order to show that~$\jump{\vh} = 0$ on~$F$, it is enough to choose~$\rh$ in~\eqref{eq:vanishing-jumps-interior} as a polynomial whose restriction to~$F$ is equal to~$\jump{\vh}$, is constant along the orthogonal direction to~$F$, and has support on~$K_1$. Since~$\jump{\vh}$ is a $d$-vector-valued polynomial of degree~$\max\{p_{K_1}, p_{K_2}\} = p_{K_1}$ defined on~$F$, such a choice is only possible because~$\alpha_F$ was chosen in Assumption~\ref{asm:alpha_F} such that~$\mvl{\rh}_{1-\alpha_F}$ reduces to the trace of~$\rh$ from the element of higher degree. The choice of~$\alpha_F$ is not relevant when~$p_{K_1} = p_{K_2}$; in that case, we can always show that~$\jump{\vh} = 0$. 
Using the definition in~\eqref{eq:local-lifting-bndry} of~$\LD(\cdot)$, one can prove that~$\vh = 0$ on~$\GD$. 
In addition, $\nablah \vh = 0$ in~$\Omega$, so we deduce that~$\vh$ is a piecewise constant function on the mesh $\Omegah$. Considering that~$\jump{\vh} = 0$, it follows that~$\vh$ is a constant function on~$\Omega$, and moreover, it has zero trace on~$\GD$. Since~$|\GD| > 0$, we conclude that~$\vh = 0$. 
\end{proof}

\begin{lemma}[Coercivity of~$\Ah$]
\label{lemma:coercivity-Ah}
Under Assumptions~\ref{asm:alpha_F} and~\ref{asm:xiF}, 
there exists a positive constant~$\CA$ independent of the meshsize~$h$, the degree vector~$\p$, and the maximum number of facets of~$\Omegah$ such that
\begin{equation}
\label{eq:coer-A}
\Ah(\vh, \vh) \geq \CA \Tnorm{\vh}{\CDG}^2 \qquad \forall \vh \in \Vhp.
\end{equation}
\end{lemma}
\begin{proof}
Let~$\vh \in \Vhp$. From the definition in~\eqref{eq:def-Ah} of the bilinear form~$\Ah(\cdot, \cdot)$, and using the symmetry of $\Ah(\cdot, \cdot)$ and the Young inequality, we obtain
\begin{alignat}{3}
\nonumber
\Ah(\vh, \vh) & = \Norm{\sqrt{\bk} \nablah \vh}{L^2(\Omega)^d}^2 - 2 \int_{\Omega} \bk \Lh \vh \cdot \nablah \vh \dx \\
\nonumber
& \quad + \sum_{F \in \FhI} \xiF \Norm{\sqrt{\bk} \LF (\jump{\vh})}{L^2(\Omega)^d}^2  + \sum_{F \in \FhD} \xiF \Norm{\sqrt{\bk} \LD(\vh{}_{|_F})}{L^2(\Omega)^d}^2 \\
\nonumber
& \geq (1 - \varepsilon) \Norm{\sqrt{\bk} \nablah \vh}{L^2(\Omega)^d}^2 - \frac{1}{\varepsilon} \Norm{\sqrt{\bk} \Lh \vh}{L^2(\Omega)^d}^2  \\
\label{eq:lower-bound-Ah}
& \quad + \sum_{F \in \FhI} \xiF \Norm{\sqrt{\bk} \LF (\jump{\vh})}{L^2(\Omega)^d}^2  + \sum_{F \in \FhD} \xiF \Norm{\sqrt{\bk} \LD(\vh{}_{|_F})}{L^2(\Omega)^d}^2.
\end{alignat}

It only remains to bound the second term on the right-hand side of~\eqref{eq:lower-bound-Ah}. Using the definition of the global lifting~$\Lh \vh$, the definition in~\eqref{def:FKout} of the set~$\FKout$, and the standard inequality~$|\boldsymbol{y}|_1 \le \sqrt{n} |\boldsymbol{y}|_2$ for vectors in~$\R^n$, we get
\begin{alignat}{3}
\nonumber
\Norm{\sqrt{\bk} \Lh \vh}{L^2(\Omega)^d}^2 & = \sum_{K \in \Omegah} \int_{K} \bigg(\sum_{F \in \FKout \cap \FhI} \sqrt{\bk} \LF(\jump{\vh}) + \sum_{F \in \FKout \cap \FhD} \sqrt{\bk} \LD(\vh{}_{|_F}) \bigg)^2 \dx \\
\nonumber
& \le \sum_{K \in \Omegah} \NF \Big(\sum_{F \in \FKout \cap \FhI} \Norm{\sqrt{\bk} \LF(\jump{\vh})}{L^2(K)^d}^2 + \sum_{F \in \FKout \cap \FhD} \Norm{\sqrt{\bk} \LD(\vh{}_{|_F})}{L^2(K)^d}^2 \Big) \\
\label{eq:bound-Lh}
& = \sum_{F \in \FhI} \NF \Norm{\sqrt{\bk} \LF(\jump{\vh})}{L^2(\Omega)^d}^2 + \sum_{F \in \FhD} \NF \Norm{\sqrt{\bk} \LD(\vh{}_{|_{F}})}{L^2(\Omega)^d}^2.
\end{alignat}

Combining~\eqref{eq:bound-Lh} with~\eqref{eq:lower-bound-Ah}, and using Assumption~\ref{asm:xiF}, 
it follows that
\begin{alignat*}{3}
\Ah(\vh, \vh) & \geq (1 - \varepsilon) \Norm{\sqrt{\bk} \nablah \vh}{L^2(\Omega)^d}^2 + \sum_{F \in \FhI} \Big(\xiF - \frac{\NF}{\varepsilon} \Big) \Norm{\sqrt{\bk} \LF(\jump{\vh})}{L^2(\Omega)^d}^2 \\
& \quad + \sum_{F \in \FhD} \Big(\xiF - \frac{\NF}{\varepsilon} \Big) \Norm{\sqrt{\bk} \LD(\vh{}_{|_{F}})}{L^2(\Omega)^d}^2 \\
& \geq (1 - \varepsilon) \Norm{\sqrt{\bk} \nablah \vh}{L^2(\Omega)^d}^2 + \sum_{F \in \FhI} \xiF \Big(1 - \frac{\gamma}{\varepsilon} \Big) \Norm{\sqrt{\bk} \LF(\jump{\vh})}{L^2(\Omega)^d}^2 \\
& \quad + \sum_{F \in \FhD} \xiF\Big(1 - \frac{\gamma}{\varepsilon} \Big) \Norm{\sqrt{\bk} \LD(\vh{}_{|_{F}})}{L^2(\Omega)^d}^2.
\end{alignat*}
Therefore, since~$0 < \gamma < 1$, we can choose~$\gamma < \varepsilon < 1$. This completes the proof of~\eqref{eq:coer-A}.
\end{proof}

In the error analysis of next section, we use the following continuity property of the bilinear form~$\Ah(\cdot, \cdot)$, which can be easily obtained using the Cauchy--Schwarz inequality and bound~\eqref{eq:bound-Lh} on~$\Lh \vh$.
\begin{lemma}[Continuity of~$\Ah(\cdot, \cdot)$]
\label{lemma:continuity-Ah}
Under Assumption~\ref{asm:alpha_F}, %
it holds
\begin{equation*}
\Ah(u, v) \le 2 \Tnorm{u}{\CDG} \Tnorm{v}{\CDG} \qquad \forall u, v \in H^1(\Omegah).
\end{equation*}
\end{lemma}

\section{Convergence analysis}
\label{sec:convergence}
In this section, we derive~$hp$-\emph{a priori} error estimates for the CDG method. Henceforth, we denote by~$\Id$ the identity operator for~$d$-vector-valued functions, and we use~$a \lesssim b$ to indicate the existence of a positive constant~$C$ independent of~$h$, $\p$, and the maximum number of facets of the elements in~$\Omegah$ such that~$a \le C b$. Similarly, we use~$a \simeq b$ whenever~$a \lesssim b$ and~$b \lesssim a$.

The second Strang's lemma and the coercivity and continuity in Lemmas~\ref{lemma:coercivity-Ah} and~\ref{lemma:continuity-Ah} of the bilinear form~$\Ah(\cdot, \cdot)$ lead to the following~\emph{a priori} error bound.
\begin{lemma}[\emph{A priori} error bound]
\label{lemma:a-priori-bounds}
Let Assumptions~\ref{asm:alpha_F} and~\ref{asm:xiF} 
hold.
Let also the diffusion coefficient~$\bk$ and the continuous weak solution~$u$ to~\eqref{eq:model} satisfy: $u \in H_{\GD}^1(\Omega)$ and~$-\nabla \cdot (\bk \nabla u) \in L^2(\Omega)$, 
and let~$\uh \in \Vhp$ be the solution to the discrete formulation~\eqref{eq:CDG-variational}. Then, the following error bound holds:
\begin{alignat}{3}
\nonumber
\Tnorm{u - \uh}{\CDG} & \le \Big(1 + \frac{2}{\CA} \Big)\Tnorm{u - \wh}{\CDG} 
+ \frac{1}{\CA} \sup_{\vh \in \Vhp \setminus \{0\}} \frac{|\Ah(u, \vh) - \ell_h(\vh)|}{\Tnorm{\vh}{\CDG}},
\end{alignat}
for all~$\wh \in \Vhp$.
\end{lemma}
So far, we have made no assumptions on the family of polytopal meshes~$\{\Omegah\}_{h > 0}$. In fact, as for the LDG method (see~\cite[Prop.~2.1]{Castillo_Cockburn_Perugia_SChotzau:2000}), the well-posedness of~\eqref{eq:CDG-variational} does not rely on any of such assumptions.
However, in the convergence analysis, we require the following assumption (cf. \cite[Asm.~30 in Ch.~4.3]{Cangiani_Dong_Georgoulis_Houston:2017}).

\begin{assumption}[Mesh assumption]
\label{asm:mesh}
For any~$K \in \Omegah$, there exists a set of nonoverlapping~$d$-dimensional simplices~$\{\sKF\}_{F \in \FK}$ such that, for all~$F \in \FK$, $\sKF \subset K$ and shares the facet~$F$ with~$K$, and the following condition holds:
\begin{equation*}
\hK \le C_s \frac{d|s_K^F|}{|F|},
\end{equation*}
where~$C_s$ is a positive constant independent of the discretization parameters, the measure of~$F$, and the maximum number of facets of~$K$.
\end{assumption}

We further recall the notion of \emph{covering} of a polytopal mesh, as well as the concept of~\emph{covering choice function} introduced in~\cite[Def. 2.2]{Hewett:2025}.
\begin{definition}[Covering of~$\Omegah$]
For each~$\Omegah$, we call a covering~$\Tchar$ of~$\Omegah$ a set of simplices or hypercubes such that, for each~$K \in \Omegah$, there is at least one~$\calK \in \Tchar$ with~$K \subset \calK$. We say that~$\varphi : \Omegah \to \Tchar$ is a covering choice function if~$K \subset \varphi(K)$ for all~$K \in \Omegah$.
\end{definition}

\begin{remark}[Assumption~\ref{asm:mesh} on the mesh]
This assumption was originally introduced in~\cite[\S2.2]{Cangiani_Dong_Georgoulis_Houston:2017} in the context of a space--time IPDG method for the heat equation. In particular, it allows for elements that are a finite union of uniformly star-shaped polytopes. Moreover, it does not impose any direct restriction on the measure of the facets, thus allowing for elements with small facets~$F$, provided that the height of the associated simplex~$\sKF$ is comparable to the diameter~$\hK$ of the element. This freedom mainly results from the fact that the simplicial subpartition does not need to cover the whole element~$K$ and is allowed to be highly irregular. We refer to~\cite[\S2.2]{Cangiani_Dong_Georgoulis:2017} and~\cite[\S4.3]{Cangiani_Dong_Georgoulis_Houston:2017} for further details, and highlight that even more general meshes with possible curved facets have been considered in~\cite{Cangiani_Dong_Georgoulis:2021}.
\eremk
\end{remark}
\subsection{Some useful tools}
\label{sec:tools}
The following trace inequalities for discrete (see~\cite[Thm.~3]{Warburton_Hesthaven:2003}) 
and continuous (see~\cite[Eq.~(1.52) in Ch.~1]{DiPietro_Droniou:2020}) functions are instrumental in the~\emph{a priori} error analysis.
\begin{lemma}[Trace inequalities]
\label{lemma:trace-inequalities}
Let Assumption~\ref{asm:mesh} on~$\Omegah$ hold. Then, 
for all~$F \in \FhI\cup \FhD$, there hold
\begin{subequations}
\label{eq:inverse-estimates}
\begin{alignat}{3}
\label{eq:polynomial-trace-inverse}
\Norm{v_h}{L^2(F)}^2 & \le \frac{(\pKF + 1)(\pKF + d)}{d} \frac{|F|}{|\sKF|} \Norm{v_h}{L^2(\sKF)}^2 && \qquad \forall v_h \in \Pp{\pK}{\KF}, \\
\label{eq:continuous-trace}
\Norm{v}{L^2(F)}^2 & \le \Ctr \Big(\frac{\pKF}{\hKF} \Norm{v}{L^2(\sKF)}^2 + \frac{\hKF}{\pKF} \Seminorm{v}{H^1(\sKF)}^2 \Big) & & \qquad \forall v \in H^1(\KF), 
\end{alignat}
\end{subequations}
where~$\KF$ is as in~\eqref{def:NF}, and~$
\sKF$ is as in Assumption~\ref{asm:mesh}.
\end{lemma}

In order to get a suitable bound on the inconsistency term in Lemma~\ref{lemma:a-priori-bounds}, we first extend the result in~\cite[Lemma 2]{Brezzi_etal:2000} to polytopal meshes, where Assumption~\ref{asm:alpha_F} on the choice of the weighted-average parameters~$\alpha_F$ plays a crucial role.  
\begin{lemma}[Jump--lifting bounds]
\label{lemma:jump-lifting}
Under Assumptions~\ref{asm:alpha_F} and~\ref{asm:mesh}, for all~$v \in H^1(\Omegah)$, $\v \in H^1(\Omegah)^d$, and~$\vh \in \Mhp$, there hold
\begin{subequations}
\label{eq:jump-lifting-bounds}
\begin{alignat}{3}
\label{eq:jump-lifting-bounds-1}
\Norm{\jump{\vh}}{L^2(F)^d} & \le %
\hKF^{1/2} \Norm{\LF(\jump{\vh})}{L^2(\Omega)^d} & & \qquad \forall F \in \FhI,\\
\label{eq:jump-lifting-bounds-2}
\Norm{\vh}{L^2(F)} & \le %
\hKF^{1/2} \Norm{\LD(\vh{}_{|_F})}{L^2(\Omega)^d} & & \qquad \forall F \in \FhD,\\
\label{eq:lifting-jump-bounds-1}
\Norm{\LF(\jump{\v})}{L^2(\Omega)^d} & \le \sqrt{C_s \frac{(\pKF + 1)(\pKF + d)}{\hKF}} \Norm{\jump{\v}}{L^2(F)^d} & & \qquad \forall F \in \FhI, \\
\label{eq:lifting-jump-bounds-2}
\Norm{\LD(v{|_F})}{L^2(\Omega)^d} & \le \sqrt{C_s \frac{(\pKF + 1)(\pKF + d)}{\hKF}} \Norm{v}{L^2(F)} & & \qquad \forall F \in \FhI, 
\end{alignat}
\end{subequations}
with~$K_F\in \Omegah$ %
defined just above~\eqref{def:NF}.
\end{lemma}
\begin{proof}
For any~$F \in \FhI$, let~$\rh = \rh(\jump{\vh}) \in \Mhp$ be chosen as in the proof of Lemma~\ref{lemma:CDG-norm} such that~$\mvl{\rh}_{1-\alpha_F} = \jump{\vh}$ on~$F$. Then, using the Cauchy--Schwarz inequality and the definition in~\eqref{eq:local-lifting-internal} of the local lifting operator, we have
\begin{alignat}{3}
\nonumber
\Norm{\jump{\vh}}{L^2(F)^d}^2 = \int_F \jump{\vh} \cdot \jump{\vh} \dS & = \int_F \jump{\vh} \cdot \mvl{\rh(\jump{\vh})}_{1- \alpha_F} \dS \\
\nonumber
& = \int_{\Omega} \LF(\jump{\vh}) \cdot \rh(\jump{\vh}) \dS \\
\label{eq:bound-jumps}
& \le \Norm{\LF(\jump{\vh})}{L^2(\Omega)^d} \Norm{\rh(\jump{\vh})}{L^2(\Omega)^d}.
\end{alignat}
From the definition of~$\rh(\jump{\vh})$ and Assumption~\ref{asm:mesh}, we get
\begin{equation*}
\Norm{\rh(\jump{\vh})}{L^2(\Omega)^d} \le \hKF^{1/2} \Norm{\jump{\vh}}{L^2(F)^d},
\end{equation*}
which, combined with~\eqref{eq:bound-jumps}, gives~\eqref{eq:jump-lifting-bounds-1}. Bound~\eqref{eq:jump-lifting-bounds-2} follows analogously.

As for~\eqref{eq:lifting-jump-bounds-1}, we use the definition in~\eqref{eq:local-lifting-internal} of the local lifting operator~$\LF(\cdot)$, the Cauchy--Schwarz inequality, the inverse-trace inequality~\eqref{eq:polynomial-trace-inverse}, and Assumption~\ref{asm:mesh} to obtain
\begin{alignat*}{3}
\Norm{\LF(\jump{\v})}{L^2(\Omega)^d}^2  & = \int_{\Omega} \LF(\jump{\v}) \cdot \LF(\jump{\v}) \dx = \int_F \jump{\v} \cdot \LF(\jump{\v}){}_{|_{K_F}} \dS \\
& \le \Norm{\jump{\v}}{L^2(F)^d} \Norm{\LF(\jump{\v}){}_{|_{K_F}}}{L^2(F)^d} \\
& \le \sqrt{\frac{(\pKF + 1)(\pKF + d)}{d} \frac{|F|}{|\sKF|}} \Norm{\jump{\v}}{L^2(F)^d} \Norm{\LF(\jump{\v})}{L^2(K_F)^d} \\
& \le \sqrt{C_s \frac{(\pKF + 1)(\pKF + d)}{\hKF}} \Norm{\jump{\v}}{L^2(F)^d} \Norm{\LF(\jump{\v})}{L^2(\Omega)^d}.
\end{alignat*}
The proof of bound~\eqref{eq:lifting-jump-bounds-2} is analogous.
\end{proof}

The last ingredient is the following approximation result from~\cite[Lemma 2.3]{Hewett:2025}, which not only generalizes the corresponding result in~\cite[Lemma 4.31]{Cangiani_Dong_Georgoulis:2021} to non-Lipschitz domains, but also removes the usual assumption that the number of mesh elements in~$\Omegah$ that can intersect each covering element in~$\Tchar$ is uniformly bounded. 
\begin{lemma}[Estimates for~$\Pihp$]
\label{lemma:approx}
Let Assumption~\ref{asm:mesh} on~$\{\Omegah\}_{h > 0}$ hold, and assume that~$h \lesssim 1$. Let also~$\Tchar$ and~$\varphi : \Omegah \to \Tchar$ be a shape-regular covering of~$\Omegah$ and a covering choice function, respectively. Suppose that~$u \in L^2(\Omega)$ and that there exist~$\lK \geq 0$ and~$\UK \in H^{\lK}(\calK)$ such that~$(\UK)_{|_K} = u_{|_K}$  for each~$K \in \varphi^{-1}(\calK)$. Then, for each~$\calK \in \Tchar$, there exists an operator~$\Pihp^{\calK} : H^{\lK}(\calK) \to \Pp{\pcalK}{\calK}$ such that
\begin{alignat*}{3}
\bigg(\sum_{K \in \varphi^{-1}(\calK)} \Norm{u - \Pihp^{\calK} u}{H^q(K)}^2 \bigg)^{1/2} \lesssim \frac{\hcalK^{\ellK - q}}{\pcalK^{\lK - q}} \Norm{\UK}{H^{\lK}(\calK)}, \qquad \text{ for~$0 \le q \le \lK$},
\end{alignat*}
where~$\hcalK := \diam(\calK)$, $\pcalK := \min_{K \in \varphi^{-1}(\calK)} \pK$, $\ellK := \min\{\pcalK + 1, \lK\}$, and the hidden constant depends only on~$d$, $\lK$, and the shape-regularity constant of~$\Tchar$.
\end{lemma}

A typical choice for the extension operator in the statement of Lemma~\ref{lemma:approx} is the one proposed by Stein in~\cite[Thm. 5 in Ch. VI]{Stein}.
\subsection{\emph{A priori} error estimates}
In next lemma, we get a bound on the inconsistency term in Lemma~\ref{lemma:a-priori-bounds}. Since we do not have a stability term in the method (see~\eqref{eq:CDG-variational}), the jump terms %
must be controlled by the local lifting functions using Lemma~\ref{lemma:jump-lifting}. 
\begin{lemma}[Bound on the inconsistency term]
\label{lemma:inconsistency}
Let the assumptions of Lemma \ref{lemma:a-priori-bounds} hold.
Then, 
\begin{alignat}{3}
\nonumber
\sup_{\vh \in \Vhp \setminus \{0\}} \frac{|\Ah(u, \vh) - \ell_h(\vh)|}{\Tnorm{\vh}{\CDG}} 
& \le \frac{1}{\sqrt{k_{\star}}}
\Big(\sum_{F \in \FhI} \frac{\hKF}{\xiF} \Norm{\mvl{(\Id - \Pih)(\bk \nabla u)}_{1-\alpha_F}}{L^2(F)^d}^2 \\
\label{eq:estimate-inconsistency}
& \quad + \sum_{F \in \FhD} \frac{\hKF}{\xiF}  \Norm{(\Id - \Pih)(\bk \nabla u) \cdot \nOmega}{L^2(F)}^2 \Big)^{1/2}.
\end{alignat}
\end{lemma}
\begin{proof}
For all~$\vh \in \Vhp$, integration by parts and standard arguments lead to the following identity:
\begin{alignat*}{3}
\Ah(u, \vh) - \ell_h(\vh) = \int_{\FhI} \mvl{(\Id - \Pih)(\bk \nabla u)}_{1- \alpha_F} \cdot \jump{\vh} \dS + \int_{\FhD} \vh (\Id - \Pih)(\bk \nabla u) \cdot \nOmega \dS,
\end{alignat*}
which, together with the Cauchy--Schwarz inequality and bounds~\eqref{eq:jump-lifting-bounds-1} and~\eqref{eq:jump-lifting-bounds-2}, implies~\eqref{eq:estimate-inconsistency}.
\end{proof}

The following local quasi-uniformity conditions are used to simplify some terms in the \emph{a priori} error estimate in Theorem~\ref{thm:a-priori-estimate}. 
\begin{assumption}[Local quasi-uniformity]
\label{asm:local-quasi-uniformity}
We assume that~$\{\Omegah\}_{h > 0}$ and the degree vectors~$\p = (\pK)_{K \in \Omegah}$ satisfy the following local quasi-uniformity conditions: for all neighboring elements~$K , K' \in \Omegah$ 
\begin{equation*}
\hK \simeq h_{K'} \quad \text{ and } \quad \pK \simeq p_{K'}.
\end{equation*}
Moreover, for all~$\calK \in \Tchar$, we assume that
\begin{equation*}
\hcalK \lesssim \hK \quad \text{ and } \quad \pK \lesssim \pcalK \qquad \text{ for all~$K \in \varphi^{-1}(\calK)$.}
\end{equation*}
\end{assumption}

\begin{theorem}[\emph{A priori} error estimates]
\label{thm:a-priori-estimate}
Let the assumptions of Lemma~\ref{lemma:a-priori-bounds} and Assumption~\ref{asm:local-quasi-uniformity} hold, and assume that~$\{\Omegah\}_{h > 0}$ and the continuous weak solution~$u$ to~\eqref{eq:model} satisfy the assumptions of Lemma~\ref{lemma:approx} with~$\lK > 3/2$. Then, for a sufficiently smooth diffusion coefficient~$\bk$, we have
\begin{alignat}{3}
\nonumber
\Tnorm{u - \uh}{\CDG}^2 & \lesssim \sum_{\calK \in \Tchar} \Big(\pcalK^{-1} + \max_{K \in \varphi^{-1}(\calK)} \Big(\max_{F \in \FKd } \xiF\Big) \Big) \frac{\hcalK^{2\ellK - 2}}{\pcalK^{2\lK - 3}}  \Norm{\UK}{H^{\lK}(\calK)}^2  \\
\label{eq:a-priori}
& \quad + \sum_{\calK \in \Tchar}  \max_{K \in \phi^{-1}(\calK)} \Big(\max_{F \in \FKd} \xiF^{-1} \Big) (1 + \pcalK) \frac{\hcalK^{2\ellK - 2}}{\pcalK^{2\lK - 3}} \Norm{\bk \nabla \UK}{H^{\lK - 1}(\calK)^d}^2,
\end{alignat}
where~$\ellK = \min\{\pcalK + 1, \lK\}$ and~$\FKd := \FK \cap (\FhI \cup \FhD)$.
\end{theorem}
\begin{proof}
Combining the \emph{a priori} error bound in Lemma~\ref{lemma:a-priori-bounds} with the bound on the inconsistency term in Lemma~\ref{lemma:inconsistency}, and choosing~$\wh \in \Vhp$ as 
\begin{equation*}
\wh{}_{|_K} := (\Pihp^{\varphi(K)} u)_{|_{K}} \qquad \forall K \in \Omegah,
\end{equation*} 
we get
\begin{alignat}{3}
\nonumber
\Tnorm{u - \uh}{\CDG}^2 & \lesssim \Norm{\sqrt{\bk} \nablah (u - \Pihp u)}{L^2(\Omega)^d}^2 + \sum_{F \in \FhI} \xiF \Norm{\sqrt{\bk} \LF(\jump{u - \Pihp u})}{L^2(\Omega)^d}^2 \\
\nonumber
& \quad + \sum_{F \in \FhD} \xiF \Norm{\sqrt{\bk} \LD(u_{|_F} - (\Pihp u)_{|_{F}})}{L^2(\Omega)^d}^2 + \sum_{F \in \FhI \cup \FhD} \frac{\hKF}{k_{\star} \xiF} \Norm{(\Id - \Pih)(\bk \nabla u){}_{|_{\KF}}}{L^2(F)^d}^2 \\
\label{eq:error-terms}
& =: J_1 + J_2 + J_3 + J_4.
\end{alignat}
In~\eqref{eq:error-terms} and what follows, we omit the explicit dependence of the projection operator~$\Pihp$ on~$\varphi(K)$ for the sake of simplicity. 

We now estimate the terms~$\{J_i\}_{i = 1}^4$. Since~$\bk \in L^{\infty}(\Omega)^{d \times d}$, by using the approximation properties in Lemma~\ref{lemma:approx} of~$\Pihp$, we have
\begin{alignat}{3}
\label{eq:J1}
J_1 & = \sum_{K \in \Omegah} \Norm{\sqrt{\bk} \nabla (u - \Pihp u)}{L^2(K)}^2 \lesssim \sum_{\calK \in \Tchar} \frac{\hcalK^{2\ellK - 2}}{\pcalK^{2\lK - 2}} \Norm{\UK}{H^{\lK}(\calK)}^2.
\end{alignat}

For each~$K \in \Omegah$, let~$\FKd := \FK \cap (\FhI \cup \FhD)$. Using bounds~\eqref{eq:lifting-jump-bounds-1} and~\eqref{eq:lifting-jump-bounds-2}, the trace inequality~\eqref{eq:continuous-trace}, the approximation properties in Lemma~\ref{lemma:approx} of~$\Pihp$, and Assumption~\ref{asm:local-quasi-uniformity}, we obtain
\begin{alignat}{3}
\nonumber
J_2 + J_3 & = \sum_{F \in \FhI} \xiF \Norm{\sqrt{\bk} \LF(\jump{u - \Pihp u})}{L^2(\Omega)^d}^2 + \sum_{F \in \FhD} \xiF \Norm{\sqrt{\bk} \LD(u_{|_F} - (\Pihp u)_{|_F})}{L^2(\Omega)^d}^2 \\
\nonumber
& \lesssim \sum_{F \in \FhI} \frac{\xiF (\pKF + 1)(\pKF +d)}{\hKF} \Norm{\jump{u - \Pihp u}}{L^2(F)^d}^2 \\
\nonumber
& \quad + \sum_{F \in \FhD} \frac{\xiF(\pKF + 1)(\pKF + d)}{\hKF} \Norm{u_{|_F} - (\Pihp u)_{|_F}}{L^2(F)}^2 \\
\nonumber
& \lesssim \sum_{K \in \Omegah} \sum_{F \in \FKd} \frac{\xiF (\pKF + 1)(\pKF + d)}{\hKF} \Big(\frac{\pK}{\hK} \Norm{u - \Pihp u}{L^2(\sKF)}^2 + \frac{\hK}{\pK} \Seminorm{u - \Pihp u}{H^1(\sKF)}^2\Big) 
\\
\nonumber
& \lesssim \sum_{K \in \Omegah} \Big(\max_{F \in \FKd} \xiF \Big) \Big(\frac{\pK^3}{\hK^2} \Norm{u - \Pihp u}{L^2(K)}^2 + \pK\Seminorm{u - \Pihp u}{H^1(K)}^2\Big) 
\\
\label{eq:J2-J3}
& \lesssim \sum_{\calK \in \Tchar} \max_{K \in \varphi^{-1}(\calK)} \Big(\max_{F \in \FKd} \xiF\Big) \frac{\hcalK^{2\ellK-2}}{\pcalK^{2\lK-3}} \Norm{\UK}{H^{\lK}(\calK)}^2.
\end{alignat}

We denote by~$\PPihp$ the natural extension of the operator~$\Pihp$ to $d$-vector-valued functions. To estimate the terms~$J_4$ and~$J_5$, which arise from the inconsistency of the method, we use the triangle inequality, the trace inequalities~\eqref{eq:polynomial-trace-inverse} and~\eqref{eq:continuous-trace}, the stability of~$\Pih$ in the~$L^2(K)$ norm, the approximation properties of~$\PPihp$, and Assumption~\ref{asm:local-quasi-uniformity}, as follows:
\begin{alignat}{3}
\nonumber
J_4 & = \frac{1}{k_{\star}}\sum_{F \in \FhI \cup \FhD} \frac{\hKF}{\xiF} \Norm{(\Id - \Pih) (\bk \nabla u)_{|_{\KF}}}{L^2(F)^d}^2 \\
\nonumber
& \lesssim \sum_{F \in \FhI \cup \FhD} \frac{\hKF}{\xiF} \Big(\Norm{(\Id - \PPihp)(\bk \nabla u)_{|_{\KF}}}{L^2(F)^d}^2 + \Norm{\Pih(\Id - \PPihp)(\bk \nabla u)_{|_{\KF}}}{L^2(F)^d}^2\Big) \\
\nonumber
& \lesssim \sum_{F \in \FhI \cup \FhD} \frac{\hKF}{\xiF} \Big(
\frac{\pKF}{\hKF} \Norm{(\Id - \PPihp) (\bk \nabla u)}{L^2(\sKF)^d}^2 + \frac{\hKF}{\pKF}\Seminorm{(\Id - \PPihp)(\bk \nabla u)}{H^1(\sKF)^d}^2 \\
\nonumber
& \quad + \frac{C_s(\pKF + 1)(\pKF + d)}{\hKF} \Norm{\Pih(\Id - \PPihp)(\bk \nabla u)}{L^2(\sKF)^d}^2
\Big) \\
\nonumber
& \lesssim \sum_{K \in \Omegah} \Big(\max_{F \in \FKd} \xiF^{-1} \Big)  \Big[\big(\pK + \pK^2 \big) \Norm{(\Id - \PPihp)(\bk \nabla u)}{L^2(K)^d}^2  + \frac{\hK^2}{\pK}\Seminorm{(\Id - \PPihp)(\bk \nabla u)}{H^1(K)^d}^2\Big] \\
\label{eq:J4-J5}
& \lesssim \sum_{\calK \in \Tchar}  \max_{K \in \phi^{-1}(\calK)} \Big(\max_{F \in \FKd} \xiF^{-1} \Big) (1 + \pcalK) \frac{\hcalK^{2\ellK - 2}}{\pcalK^{2\lK - 3}} \Norm{\bk \nabla \UK}{H^{\lK - 1}(\calK)^d}^2.
\end{alignat}
Combining~\eqref{eq:J1}, \eqref{eq:J2-J3}, and~\eqref{eq:J4-J5} with \eqref{eq:error-terms}, we get~\eqref{eq:a-priori}.
\end{proof}

\begin{remark}[Additional suboptimality in~$\pcalK$]
The factor $(1+\pcalK)$ in the last term of \eqref{eq:a-priori} leads to a suboptimality in~$\pcalK$ that is a power higher than the one typically obtained for most DG methods. This occurs because the simple choice of the extension polynomials used in the proof of the bounds~\eqref{eq:jump-lifting-bounds-1} and~\eqref{eq:jump-lifting-bounds-2} does not yield the $\pK^{-1}$ factor obtained for %
tensor-product meshes (see \cite[Lemma 7.1]{Schotzau_Schwab_Toselli:2002} and~\cite[Prop.~3]{Pazner:2020}).
Recovering the factor~$\pK^{-1}$ for general polytopal meshes is challenging and requires designing a suitable extension polynomial operator. %

A simple way to avoid such suboptimality is to add, in~\eqref{eq:CDG-variational}, the same stability term used in the LDG method. However, this would over-stabilize the scheme; in fact, the terms in~\eqref{eq:CDG-variational} involving the local lifting functions already act as a stabilization mechanism in the CDG method (see~Remark~\ref{rem:stabilization}). Moreover, the benchmark numerical test in Section~\ref{sec:p-convergence} shows that the CDG method achieves optimal~$p$-convergence rates in situations where the $p$-version of the IPDG method exhibits the theoretically predicted suboptimal convergence by half an order; cf. \cite{georgoulis_suboptimality_2010}.
\eremk
\end{remark}

\section{Computational aspects}
\label{sec:comp_aspects}
In this section, we present fast algorithms to assemble the linear system~$\mathcal{A} U = \mathbf{b}$ for the CDG, the LDG, and the BR2 methods. 
In particular, these algorithms show that the assembling of the stiffness matrices can be carried out without storing large block-structured matrices. 
The ideas here are inspired by the analogous algorithm introduced in~\cite{Castillo_Sequeira:2013} for the LDG method on simplicial meshes. 

Let~$\{K_r\}_{r = 1}^{\Nh}$ be a prescribed order of the elements of the polytopal mesh~$\Omegah$ with total number of elements~$\Nh$. For~$r = 1, \ldots, \Nh$, we set~$\lKr := \dim(\Pp{\pKr}{K_r})$, and fix bases~$\Phi_r := \{\phi_i^r\}_{i = 1}^{\lKr}$ and~$\bS_r := \{\bs_i^r\}_{i = 1}^{d\cdot\lKr}$ of~$\Pp{\pKr}{K_r}$ and~$\Pp{\pKr}{K_r}^d$, respectively. Moreover, to avoid unnecessary computations, we set~$\bS_r = [\Phi_r]^d$. 
Fixed bases for the global spaces~$\Vhp$ and~$\Mhp$ can then be naturally defined. The implemented code is based on the \texttt{lymph} library and the bases are constructed from the tensor product of the Legendre polynomials in one dimension, considering only the 2D polynomials of degree less than or equal to~$p_{K_r}$. The computation of the integrals is based on a quadrature-free approach for the matrices and on a sub-tessellation of the polygons coupled with Gauss-Legendre rules for the forcing term and the face integrals (see \cite[\S4.2]{antonietti_lymph_2024}).

Let~$\M$, $\D$, $\Bgrad$, and~$\Bav$ be the matrices associated with the following bilinear forms:
\begin{alignat*}{3}
\mh(\qh, \rh) & := \sum_{K \in \Omegah} \int_K \qh \cdot \rh \dx  & & \qquad \forall (\qh, \rh) \in \Mhp \times \Mhp, \\
\dh(\qh, \sh) & := \sum_{K \in \Omegah} \int_K \bk \qh \cdot \sh \dx  & & \qquad \forall (\qh, \sh) \in \Mhp \times \Mhp, \\
\bgradh(\uh, \rh) & := - \sum_{K \in \Omegah} \int_K \nabla \uh \cdot \rh \dx & & \qquad \forall (\uh, \rh) \in \Vhp \times \Mhp, \\
\bav(\uh, \rh) & := \int_{\FhI} \jump{\uh} \cdot \mvl{\rh}_{1 - \alpha_F} \dS + \int_{\FhD} \uh \rh \cdot \nOmega \dS & & \qquad \forall (\uh, \rh) \in \Vhp \times \Mhp.
\end{alignat*}
By definition, the matrices~$\M$, $\D$, and~$\Bgrad$ have a block-diagonal structure and their $r$th diagonal blocks, which involve only volume integrals, are given by
\begin{alignat*}{3}
\M_{rr} := \bigg[\int_{K_r} \bs_i^r \cdot \bs_j^r \dx\bigg]_{ij}, \qquad 
\D_{rr} := \bigg[\int_{K_r} \bs_i^r \cdot \bk_{|_{K_r}} \bs_j^r \dx\bigg]_{ij}, \qquad 
\Bgrad_{rr}  := -\bigg[\int_{K_r}  \bs_i^r \cdot \nabla \phi_j^r \dx \bigg]_{ij},
\end{alignat*}
where the ranges of the indices~$i$ and~$j$ have been omitted for the sake of simplicity.

The matrix~$\Bav$ has a block structure. In the cases of CDG and weighted LDG (henceforth identified as~$\LDGw$) methods, a block~$\Bav_{rs}$ is nonzero only if there is a facet~$F \in \FKrout \cap \FKs$ (including the trivial case~$r = s$). For each element~$K_r \in \Omegah$, let~$\{F_r^{(\ell)}\}_{\ell = 1}^{\card(\FKrout)}$ be the facets in~$\FKrout$ with a prescribed order. Each facet~$F_r^{(\ell)}$ adds the following contributions to the diagonal block~$\Bav_{rr}$, and if~$F_r^{(\ell)} \in \FhI$ is shared by~$K_r$ and~$K_{s(\ell)}$, it also contributes to the off-diagonal block~$\Bav_{rs(\ell)}$ as follows:
\begin{equation*}
   \Bav_{rr}{}^{(\ell)} = \bigg[\int_{F_r^{(\ell)}} \bs_{i}^r \cdot (\phi_j^r \nKr)\dS \bigg]_{ij}, \qquad \Bavl_{rs(\ell)} = -\bigg[\int_{F_r^{(\ell)}} \bs_i^r \cdot (\phi_j^{s(\ell)} \nKr) \dS \bigg]_{ij}.
\end{equation*}
We now focus on the BR2 and the LDG methods with~$\alpha_F = 1/2$ for all~$f \in \FhI$. This choice of the %
weighted-average parameters leads to a full stencil for the LDG method (which is therefore identified as $\LDGf$).
For each element~$K_r \in \Omegah$, let~$\{F_r^{(\ell)}\}_{\ell = 1}^{\card(\FKr)}$ %
be the facets in~$\FKr$ with a prescribed order. Then, the two associated block contributions are
\begin{equation*}
   \Bav_{rr}{}^{(\ell)} = \bigg[\frac{1}{2}\int_{F_r^{(\ell)}} \bs_{i}^r \cdot (\phi_j^r \nKr)\dS \bigg]_{ij}, \qquad \Bavl_{rs(\ell)} = -\bigg[\frac{1}{2}\int_{F_r^{(\ell)}} \bs_i^r \cdot (\phi_j^{s(\ell)} \nKr) \dS \bigg]_{ij}.
\end{equation*}
In addition, let~$\mathbf{b}_f$ and~$\mathbf{g}$ be the vectors associated with the following linear functionals:
\begin{alignat*}{3}
\ell(\vh) = \int_{\Omega} f \vh \dx \quad \forall \vh \in \Vhp, \qquad \text{ and } \quad  \gh(\rh) & := \int_{\FhD} g_\mathrm{D} \rh \cdot \nOmega \dS \quad \forall \rh\in\Mhp.
\end{alignat*}
\begin{algorithm}[H]
\caption{\sc Fast assembly of the linear system for the CDG and BR2 methods}
\label{alg:CDG-BR2}
\SetKw{INPUT}{Input:}
\SetKw{TO}{ to }
\SetKw{SET}{Set}
\SetKw{INITIALIZE}{Initialize}
\SetKw{STOP}{stop}
\SetKw{COMPUTE}{Compute}
\SetKw{SOLVE}{Solve}
\INITIALIZE the matrix $\calA$ and the vector $\mathbf{b}$\\
\For {$r = 1$ \TO $\Nh$}{
    \COMPUTE $\M_{rr}^{-1}$, $\D_{rr}$, $\Bgrad_{rr}$, and~$\mathbf{b}_{f, r}$ \\
    \SET $\calD_{rr} = \M_{rr}^{-1} \D_{rr} \M_{rr}^{-1}$ and $\mathbf{b}_r = \mathbf{b}_{f, r}$\\
    $\calA_{rr} = \calA_{rr} + [\Bgrad_{rr}]^{\sf T} \calD_{rr} [\Bgrad_{rr}]$   \\
    \SET $L_r^{\star} = \card(\FKrout)$ (for CDG) or~$L_r^{\star} = \card(\FKr)$ (for BR2) \\  
    \For {$\ell = 1$ \TO $L_r^{\star}$
    }{
          \eIf {$F_r^{(\ell)} \in \FhI$}{
            \SET $s(\ell)$ as the index of the neighbor sharing~$F_r^{(\ell)}$\\
            \COMPUTE the diagonal contribution~$\Bav_{rr}{}^{(\ell)}$, the off-diagonal one~$\Bavl_{rs(\ell)}$, and $\mathbf{g}_r^{(\ell)}$\\
            $\calA_{rr} = \calA_{rr} - [\Bgrad_{rr}]^{\sf T} \calD_{rr} [\Bavl_{rr}] - [\Bavl_{rr}]^{\sf T} \calD_{rr} [\Bgrad_{rr}] + \chi_{F_r^{(\ell)}} [\Bav_{rr}{}^{(\ell)}]^{\sf T} \calD_{rr} [\Bav_{rr}{}^{(\ell)}]$\\
            $\calA_{rs(\ell)} = \calA_{rs(\ell)} - [\Bgrad_{rr}]^{\sf T} \calD_{rr}[\Bavl_{rs(\ell)}] + \chi_{F_r^{(\ell)}} [\Bav_{rr}{}^{(\ell)}]^{\sf T} \calD_{rr} [\Bavl_{rs(\ell)}]$ \\
            $\calA_{s(\ell)r} = \calA_{s(\ell)r} - [\Bavl_{rs(\ell)}]^{\sf T} \calD_{rr} [\Bgrad_{rr}] + \chi_{F_r^{(\ell)}} [\Bavl_{rs(\ell)}]^{\sf T} \calD_{rr} [\Bav_{rr}{}^{(\ell)}]$ \\
            $\calA_{s(\ell)s(\ell)} = \calA_{s(\ell)s(\ell)} 
            + \chi_{F_r^{(\ell)}} [\Bavl_{rs(
            \ell)}]^{\sf T} \calD_{rr} [\Bavl_{rs(\ell)}]$
          }{
            \COMPUTE the diagonal contribution~$\Bav_{rr}{}^{(\ell)}$\\
            $\mathbf{b}_{r} = \mathbf{b}_{r} - [\Bgrad_{rr}]^{\sf T} \calD_{rr} [\mathbf{g}_{r}^{(\ell)}] + \chi_{F_r^{(\ell)}} [\Bav_{rr}{}^{(\ell)}]^{\sf T} \calD_{rr} [\mathbf{g}_{r}^{(\ell)}]$ \\
            $\calA_{rr} = \calA_{rr} - [\Bgrad_{rr}]^{\sf T} \calD_{rr} [\Bav_{rr}{}^{(\ell)}] - [\Bav_{rr}{}^{(\ell)}]^{\sf T} \calD_{rr} [\Bgrad_{rr}]  + \chi_{F_r^{(\ell)}} [\Bav_{rr}{}^{(\ell)}]^{\sf T} \calD_{rr} [\Bav_{rr}{}^{(\ell)}]$\\
          }
    }
}
\end{algorithm}

The computation of~$\mathbf{b}_f$, which is the same for all the methods considered, is simple and does not require further details. As for~$\mathbf{g}$, the following vector contribution is obtained for any~$F_r^{(\ell)} \in \FhD \cap \FKr$:
\begin{alignat*}{3}
\mathbf{g}_{r}^{(\ell)} := \bigg[ - \int_{F_r^{(\ell)}} \bs_i^r \cdot (g_\mathrm{D} \nKr) \dS \bigg]_{i}.
\end{alignat*}

It is then evident that all the local contributions involved in the assembling of the linear systems for these methods are essentially the same, which allows for their implementation in an almost unified framework. 
In Algorithm~\ref{alg:CDG-BR2}, we report the procedure to assemble the stiffness matrix $\calA$ and the right-hand side vector~$\mathbf{b}$ for the CDG and BR2 methods. 
From this algorithm, it can be easily seen that the choice~$\alpha_F = 0$ or~$\alpha_F = 1$ in the CDG method allows us to skip lots of computations, 
as it reduces the number $L_r^{\star}$ introduced in \texttt{line 6} and, consequently, the number of loop iterations in the following lines.
In Algorithm~\ref{alg:LDG}, we report the corresponding algorithm for both version of the LDG method.
\begin{remark}[Stability matrix in LDG method]
The implementation of the LDG method also involves a stability matrix $S$, and an additional vector $\tilde{\mathbf{g}}$ associated with the Dirichlet boundary condition. 
More precisely, these terms are associated with the bilinear form and the linear functional 
\begin{equation*}
\mathbf{s}_h(\uh,\vh) := \int_{\FhI} \eta_F \jump{\uh} \cdot \jump{\vh} \dS + \int_{\FhD} \eta_F \uh \vh \dS, \qquad \tilde{g}_h(\vh) := \int_{\FhD} \eta_F g_\mathrm{D} \vh \dS.
\end{equation*}
For the detailed definition of the penalty term $\eta_F$ used in our code, we refer to~\cite[Eq.~(4.26) in \S4.3]{Cangiani_Dong_Georgoulis_Houston:2017}. From a computational perspective, the matrix~$S$ has a natural compact stencil and is not affected by the choice of the weighted-average parameters.  
However, the presence of this additional matrix increases the computational cost w.r.t. the CDG method.
\eremk
\end{remark}
\begin{algorithm}[H]
\caption{\sc Fast assembly of the linear system for the LDG method}
\label{alg:LDG}
\SetKw{INPUT}{Input:}
\SetKw{TO}{ to }
\SetKw{SET}{Set}
\SetKw{INITIALIZE}{Initialize}
\SetKw{STOP}{stop}
\SetKw{COMPUTE}{Compute}
\SetKw{SOLVE}{Solve}
\For {$r = 1$ \TO $\Nh$}{
    \COMPUTE $\M_{rr}^{-1}$, $\D_{rr}$, and $\Bgrad_{rr}$ \\
    \SET $\calD_{rr} = \M_{rr}^{-1} \D_{rr} \M_{rr}^{-1}$ \\
    \COMPUTE the diagonal blocks~$\Bav_{rr}$ and~$S_{rr}$, the off-diagonal blocks~$\Bav_{rs}$ and~$S_{rs}$, and the 
    vectors~$\mathbf{b}_{f, r}$, $\mathbf{g}_r$, and $\tilde{\mathbf{g}}_r$\\
    \SET $\B_{rr} = \Bgrad_{rr} + \Bav_{rr}$ and $\B_{rs} = \Bav_{rs}$ \\
    $\calA_{rr} = \calA_{rr} + S_{rr}$ and $\mathbf{b}_r = \mathbf{b}_{f,r} + \tilde{\mathbf{g}}_r$\\
    \For {$i  \in \NK
    $}{$\calA_{ri} = \calA_{ri} + S_{ri}$
    }
    \SET $\mathcal{N}_{K_r}^{\mathrm{loop}} = \NKrout$ (for~$\LDGw$) or     $\mathcal{N}_{K_r}^{\mathrm{loop}} = \NKr$ (for~$\LDGf$) \footnote{$\NKrout$ is defined in~\eqref{def:NKout}, and $\mathcal{N}_{K_r}$ is the set neighbors of $K_r$.} \\
    \For {$i  \in \mathcal{N}_{K_r}^{\mathrm{loop}}$}{
          \For {$j  \in \mathcal{N}_{K_r}^{\mathrm{loop}}$}{
            $\calA_{ij} = \calA_{ij} + [\B_{ri}]^{\sf T} \calD_{rr} [\B_{rj}]$
          }
    \If {$\FKr\cap\FhD \neq \emptyset$}{
        $\mathbf{b}_r  = \mathbf{b}_r + [\B_{ri}]^{\sf T} \mathbf{g}_r$}
    }
}
\end{algorithm}

\footnotetext{$\NKrout$ is defined in~\eqref{def:NKout}, and $\mathcal{N}_{K_r}$ is the set neighbors of $K_r$.}

\section{Numerical experiments}
\label{sec:results}
In this section, we present some numerical tests to assess the accuracy and efficiency of the CDG method. 
In Section~\ref{sec:test_case_convergence}, we discuss the convergence of the CDG method. In Section~\ref{sec:test_case_methods_comparison}, we compare the CDG computational efficiency with respect to that of the LDG and BR2 methods. The choice of the method parameters is studied in Sections~\ref{sec:chi_F}--\ref{sec:choice-alpha-p-variable}. A comparison of the~$p$-version of the CDG and IPDG methods for a singular solution is carried out in Section~\ref{sec:p-convergence}.
\par
All numerical simulations %
are based on the \texttt{lymph} library \cite{antonietti_lymph_2024}, implementing 
DG methods on polytopic meshes. The polygonal meshes are constructed using PolyMesher~\cite{talischi_polymesher_2012}.
\par
For the numerical tests in this section, 
we consider the space domain~$\Omega=(0,1)^2$
and Dirichlet boundary conditions on~$\partial\Omega$ (namely, $\Gamma_\mathrm{N}=\emptyset$), unless otherwise specified. Concerning the diffusion tensor,
we consider $\boldsymbol{\kappa}=\mathbf{I}$, where~$\mathbf{I}$ represents the identity matrix of size~$2$. We adopt the manufactured exact solution $u(x,y) = \sin(2 \pi x)\cos(2 \pi y)$ for the test cases in Sections~\ref{sec:test_case_convergence}--\ref{sec:alpha-F-homogeneous-p}. For the test cases in Sections~\ref{sec:choice-alpha-p-variable}--\ref{sec:p-convergence}, the corresponding exact solutions are specified directly within those sections. In all cases, the forcing term $f$ and the Dirichlet boundary condition $g_\mathrm{D}$ are computed accordingly.

\subsection{Convergence properties of CDG method}
\label{sec:test_case_convergence}
\begin{figure}[t!]
    \begin{subfigure}[b]{0.33\textwidth}
          \resizebox{\textwidth}{!}{\input{Errors_h_L2.tex}}
          \caption{$L^2$ errors w.r.t.~the mesh size~$h$.}
        \label{fig:Errors_h_L2}
    \end{subfigure}
    \begin{subfigure}[b]{0.33\textwidth}
          \resizebox{\textwidth}{!}{\input{Errors_h_DG.tex}}
          \caption{CDG errors w.r.t.~the mesh size~$h$.}
        \label{fig:Errors_h_DG}
    \end{subfigure}    
    \begin{subfigure}[b]{0.33\textwidth}
          \resizebox{\textwidth}{!}{\input{Errors_p.tex}}
          \caption{Errors w.r.t.~the polynomial order~$p$.}
        \label{fig:Errors_p}
    \end{subfigure}
     \caption{Computed errors in $L^2$ norm (a) and CDG norm (b) w.r.t.~the mesh size~$h$ and computed errors in $L^2$ norm and CDG norm (c) w.r.t.~the polynomial order~$p$.}
\end{figure}
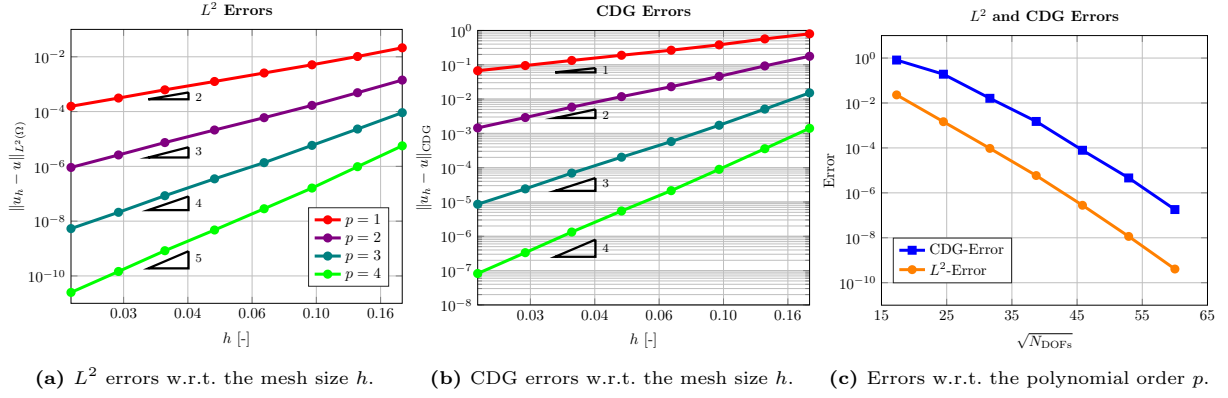
We perform a convergence test for uniform degrees of approximation~$p=1,\, 2,\, 3,\, 4$ and regular Voronoi meshes, using, for each degree, eight different refinements with number of elements $N_\mathrm{el} = 100,\, 200,\, 400, \dots,$\, $12800$.
The weighted-average parameters~$\alpha_F$ and the parameter~$\gamma$ in Assumption~\ref{asm:xiF} are set as 
$\alpha_F=1$ if $\boldsymbol{n}_F\cdot(1,0)^{\top} \geq 0$ (and~$\alpha_F = 0$ otherwise), and $\gamma=0.9$.
\par
In Figure \ref{fig:Errors_h_L2}, we report the computed errors in $L^2(\Omega)$ norm, 
and observe a decrease in the error with order~$\mathcal{O}(h^{p+1})$. 
Moreover, in Figure~\ref{fig:Errors_h_DG}, the errors in the CDG norm show a convergence of order~$\mathcal{O}(h^p)$. Finally, in Figure~\ref{fig:Errors_p}, we report the errors in the CDG and $L^2(\Omega)$ norms for the~$p$-version of the method, namely, maintaining a fixed mesh of 100 elements ($h \approx 0.1851$) and increasing the polynomial order $p=1,\dots,7$. We observe an exponential decay of
the error of order $\mathcal{O}(e^{-c\sqrt{N_\mathrm{DOFs}}})$.
\subsection{Comparison between CDG, BR2, and LDG methods}
\label{sec:test_case_methods_comparison}
As a second test case, we discuss the computational efficiency of the CDG method with respect to the BR2 and LDG methods. In particular, we discuss the LDG method by considering $\LDGf$ and $\LDGw$ versions separately. 
For~$\LDGw$, we have adopted  
the same choice of $\alpha_F$ as for the CDG method, namely, we consider $\alpha_F=1$ if $\beta_F = \boldsymbol{n}_F\cdot(1,0)^{\top} \geq 0$ (and~$\alpha_F = 0$ otherwise). We test different polynomial degrees $p=2,4$ and, for each degree, five different mesh Voronoi-type meshes with $N_\mathrm{el}= 3200,6400,12800,25600,51200$. 
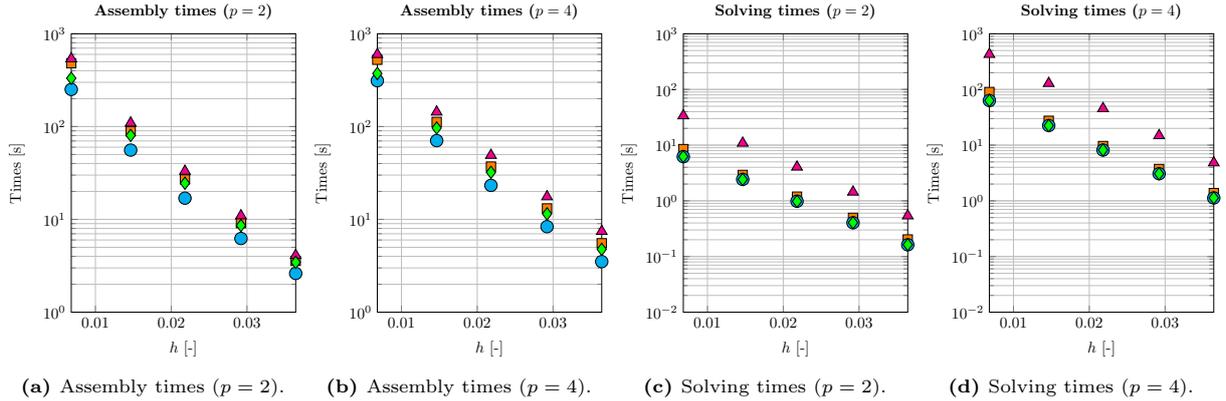
\begin{figure}[t!]

    \begin{subfigure}[b]{0.245\textwidth}
          \resizebox{\textwidth}{!}{\input{AssemblyTimesp2.tex}}
          \caption{Assembly times ($p=2$).}
        \label{fig:Ass_times_p2}
    \end{subfigure}
    \begin{subfigure}[b]{0.245\textwidth}
          \resizebox{\textwidth}{!}{\input{AssemblyTimesp4.tex}}
          \caption{Assembly times ($p=4$).}
        \label{fig:Ass_times_p4}
    \end{subfigure}
    \begin{subfigure}[b]{0.245\textwidth}
          \resizebox{\textwidth}{!}{\input{SolvingTimesp2.tex}}
          \caption{Solving times ($p=2$).}
        \label{fig:Solv_times_p2}
    \end{subfigure}
    \begin{subfigure}[b]{0.245\textwidth}
          \resizebox{\textwidth}{!}{\input{SolvingTimesp4.tex}}
          \caption{Solving times ($p=4$).}
        \label{fig:Solv_times_p4}
    \end{subfigure}
    \caption{Comparison of computational times for the different methods (CDG, BR2, $\LDGw$, and $\LDGf$) and different polynomial orders ($p=2,4$). The considered times are divided in assembly time (a--b) and solving time (c--d).} 
     \label{fig:Computational_Times_Methods}
\end{figure}
We first consider the assembly times according to Algorithms \ref{alg:CDG-BR2} and \ref{alg:LDG}, for the CDG, BR2, and LDG methods.  
In Figures~\ref{fig:Ass_times_p2} and~\ref{fig:Ass_times_p4}, we report the assembly times computed for different polynomial degrees $p=2$ and~$p = 4$. We observe that the CDG method provides a faster assembly compared to all the other methods. In particular, for the finest mesh with $p=4$, the BR2 assembly requires approximately~$19.55\%$ more time than CDG. This is due to the double computations in the face integrals resulting from the choice~$\alpha_F = 1/2$.  
Concerning the LDG methods, the assembly times are higher by approximately~$67.95\%$ (for~$\LDGw$) and~$91.66\%$ (for~$\LDGf$) than that of CDG.  
These large values are mostly due to the presence of the two nested loops over the neighbors, which also causes the larger stencil (see Algorithm \ref{alg:LDG}).
\par
In Figures \ref{fig:Solv_times_p2} and~\ref{fig:Solv_times_p4}, we report the corresponding solving times for different polynomial degrees $p=2$ and~$p = 4$. 
We employ the \texttt{backslash} operator as 
linear solver to ensure a fair comparison, but we expect the more compact sparsity patterns of CDG and BR2 to have a positive impact on the performance of iterative solvers as well. The results show similar solving times for the BR2 and CDG methods, which can be attributed to their identical number of nonzero entries (see Table \ref{tab:conditioning}) and the same stencil (see Figure~\ref{fig:Sparsity_Pattern_Methods}). In particular, the identical sparsity patterns of CDG and BR2 for polygonal meshes occur because the further reduction of nonzero entries obtained in~\cite[\S5.1]{Periaire_Persson:2008} for the CDG method on simplicial meshes is lost. Indeed, on simplices, this reduction results from the use of Lagrangian bases with some degrees of freedom associated with nodes on the element edges. Such a construction is not possible for polygonal elements of arbitrary shape, whose number of edges does not match the dimension of the underlying polynomial space.
Comparing the solving times of CDG and BR2 with those of the LDG methods reveals significantly longer times for the latter. 
In particular, the $\LDGf$ method requires approximately $577.78\%$ more time than CDG, while $\LDGw$ requires about $42.85\%$ more. This is likely due to the much higher number of nonzero entries in the stiffness matrix $\mathcal{A}$, as reported in Table \ref{tab:conditioning} and illustrated in Figure~\ref{fig:Sparsity_Pattern_Methods}. On the other hand, the LDG matrices exhibit better conditioning than those of the CDG method. 
The~$2$-condition number~$\kappa_2(\calA)$ grows
with the same order~$\mathcal{O}(h^{-2})$ for all the discussed methods.
\par
\begin{table}[t]
\begin{subtable}{.49\linewidth}\centering
\begin{tabular}{|c|C|C|C|}
\hline
\textbf{Elements}               
                                & $\boldsymbol{3200}$           
                                & $\boldsymbol{12800}$
                                & $\boldsymbol{51200}$ \\
\textbf{\footnotesize{($\boldsymbol{h}$)}} 
& \footnotesize{($\boldsymbol{0.0337}$)}
& \footnotesize{($\boldsymbol{0.0172}$)} 
& \footnotesize{($\boldsymbol{0.0086}$)} \\ \hline
\multirow{2}{*}{\textbf{CDG}}              
                                & $7.18\times10^{4}$            
                                & $2.75\times10^{5}$            
                                & $1.22\times10^{6}$          \\  
                                & \footnotesize{($197\,280$)}    
                                & \footnotesize{($796\,824$)}  
                                & \footnotesize{($3\,202\,596$)} \\ \hline
\multirow{2}{*}{\textbf{BR2}}              
                                & $6.72\times10^{4}$            
                                & $2.55\times10^{5}$            
                                & $1.10\times10^{6}$          \\  
                                & \footnotesize{($197\,280$)}    
                                & \footnotesize{($796\,824$)}  
                                & \footnotesize{($3\,202\,596$)} \\ \hline   
\multirow{2}{*}{$\boldsymbol{\LDGw}$}             
                                & $7.78\times10^{4}$            
                                & $3.03\times10^{5}$            
                                & $1.36\times10^{6}$          \\  
                                & \footnotesize{($257\,742$)}   
                                & \footnotesize{($1\,043\,838$)}
                                & \footnotesize{($4\,200\,642$)} \\ \hline   
\multirow{2}{*}{$\boldsymbol{\LDGf}$}      
                                & $6.36\times10^{4}$            
                                & $2.53\times10^{5}$            
                                & $1.05\times10^{6}$          \\  
                                & \footnotesize{($533\,268$)}   
                                & \footnotesize{($2\,170\,296$)} 
                                & \footnotesize{($8\,757\,702$)} \\ \hline      
\end{tabular}
\caption{Condition number and nonzero entries of $\mathcal{A}$ ($p=1$).}
\label{tab:cond_p1}
\end{subtable}%
\hspace{+0.0em}
\begin{subtable}{.49\linewidth}\centering
\begin{tabular}{|c|C|C|C|}
\hline
\textbf{Elements}               & $\boldsymbol{3200}$           
                                & $\boldsymbol{12800}$
                                & $\boldsymbol{51200}$ \\
\textbf{\footnotesize{($\boldsymbol{h}$)}} 
& \footnotesize{($\boldsymbol{0.0337}$)}
& \footnotesize{($\boldsymbol{0.0172}$)} 
& \footnotesize{($\boldsymbol{0.0086}$)} \\ \hline
\multirow{2}{*}{\textbf{CDG}}   & $3.93\times10^{5}$            
                                & $1.47\times10^{6}$            
                                & $6.42\times10^{6}$          \\
                                & \footnotesize{($789\,120$)}   
                                & \footnotesize{($3\,187\,296$)}
                                & \footnotesize{($12\,810\,384$)} \\ \hline
\multirow{2}{*}{\textbf{BR2}}   
                                & $3.47\times10^{5}$            
                                & $1.34\times10^{6}$            
                                & $5.72\times10^{6}$          \\
                                & \footnotesize{($789\,120$)}   
                                & \footnotesize{($3\,187\,296$)}
                                & \footnotesize{($12\,810\,384$)} \\ \hline   
\multirow{2}{*}{$\boldsymbol{\LDGw}$}    
                                & $3.71\times10^{5}$            
                                & $1.42\times10^{6}$    
                                & $6.16\times10^{6}$ \\
                                & \footnotesize{($1\,030\,968$)}
                                & \footnotesize{($4\,175\,352$)} 
                                & \footnotesize{($16\,802\,568$)} \\ \hline   
\multirow{2}{*}{$\boldsymbol{\LDGf}$}    
                                & $2.85\times10^{5}$            
                                & $1.13\times10^{6}$          
                                & $4.75\times10^{6}$            \\
                                & \footnotesize{($2\,133\,072$)} 
                                & \footnotesize{($8\,681\,184$)} 
                                & \footnotesize{($35\,030\,808$)} \\ \hline      
\end{tabular}
\caption{Condition number and nonzero entries of $\mathcal{A}$ ($p=2$).}
\label{tab:cond_p2}
\end{subtable}
\\[6pt]
\hspace{-0.6em}
\begin{subtable}{.49\linewidth}\centering
\begin{tabular}{|c|C|C|C|}
\hline
\textbf{Elements}               
                                & $\boldsymbol{3200}$           
                                & $\boldsymbol{12800}$          
                                & $\boldsymbol{51200}$ \\
\textbf{\footnotesize{($\boldsymbol{h}$)}} 
& \footnotesize{($\boldsymbol{0.0337}$)}
& \footnotesize{($\boldsymbol{0.0172}$)}
& \footnotesize{($\boldsymbol{0.0086}$)} \\ \hline
\multirow{2}{*}{\textbf{CDG}}   
                                & $1.17\times10^{6}$
                                & $4.86\times10^{6}$
                                & $1.89\times10^{7}$          \\  
                                & \footnotesize{($2\,192\,000$)} 
                                & \footnotesize{($8\,853\,600$)} 
                                & \footnotesize{($35\,584\,400$)} \\ \hline
\multirow{2}{*}{\textbf{BR2}}   
                                & $1.05\times10^{6}$
                                & $4.26\times10^{6}$
                                & $1.76\times10^{7}$          \\  
                                & \footnotesize{($2\,192\,000$)} 
                                & \footnotesize{($8\,853\,600$)} 
                                & \footnotesize{($35\,584\,400$)} \\ \hline
\multirow{2}{*}{$\boldsymbol{\LDGw}$}            
                                & $1.09\times10^{6}$            
                                & $4.20\times10^{6}$            
                                & $1.81\times10^{7}$          \\  
                                & \footnotesize{($2\,863\,800$)}
                                & \footnotesize{($11\,598\,200$)} 
                                & \footnotesize{($46\,673\,800$)} \\ \hline    
\multirow{2}{*}{$\boldsymbol{\LDGf}$}    
                                & $8.23\times10^{5}$            
                                & $3.33\times10^{6}$              
                                & $1.36\times10^{7}$         \\   
                                & \footnotesize{($5\,925\,200$)}   
                                & \footnotesize{($24\,114\,400$)}  
                                & \footnotesize{($97\,307\,800$)} \\ \hline      
\end{tabular}
\caption{Condition number and nonzero entries of $\mathcal{A}$ ($p=3$).}
\label{tab:cond_p3}
\end{subtable}
\begin{subtable}{.49\linewidth}\centering
\begin{tabular}{|c|C|C|C|}
\hline
\textbf{Elements}     
& $\boldsymbol{3200}$           
& $\boldsymbol{12800}$          
& $\boldsymbol{51200}$ \\
\textbf{\footnotesize{($\boldsymbol{h}$)}} 
& \footnotesize{($\boldsymbol{0.0337}$)}
& \footnotesize{($\boldsymbol{0.0172}$)}
& \footnotesize{($\boldsymbol{0.0086}$)} \\ \hline
\multirow{2}{*}{\textbf{CDG}}             
                                & $2.78\times10^{6}$            
                                & $1.25\times10^{7}$ 
                                & $4.87\times10^{7}$ \\   
                                & \footnotesize{($4\,932\,000$)} 
                                & \footnotesize{($19\,920\,600$)}
                                & \footnotesize{($80\,064\,900$)}\\ \hline
\multirow{2}{*}{\textbf{BR2}}             
                                & $1.05\times10^{6}$            
                                & $4.26\times10^{6}$ 
                                & $1.76\times10^{7}$ \\   
                                & \footnotesize{($4\,932\,000$)} 
                                & \footnotesize{($19\,920\,600$)}
                                & \footnotesize{($80\,064\,900$)}\\ \hline
\multirow{2}{*}{$\boldsymbol{\LDGw}$}     
                                & $2.68\times10^{6}$            
                                & $5.25\times10^{6}$            
                                & $4.46\times10^{7}$          \\
                                & \footnotesize{($6\,443\,550$)}
                                & \footnotesize{($12\,976\,200$)} 
                                & \footnotesize{($105\,016\,050$)} \\ \hline 
\multirow{2}{*}{$\boldsymbol{\LDGf}$} 
                                & $1.97\times10^{6}$            
                                & $8.02\times10^{6}$            
                                & $3.30\times10^{7}$          \\   
                                & \footnotesize{($13\,331\,700$)}   & \footnotesize{($54\,257\,400$)}   & \footnotesize{($218\,942\,550$)} \\ \hline      
\end{tabular}
\caption{Condition number and nonzero entries of $\mathcal{A}$ ($p=4$).}
\label{tab:cond_p4}
\end{subtable}
\caption{Condition number (first row in each cell) and nonzero entries of $\mathcal{A}$ (second row in each cell) for the different methods (CDG, BR2, $\LDGw$, and $\LDGf$) and for different polynomial orders ($p=1,2,3,4$).}
\label{tab:conditioning}
\end{table}
\begin{figure}[t!]
	\centering
	{\includegraphics[width=\textwidth]{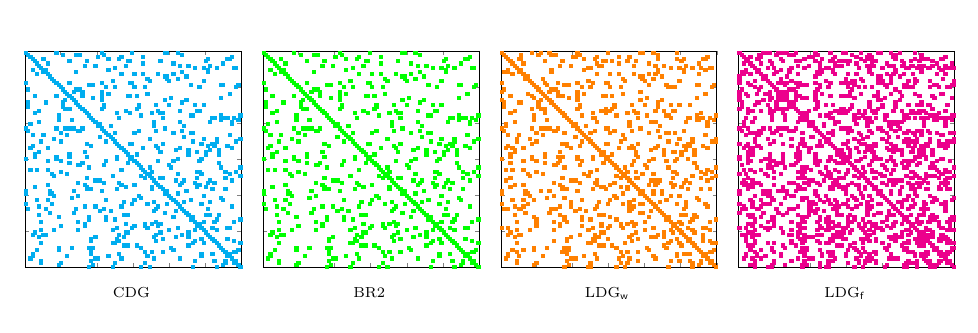}}
    \caption{Comparison of sparsity patterns for the different methods (CDG, BR2, $\LDGw$, and $\LDGf$) with a mesh of 100 elements and polynomial degree $p=1$.}
     \label{fig:Sparsity_Pattern_Methods}
\end{figure}

\subsection{Impact of the parameter \texorpdfstring{$\xiF$}{chi_F} for different types of mesh}
\label{sec:chi_F}
In this section, we discuss the impact of the value of the parameter $\xiF$ on the convergence of the numerical solution and on the coercivity of the bilinear form associated with the CDG method. For this test, we consider four different types of mesh: structured triangular, Cartesian, Voronoi, and agglomerated meshes. In the latter, the agglomeration is performed using the MAGNET library~\cite{antonietti_magnet_2025}. Examples of each mesh family are shown in Figure~\ref{fig:test_case_5.4_meshes} below.
\par
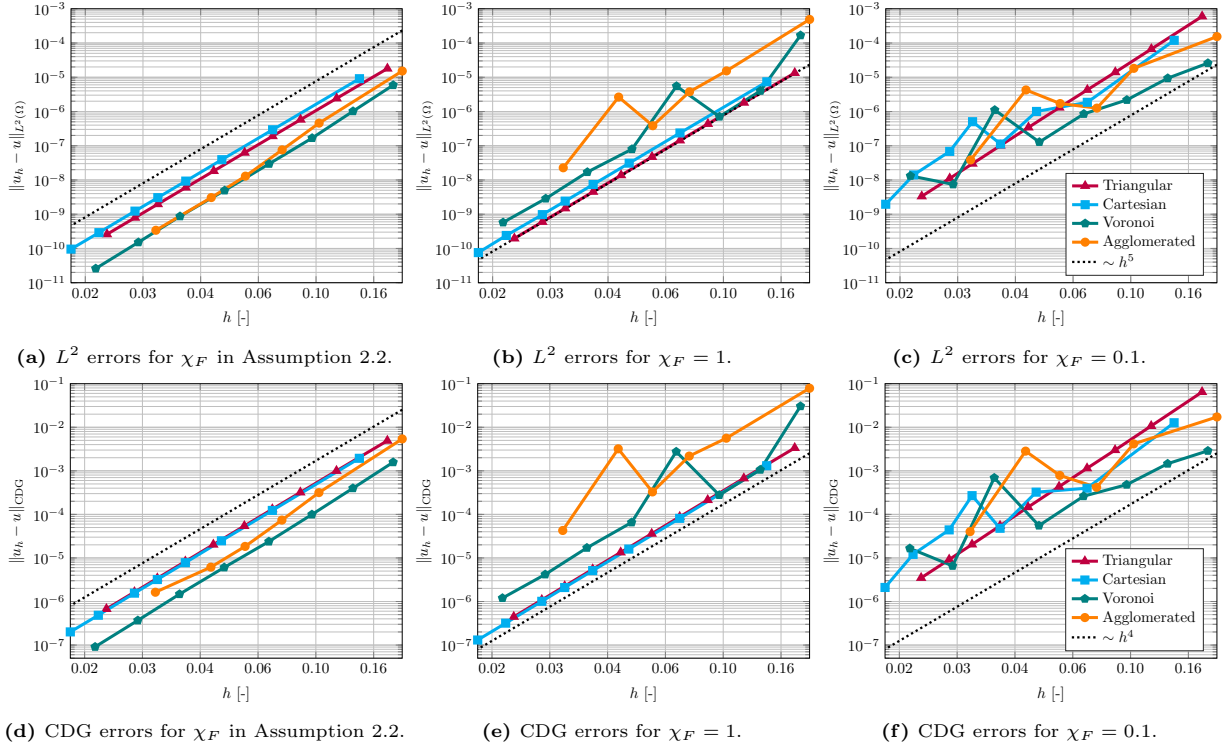
\begin{figure}[t!]
    \begin{subfigure}[b]{0.33\textwidth}
          \resizebox{\textwidth}{!}{\input{MeshComparison_L2_XF.tex}}
          \caption{$L^2$ errors for~$\xiF$ in Assumption 2.2.}
        \label{fig:Errors_h_L2_chiF}
    \end{subfigure}
    \begin{subfigure}[b]{0.33\textwidth}
          \resizebox{\textwidth}{!}{\input{MeshComparison_L2_X1.tex}}
          \caption{$L^2$ errors for~$\xiF=1$.}
        \label{fig:Errors_h_L2_chi=1}
    \end{subfigure}   
    \begin{subfigure}[b]{0.33\textwidth}
          \resizebox{\textwidth}{!}{\input{MeshComparison_L2_X01.tex}}
          \caption{$L^2$ errors for~$\xiF=0.1$.}
        \label{fig:Errors_h_L2_chi=0.1}
    \end{subfigure} \\
    \begin{subfigure}[b]{0.33\textwidth}
          \resizebox{\textwidth}{!}{\input{MeshComparison_DG_XF.tex}}
          \caption{CDG errors for~$\xiF$ in Assumption 2.2.}
        \label{fig:Errors_h_DG_chiF}
    \end{subfigure}    
    \begin{subfigure}[b]{0.33\textwidth}
          \resizebox{\textwidth}{!}{\input{MeshComparison_DG_X1.tex}}
          \caption{CDG errors for~$\xiF=1$.}
        \label{fig:Errors_h_DG_chi=1}
    \end{subfigure} 
    \begin{subfigure}[b]{0.33\textwidth}
          \resizebox{\textwidth}{!}{\input{MeshComparison_DG_X01.tex}}
          \caption{CDG errors for~$\xiF=0.1$.}
        \label{fig:Errors_h_DG_chi=0.1}
    \end{subfigure}
     \caption{Computed errors in $L^2(\Omega)$ norm (a-c) and CDG norm (d-f) w.r.t.~the mesh size~$h$ and for different mesh types. The errors have been computed both considering $\xiF$ as in Assumption~\ref{asm:xiF} (left), $\xiF=1$ (center), and $\xiF=0.1$ (right).}
     \label{fig:Errors_h_chi}
\end{figure}
For each type of mesh, we construct a sequence of refinements that we use to perform a convergence test with a fixed polynomial degree $p=4$. In particular, we consider three different choices of $\xiF$, first as in Assumption~\ref{asm:xiF}, and then fixing it as a constant, namely $\xiF=1$ or $\xiF=0.1$ for all facets~$F \in \FhI \cup \FhD$.
In Figure~\ref{fig:Errors_h_chi}, we report the errors in both the $L^2(\Omega)$ norm (see Figures~\ref{fig:Errors_h_L2_chiF}--\ref{fig:Errors_h_L2_chi=0.1}) and the CDG norm (see Figures~\ref{fig:Errors_h_DG_chiF}--\ref{fig:Errors_h_DG_chi=0.1}). We observe monotone convergence with the correct order for all mesh types only for the choice of $\xiF$ in Assumption~\ref{asm:xiF} (Figures~\ref{fig:Errors_h_L2_chiF} and~\ref{fig:Errors_h_DG_chiF}). On the contrary, the constant choice $\xiF=1$ causes loss of monotone convergence in meshes with a larger number of faces per element, namely the Voronoi and agglomerated meshes (Figures~\ref{fig:Errors_h_L2_chi=1} and~\ref{fig:Errors_h_DG_chi=1}). Finally, a significant loss of the convergence rates is observed for the choice $\xiF=0.1$ 
also for the Cartesian grids (Figures~\ref{fig:Errors_h_L2_chi=0.1} and~\ref{fig:Errors_h_DG_chi=0.1}). No significant changes are observed for the triangular mesh in terms of convergence rates; however, the errors are one order larger for~$\xiF = 0.1$ than for the other cases.
\par
\begin{figure}[t!]
    \resizebox{\textwidth}{!}{\input{MinimumEigenvalueChi.tex}} 
    \caption{Minimum eigenvalue of the matrix $\mathcal{A}$ for different meshes and polynomial orders: coarse mesh with $p=1$ (left), coarse mesh with $p=4$ (center), and fine mesh with $p=1$ (right). Each plot is divided into an upper panel where the eigenvalue is positive in log scale and a lower panel for the negative values in normal scale. The vertical lines are the corresponding~$\max(\NF)$ for each mesh, that is the theoretical threshold
    for~$\xiF$.
    \label{fig:minimum_eigval}}
\end{figure}
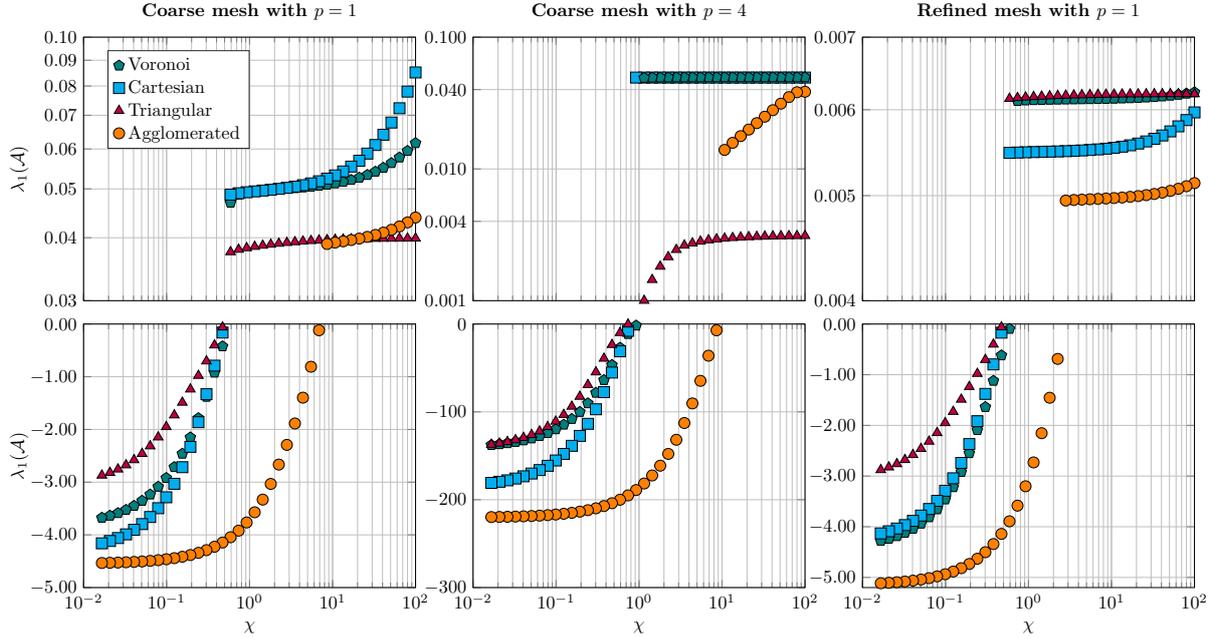
To analyze the reasons for the loss of convergence, we plot the minimum eigenvalue of the matrix $\mathcal{A}$ for the different mesh types. In particular, we consider a coarse mesh ($\sim 100$ elements) with $p=1$ and $p=4$ and a fine mesh ($\sim 1000$ elements) with $p=1$. We report the results in Figure~\ref{fig:minimum_eigval}, where we observe that eigenvalue positivity depends on the choice of $\xiF$. The vertical lines indicate the value of~$\max(\NF)$ for each mesh, which corresponds to the theoretical threshold for $\xiF$. 
We observe that the positivity of the eigenvalues is preserved for values of $\chi$ smaller than those prescribed by Assumption~\ref{asm:xiF}. The coarse meshes employed in this test case are shown in Figure~\ref{fig:test_case_5.4_meshes}, together with the associated $\max(\NF)$ and $\mathrm{mean}(\NF)$.
\begin{figure}[t!]
	\centering
	{\includegraphics[width=0.97\textwidth]{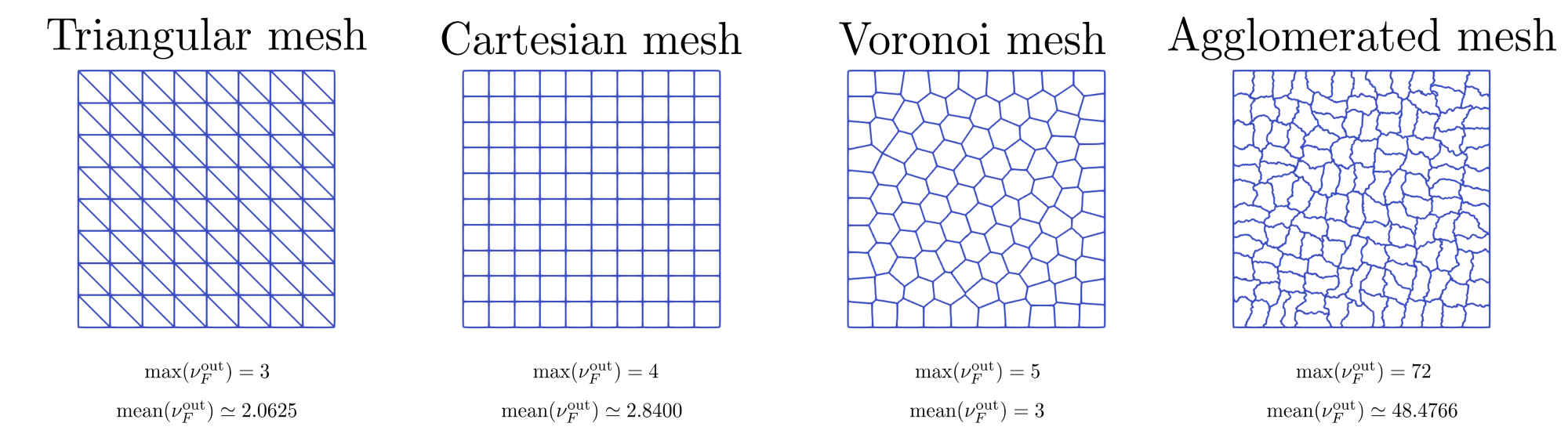}}
	\caption{Coarse computational meshes ($\simeq 100$ elements) with associated $\max(\NF)$ and $\mathrm{mean}(\NF)$.}
	\label{fig:test_case_5.4_meshes}
\end{figure}
\par
The above discussion explains why the agglomerated mesh, where each element has a large number of facets (see Figure \ref{fig:test_case_5.4_meshes}), requires a larger $\xiF$ to maintain coercivity. Indeed, a refinement of the agglomerated mesh that reduces the number of facets restores eigenvalue positivity for smaller values of $\xiF$. Coherently, no significant changes occur for the other mesh types upon refinement. Moreover, the minimum value of~$\chi$ that produces a positive-definite stiffness matrix depends just mildly on the degree of approximation~$p$. However, for values of~$\chi$ below this threshold, substantially larger negative values are obtained. 
\subsection{Choice of~\texorpdfstring{$\alpha_F$}{alpha_F} in the \texorpdfstring{$p$}{p}-homogeneous case}
\label{sec:alpha-F-homogeneous-p}
In this test case, we discuss the freedom in choosing the parameters $\alpha_F$ in the case of a homogeneous polynomial degree $p$.
We perform a simulation on the domain $\Omega = (0,1)^2$, using an agglomerated mesh of 128 elements for both $p=2$ and $p=5$.
We consider different criteria for choosing~$\alpha_F$ based on: directional fluxes, the prescribed ordering of element indices, element sizes, or a random choice.
As illustrated in Table~\ref{tab:test_case_aF}, all criteria for  setting $\alpha_F = 1$ lead to comparable CDG errors and condition numbers of the matrix $\mathcal{A}$, with only marginal variations among the different strategies. This indicates that, at least for this test case,  the method is rather insensitive to the particular choice of $\alpha_F$.
\begin{table}[t]
\centering
\renewcommand{\arraystretch}{1.2}
\begin{tabular}{|c|c|c|c|c|}
\hline
 \multirow{2}{*}{\textbf{Criterion to set} {$\alpha_F = 1$}} & \multicolumn{2}{c|}{{$\boldsymbol{p=2}$}} & \multicolumn{2}{c|}{{$\boldsymbol{p=5}$}} \\ \cline{2-5}
 & $\Tnorm{u-u_h}{\mathrm{CDG}}$ & $\kappa_2(\mathcal{A})$ & $\Tnorm{u-u_h}{\mathrm{CDG}}$ & \textbf{$\kappa_2(\mathcal{A})$}  \\\hline
$\boldsymbol{n}_F\cdot(1,0)^{\top} \geq 0$  & $2.20\times10^{-1}$ &  $3.34\times10^{4}$ & $1.45\times10^{-4}$ &  $1.41\times10^{7}$ \\ \hline
$\boldsymbol{n}_F\cdot(1,1)^{\top} \geq 0$ & $2.18\times10^{-1}$ &  $3.13\times10^{4}$ & $1.46\times10^{-4}$ &  $1.38\times10^{7}$ \\ \hline
$\boldsymbol{n}_F\cdot(0,1)^{\top} \geq 0$ & $2.18\times10^{-1}$ &  $3.43\times10^{4}$ & $1.38\times10^{-4}$ &  $1.38\times10^{7}$ \\ \hline
$\mathrm{id}(K_1)>\mathrm{id}(K_2)$        & $2.28\times10^{-1}$ &  $5.39\times10^{4}$ & $1.32\times10^{-4}$ &  $1.53\times10^{7}$ \\ \hline
$|K_1|>|K_2|$ 
& $2.29\times10^{-1}$ &  $6.10\times10^{4}$ & $1.35\times10^{-4}$ &  $1.33\times10^{7}$ \\ \hline
Random 
& $2.32\times10^{-1}$ &  $6.10\times10^{4}$ & $1.56\times10^{-4}$ &  $2.00\times10^{7}$ \\ \hline
\end{tabular}
\caption{Computational CDG errors and $2$-condition number of matrix $\mathcal{A}$ w.r.t. different criteria for choosing $\alpha_F$.}
	\label{tab:test_case_aF}
\end{table}
\subsection{Choice of~\texorpdfstring{$\alpha_F$}{alpha_F} in the \texorpdfstring{$p$}{p}-variable case}
\label{sec:choice-alpha-p-variable}
\begin{figure}[t!]
	\centering
	{\includegraphics[width=\textwidth]{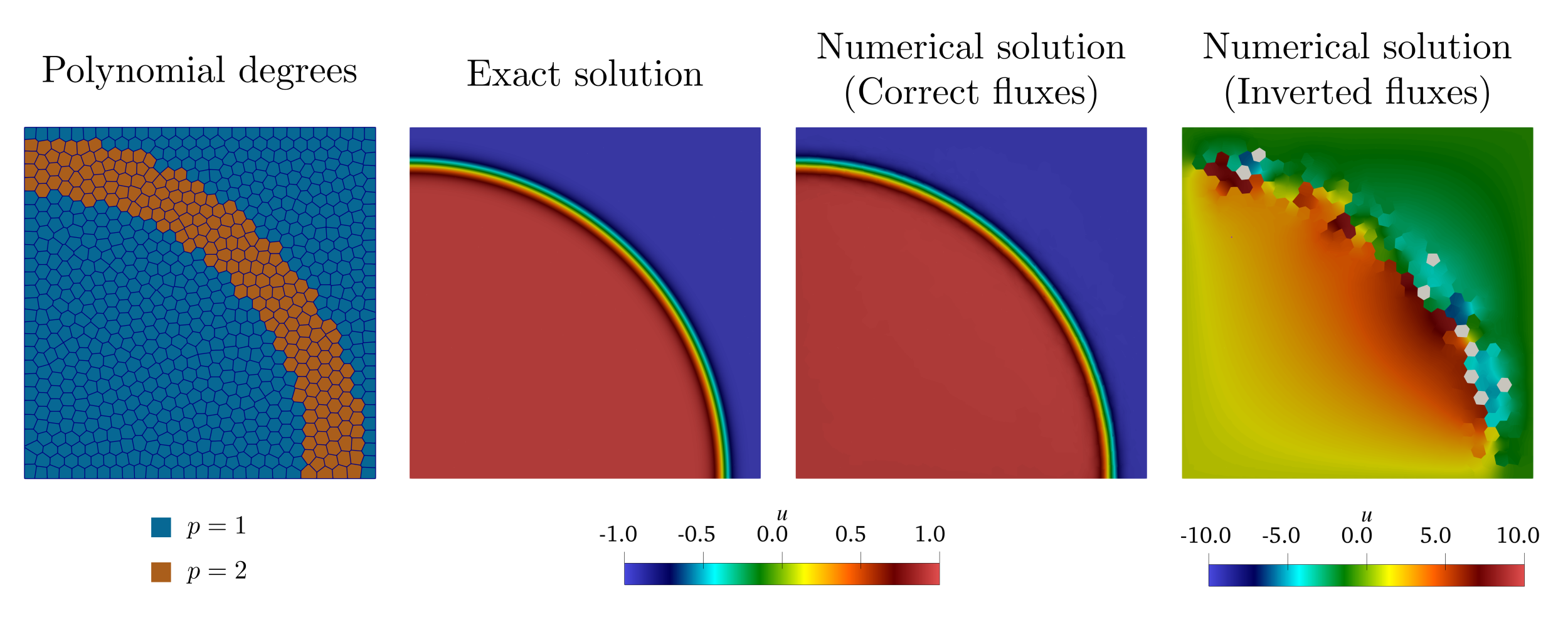}}
	\caption{Computational mesh and polynomial degrees distribution (first figure), exact solution of the problem (second figure), and numerical solutions computed with the correct fluxes (third figure) and with the inverted ones (fourth figure). In the latter, the color gray %
    corresponds to \texttt{NaN} values.}
	\label{fig:test_case_5.4}
\end{figure}
We now study the impact of Assumption~\ref{asm:alpha_F} on the stability of the numerical method for variable degrees of approximation.
We perform a simulation on the domain $\Omega = (0,1)^2$, using a regular Voronoi mesh of 800 elements (see Figure~\ref{fig:test_case_5.4}, leftmost panel). We consider the problem with exact solution~$u(x,y) = \tanh\left(-20(x^2+y^2-0.8)\right)$ and a variable polynomial degree with~$p = 2$ in the region of highest gradient, and~$p = 1$ elsewhere, as illustrated in Figure~\ref{fig:test_case_5.4} (first panel).
\par
In Figure~\ref{fig:test_case_5.4} (third panel), we report the numerical solution computed with   
$\alpha_F$ chosen as in Assumption~\ref{asm:alpha_F} (``correct fluxes"), which provides an accurate approximation. Moreover, we report the numerical solution obtained by choosing ``inverted fluxes", namely by choosing~$\alpha_F$ such that the trace of $\sigmah^F$ is taken from the lowest-order mesh element in the numerical fluxes. 
Figure~\ref{fig:test_case_5.4}  (rightmost panel) exhibits a severe degradation of the approximation, resulting from the ill-conditioning of the stiffness matrix.
\par
\begin{figure}[t!]
	\centering
	{\includegraphics[width=\textwidth]{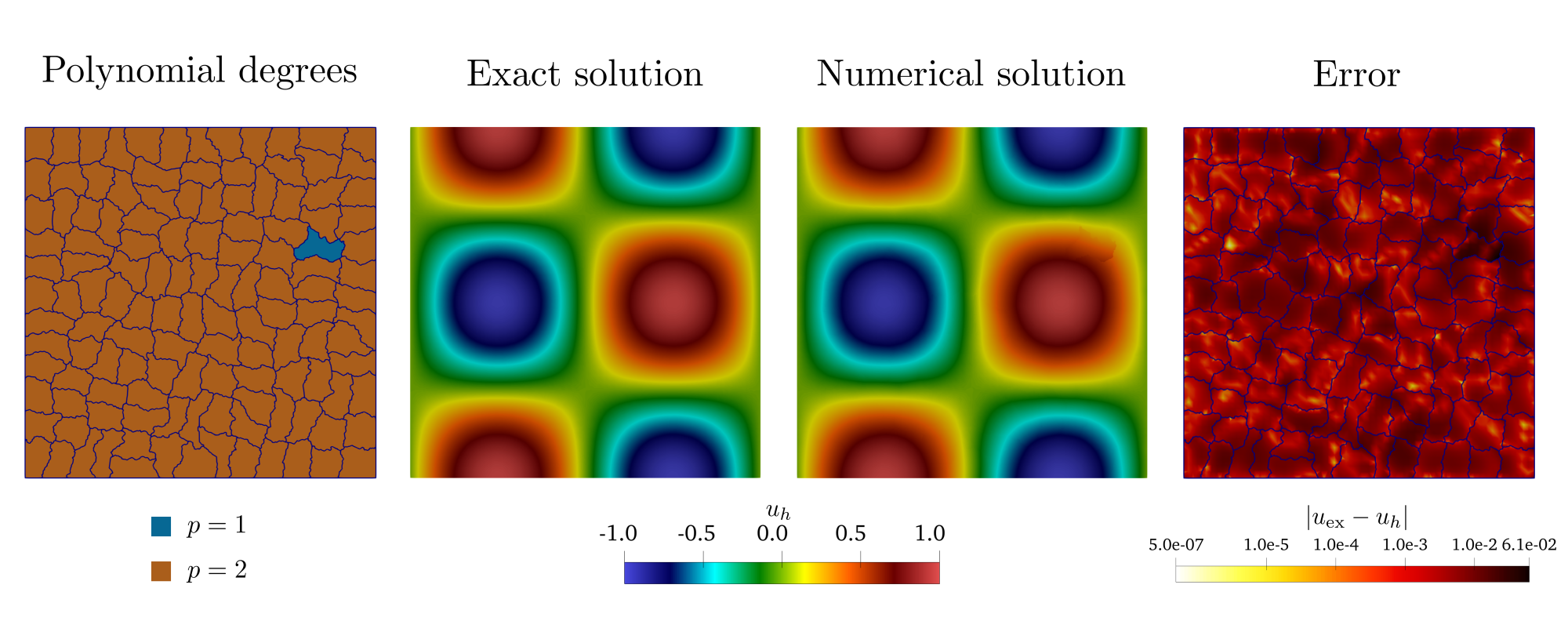}}
	\caption{ 
    Agglomerated mesh and polynomial degrees distribution (first figure), exact solution of the problem (second figure), numerical solution (third figure), and pointwise error (fourth figure).}
	\label{fig:test_case_5.4.2}
\end{figure}
We additionally perform a pathological test on the domain $\Omega = (0,1)^2$, using an agglomerated mesh composed of 128 elements. We fix the polynomial degree equal to 1 on a single element, and 
set~$p = 2$ on all the remaining ones (see Figure~\ref{fig:test_case_5.4.2}, first panel).
The numerical solution is reported in Figure~\ref{fig:test_case_5.4.2}  (third panel), and shows a good agreement with the exact solution. In the fourth panel, we display the pointwise error between the numerical and exact solutions, which show that the scheme retains good approximation properties even in this unfavorable configuration, where no degradation is observed close to the element with~$p = 1$. This result provides additional evidence that Assumption~\ref{asm:alpha_F} does not compromise the stability or accuracy of the method, even in the presence of pathological polynomial distributions.
\subsection{$p$-convergence for a singular solution}
\label{sec:p-convergence}
To assess whether %
the suboptimal convergence in~$p$ predicted by the \textit{a priori} error estimate in Theorem~\ref{thm:a-priori-estimate} is observed practice, we consider the following benchmark test case from~\cite[Ex. 3.3]{georgoulis_suboptimality_2010}. 
We take the model problem on the domain~$\Omega = (-1,1)\times(0,2)$ with source term and nonhomogeneous Dirichlet boundary conditions chosen so that the exact solution is given by
\begin{equation*}
u(x,y) = (x^2+y^2)^{3/2}.
\end{equation*}
As described in~\cite[Ex. 3.3]{georgoulis_suboptimality_2010}, this %
solution satisfies~$u \in H^{4-\varepsilon}(\Omega) \setminus H^4(\Omega)$ for any $\varepsilon>0$, and has a pointwise singularity at~$(0, 0) \in \partial \Omega$.
Therefore, Theorem~\ref{thm:a-priori-estimate} predicts a suboptimal convergence rate of~$\mathcal{O}(p^{-2})$.
\par
We test the convergence using an unstructured simplicial mesh with 26 triangles and a Voronoi-type polygonal mesh with 10 elements, both chosen so that the point singularity lies on a boundary edge of~$\Omegah$. %
Fixing the mesh, we vary the uniform polynomial degree $p = 1,\dots,12$ and compute the errors in the corresponding %
energy norms for 
the CDG method in Section~\ref{sec:CDG} and the IPDG method from~\cite{georgoulis_suboptimality_2010}. In Figures~\ref{fig:Errors_p_tria}--\ref{fig:Errors_p_voro}, we report the results, which show an optimal decay~$\mathcal{O}(p^{-3})$ for the CDG method, and the expected suboptimal convergence~$\mathcal{O}(p^{-5/2})$ for the IPDG method (consistent with the results in \cite[\S3]{georgoulis_suboptimality_2010}). 
This suggests that the error estimate in Theorem~\ref{thm:a-priori-estimate} may be not sharp in the power of~$p$, which further motivates the CDG method and suggests that a different theoretical approach should be investigated to establish the observed optimal~$p$-convergence. 
\begin{figure}[t!]
    \begin{subfigure}[b]{0.45\textwidth}
          \resizebox{\textwidth}{!}{\input{Errors_p_Tria.tex}}
          \caption{Simplicial meh}
        \label{fig:Errors_p_tria}
    \end{subfigure}
    \hspace{0.1in}
    \begin{subfigure}[b]{0.45\textwidth}
          \resizebox{\textwidth}{!}{\input{Errors_p_Voro.tex}}
          \caption{Voronoi mesh}
        \label{fig:Errors_p_voro}
    \end{subfigure}    
     \caption{Comparison of the $p$-convergence for the CDG and the IPDG methods in the energy norm for (a) a simplicial mesh, and (b) a Voronoi mesh.
     }
\end{figure}
\section{Concluding remarks}
\label{sec:conclusion}
In this work, we have presented theoretical and computational aspects of the $hp$-version of the CDG method on polytopal meshes. In particular, we proved the well-posedness of the scheme and derived \emph{a priori} error estimates. In addition, we introduced efficient strategies to implement %
the CDG, LDG, and BR2 methods %
within a common framework. The numerical experiments validated the theory and discussed the advantages of the CDG method in terms of computational cost and stencil compactness, making it particularly appealing for applications involving polytopal meshes. Although an implementation for 3D polytopal meshes is not yet available, we expect these computational advantages to be even more significant in three dimensions.
Possible future research directions include the extension to more complex problems, the study of iterative solvers and preconditioners, and the development of~$hp$-adaptivity driven by suitable \emph{a posteriori} error estimators.
\section*{Acknowledgments}
This research was partially supported by the European Union (ERC Synergy, NEMESIS, project number 101115663). Views and opinions expressed are, however, those of the authors only and do not necessarily reflect those of the EU or the ERC Executive Agency.
This research was funded in part by the Austrian Science Fund (FWF) project 10.55776/F65. The authors are members of the INdAM-GNCS group. The present research is part of the activities of Dipartimento di Eccellenza 2023-2027 (Dipartimento di Matematica, Politecnico di Milano).

\bibliographystyle{plain}
\bibliography{references.bib}
\end{document}

%% file: Errors_h_L2.tex
\definecolor{mycolor2}{rgb}{0.00000,1.00000,1.00000}%
\begin{tikzpicture}
\begin{axis}[%
width=3in,
height=2in,
at={(2.6in,1.099in)},
scale only axis,
xmode=log,
xmin=0.017228947735558,
xmax=0.185101829426493,
xminorticks=true,
xlabel = {$h$ [-]},
ylabel = {$\|u_h-u\|_{L^2(\Omega)}$},
xticklabel={\pgfmathparse{exp(\tick)}\pgfmathprintnumber{\pgfmathresult}},
x tick label style={
/pgf/number format/.cd, fixed, fixed zerofill,
precision=2},
ymode=log,
ymin=1e-11,
ymax=1e-1,
yminorticks=true,
axis background/.style={fill=white},
title style={font=\bfseries},
xmajorgrids,
xminorgrids,
ymajorgrids,
yminorgrids,
legend style={at={(0.96,0.40)},legend cell align=left, draw=white!15!black}
]
              
\addplot [color=red, line width=2.0pt, mark=*]
  table[row sep=crcr]{%
0.185101829426493   2.129372925426900e-02 \\
0.134359609879251   1.025754194982700e-02 \\
0.096923683183123   5.104460239877000e-03 \\
0.068774852532456   2.560814959594000e-03 \\
0.048201665231121   1.254719654223000e-03 \\
0.033765166440742   6.221363833215340e-04 \\
0.024216305406150   3.115828467914513e-04 \\
0.017228947735558   1.561381742036755e-04 \\
};
\addlegendentry{$p=1$}

\addplot [color=violet, line width=2.0pt, mark=*]
  table[row sep=crcr]{%
0.185101829426493   1.411477302974000e-03 \\
0.134359609879251   4.883835374558802e-04 \\
0.096923683183123   1.681756573401179e-04 \\
0.068774852532456   5.982207956090846e-05 \\
0.048201665231121   2.123464802329850e-05 \\
0.033765166440742   7.390415642517033e-06 \\
0.024216305406150   2.606980852261043e-06 \\
0.017228947735558   9.122894860779180e-07 \\
};
\addlegendentry{$p=2$}

\addplot [color=teal, line width=2.0pt, mark=*] 
  table[row sep=crcr]{%
0.185101829426493   9.133131050728428e-05 \\
0.134359609879251   2.308256115941455e-05 \\
0.096923683183123   5.832737826695899e-06 \\
0.068774852532456   1.365859318639999e-06 \\
0.048201665231121   3.504399248774234e-07 \\
0.033765166440742   8.487973688809233e-08 \\
0.024216305406150   2.101475907001714e-08 \\
0.017228947735558   5.345633070570796e-09 \\
};
\addlegendentry{$p=3$}

\addplot [color=green, line width=2.0pt, mark=*]
  table[row sep=crcr]{%
0.185101829426493   5.620551478098124e-06 \\
0.134359609879251   9.733484608665353e-07 \\
0.096923683183123   1.613422783428660e-07 \\
0.068774852532456   2.836567423416056e-08 \\
0.048201665231121   4.748715100591094e-09 \\
0.033765166440742   8.337420112081393e-10 \\
0.024216305406150   1.453710742805605e-10 \\
0.017228947735558   2.499088964945406e-11 \\
};
\addlegendentry{$p=4$}

\node[right, align=left, text=black, font=\footnotesize]
at (axis cs:0.0405,3.5e-4) {$2$};

\addplot [color=black, line width=1.5pt]
  table[row sep=crcr]{%
0.040   5.00e-4\\
0.030   2.81e-4\\
0.040   2.81e-4\\
0.040   5.00e-4\\
};

\node[right, align=left, text=black, font=\footnotesize]
at (axis cs:0.0405,3.5e-6) {$3$};

\addplot [color=black, line width=1.5pt]
  table[row sep=crcr]{%
0.040   5.00e-6\\
0.030   2.11e-6\\
0.040   2.11e-6\\
0.040   5.00e-6\\
};

\node[right, align=left, text=black, font=\footnotesize]
at (axis cs:0.0405,4.5e-8) {$4$};

\addplot [color=black, line width=1.5pt]
  table[row sep=crcr]{%
0.040   8.00e-8\\
0.030   2.53e-8\\
0.040   2.53e-8\\
0.040   8.00e-8\\
};
\node[right, align=left, text=black, font=\footnotesize]
at (axis cs:0.0405,4.5e-10) {$5$};

\addplot [color=black, line width=1.5pt]
  table[row sep=crcr]{%
0.040   8.00e-10\\
0.030   1.89e-10\\
0.040   1.89e-10\\
0.040   8.00e-10\\
};

\end{axis}
\end{tikzpicture}%

%% file: Errors_h_DG.tex
\definecolor{mycolor2}{rgb}{0.00000,1.00000,1.00000}%
\begin{tikzpicture}
\begin{axis}[%
width=3in,
height=2in,
at={(2.6in,1.099in)},
scale only axis,
xmode=log,
xmin=0.017228947735558,
xmax=0.185101829426493,
xminorticks=true,
xlabel = {$h$ [-]},
ylabel = {$|\!|\!|u_h-u|\!|\!|_\mathrm{CDG}$},
xticklabel={\pgfmathparse{exp(\tick)}\pgfmathprintnumber{\pgfmathresult}},
x tick label style={
/pgf/number format/.cd, fixed, fixed zerofill,
precision=2},
ymode=log,
ymin=1e-8,
ymax=1e+0,
yminorticks=true,
axis background/.style={fill=white},
title style={font=\bfseries},
xmajorgrids,
xminorgrids,
ymajorgrids,
yminorgrids,
legend style={at={(0.96,0.35)},legend cell align=left, draw=white!15!black}
]
              
\addplot [color=red, line width=2.0pt, mark=*]
  table[row sep=crcr]{%
0.185101829426493   0.798073971266559 \\
0.134359609879251   0.567330300289801 \\
0.096923683183123   0.380392426794669 \\
0.068774852532456   0.265421587253200 \\
0.048201665231121   0.189184115497626 \\
0.033765166440742   0.133479104256488 \\
0.024216305406150   0.094927594798161 \\
0.017228947735558   0.066916393933002 \\
};

\addplot [color=violet, line width=2.0pt, mark=*]
  table[row sep=crcr]{%
0.185101829426493   0.177061158946983 \\
0.134359609879251   0.092592060318450 \\
0.096923683183123   0.045878125346317 \\
0.068774852532456   0.022935879090370 \\
0.048201665231121   0.011773496639597 \\
0.033765166440742   0.005813051631681 \\
0.024216305406150   0.002899601307539 \\
0.017228947735558   0.001443806947121 \\
};

\addplot [color=teal, line width=2.0pt, mark=*] 
  table[row sep=crcr]{%
0.185101829426493   1.515452660515000e-02 \\
0.134359609879251   5.108443376894000e-03 \\
0.096923683183123   1.724567801878000e-03 \\
0.068774852532456   5.777559016212801e-04 \\
0.048201665231121   2.023574333953484e-04 \\
0.033765166440742   6.954828789362314e-05 \\
0.024216305406150   2.417641182355588e-05 \\
0.017228947735558   8.601302611166890e-06 \\
};

\addplot [color=green, line width=2.0pt, mark=*]
  table[row sep=crcr]{%
0.185101829426493   1.406280898653000e-03 \\
0.134359609879251   3.573907157187062e-04 \\
0.096923683183123   8.962435038196228e-05 \\
0.068774852532456   2.141453372402536e-05 \\
0.048201665231121   5.468693579045583e-06 \\
0.033765166440742   1.332675068687397e-06 \\
0.024216305406150   3.334240665045801e-07 \\
0.017228947735558   8.229837231791981e-08 \\
};

\node[right, align=left, text=black, font=\footnotesize]
at (axis cs:0.0405,7.0e-2) {$1$};

\addplot [color=black, line width=1.5pt]
  table[row sep=crcr]{%
0.040   8.00e-2\\
0.030   6.00e-2\\
0.040   6.00e-2\\
0.040   8.00e-2\\
};

\node[right, align=left, text=black, font=\footnotesize]
at (axis cs:0.0405,3.5e-3) {$2$};

\addplot [color=black, line width=1.5pt]
  table[row sep=crcr]{%
0.040   5.00e-3\\
0.030   2.81e-3\\
0.040   2.81e-3\\
0.040   5.00e-3\\
};

\node[right, align=left, text=black, font=\footnotesize]
at (axis cs:0.0405,3.5e-5) {$3$};

\addplot [color=black, line width=1.5pt]
  table[row sep=crcr]{%
0.040   5.00e-5\\
0.030   2.11e-5\\
0.040   2.11e-5\\
0.040   5.00e-5\\
};

\node[right, align=left, text=black, font=\footnotesize]
at (axis cs:0.0405,4.5e-7) {$4$};

\addplot [color=black, line width=1.5pt]
  table[row sep=crcr]{%
0.040   8.00e-7\\
0.030   2.53e-7\\
0.040   2.53e-7\\
0.040   8.00e-7\\
};

\end{axis}
\end{tikzpicture}%

%% file: Errors_p.tex
\begin{tikzpicture}
\begin{axis}[%
width=3in,
height=2in,
at={(2.6in,1.099in)},
scale only axis,
xmin=15,
xmax=65,
xminorticks=true,
xlabel = {$\sqrt{N_\mathrm{DOFs}}$},
ylabel = {Error},
xtick = {15,25,35,45,55,65},
x tick label style={
/pgf/number format/.cd, fixed, fixed zerofill,
precision=0},
ymode=log,
ymin=1e-11,
ymax=1e+01,
yminorticks=true,
axis background/.style={fill=white},
title style={font=\bfseries},
xmajorgrids,
xminorgrids,
ymajorgrids,
yminorgrids,
legend style={at={(0.40,0.25)},legend cell align=left, draw=white!15!black}
]
              
\addplot [color=blue, line width=2.0pt, mark=square*]
  table[row sep=crcr]{%
1.732050807568877e+01   8.146627209436224e-01 \\
2.449489742783178e+01   1.909882455731717e-01 \\
3.162277660168379e+01   1.615501362230014e-02 \\
3.872983346207417e+01   1.508272806529846e-03 \\
4.582575694955840e+01   7.874575113382561e-05 \\
5.291502622129181e+01   4.608884914039619e-06 \\
6.000000000000000e+01   1.805279608779096e-07 \\
};
\addlegendentry{CDG-Error}

\addplot [color=orange, line width=2.0pt, mark=*]
  table[row sep=crcr]{%
1.732050807568877e+01   2.292589086788441e-02 \\
2.449489742783178e+01   1.467625541318020e-03 \\
3.162277660168379e+01   9.515693681257038e-05 \\
3.872983346207417e+01   5.890611183642446e-06 \\
4.582575694955840e+01   2.799066417954190e-07 \\
5.291502622129181e+01   1.150845564347450e-08 \\
6.000000000000000e+01   4.065194484857921e-10 \\
};
\addlegendentry{$L^2$-Error}

\end{axis}
\end{tikzpicture}%

%% file: AssemblyTimesp2.tex
\definecolor{mycolor2}{rgb}{0.00000,1.00000,1.00000}%
\begin{tikzpicture}
\begin{axis}[%
width=2in,
height=2in,
at={(2.6in,1.099in)},
scale only axis,
xmode=log,
xmin=0.008636811860669,
xmax=0.033765166440742,
xminorticks=true,
xlabel = {$h$ [-]},
ylabel = {Times [s]},
xticklabel={\pgfmathparse{exp(\tick)}\pgfmathprintnumber{\pgfmathresult}},
x tick label style={
/pgf/number format/.cd, fixed, fixed zerofill,
precision=2},
ymode=log,
ymin=1,
ymax=1000,
yminorticks=true,
axis background/.style={fill=white},
title style={font=\bfseries},
xmajorgrids,
xminorgrids,
ymajorgrids,
yminorgrids,
legend style={at={(0.96,0.95)},legend cell align=left, draw=white!15!black}
]

\addplot [only marks, color=orange, line width=0.5pt, mark=square*,  mark size=3.0pt, mark options={fill=orange,draw=black}]
  table[row sep=crcr]{%
0.185101829426493   0.102795 \\
0.134359609879251   0.155045 \\
0.096923683183123   0.316491 \\
0.068774852532456   0.670189 \\
0.048201665231121   1.488096 \\
0.033765166440742   3.571106 \\
0.024216305406150   9.115569 \\
0.017228947735558   26.92537 \\
0.012388797881768   88.86353 \\
0.008636811860669   480.7655 \\
};
\addlegendentry{$\LDGw$}

\addplot [only marks, color=magenta, line width=0.5pt, mark=triangle*,  mark size=4.0pt, mark options={fill=magenta,draw=black}]
  table[row sep=crcr]{%
0.185101829426493   0.108754 \\
0.134359609879251   0.172023 \\
0.096923683183123   0.352511 \\
0.068774852532456   0.752717 \\
0.048201665231121   1.684777 \\
0.033765166440742   4.099274 \\
0.024216305406150   10.94061 \\
0.017228947735558   33.05750 \\
0.012388797881768   109.4380 \\
0.008636811860669   538.5046 \\
};
\addlegendentry{$\LDGf$}

\addplot [only marks, color=cyan, line width=0.5pt, mark=*,  mark size=4.0pt, mark options={fill=cyan,draw=black}]
  table[row sep=crcr]{%
0.185101829426493   0.084339 \\
0.134359609879251   0.124722 \\
0.096923683183123   0.253036 \\
0.068774852532456   0.525580 \\
0.048201665231121   1.137922 \\
0.033765166440742   2.615263 \\
0.024216305406150   6.221774 \\
0.017228947735558   16.92096 \\
0.012388797881768   55.78499 \\
0.008636811860669   251.2783 \\
};
\addlegendentry{CDG}

\addplot [only marks, color=green, line width=0.5pt, mark=diamond*,  mark size=4.0pt, mark options={fill=green,draw=black}]
  table[row sep=crcr]{%
0.033765166440742   3.445701 \\
0.024216305406150   8.589941 \\
0.017228947735558   24.51255 \\
0.012388797881768   80.34138 \\
0.008636811860669   331.9446 \\
};
\addlegendentry{BR2}
\end{axis}
\end{tikzpicture}%

%% file: AssemblyTimesp4.tex
\definecolor{mycolor2}{rgb}{0.00000,1.00000,1.00000}%
\begin{tikzpicture}
\begin{axis}[%
width=2in,
height=2in,
at={(2.6in,1.099in)},
scale only axis,
xmode=log,
xmin=0.008636811860669,
xmax=0.033765166440742,
xminorticks=true,
xlabel = {$h$ [-]},
ylabel = {Times [s]},
xticklabel={\pgfmathparse{exp(\tick)}\pgfmathprintnumber{\pgfmathresult}},
x tick label style={
/pgf/number format/.cd, fixed, fixed zerofill,
precision=2},
ymode=log,
ymin=1,
ymax=1000,
yminorticks=true,
axis background/.style={fill=white},
title style={font=\bfseries},
xmajorgrids,
xminorgrids,
ymajorgrids,
yminorgrids,
legend style={at={(0.96,0.95)},legend cell align=left, draw=white!15!black}
]

\addplot [only marks, color=orange, line width=0.5pt, mark=square*,  mark size=3.0pt, mark options={fill=orange,draw=black}]
  table[row sep=crcr]{%
0.185101829426493   0.150645 \\
0.134359609879251   0.254393 \\
0.096923683183123   0.513531 \\
0.068774852532456   1.063392 \\
0.048201665231121   2.334390 \\
0.033765166440742   5.534315 \\
0.024216305406150   13.13893 \\
0.017228947735558   37.10646 \\
0.012388797881768   112.2479 \\
0.008636811860669   524.6346 \\
};

\addplot [only marks, color=magenta, line width=0.5pt, mark=triangle*,  mark size=4.0pt, mark options={fill=magenta,draw=black}]
  table[row sep=crcr]{%
0.185101829426493   0.189917 \\
0.134359609879251   0.339313 \\
0.096923683183123   0.699474 \\
0.068774852532456   1.503734 \\
0.048201665231121   3.240278 \\
0.033765166440742   7.433748 \\
0.024216305406150   17.54717 \\
0.017228947735558   49.08221 \\
0.012388797881768   143.6343 \\
0.008636811860669   598.9127 \\
};

\addplot [only marks, color=cyan, line width=0.5pt, mark=*,  mark size=4.0pt, mark options={fill=cyan,draw=black}]
  table[row sep=crcr]{%
0.185101829426493   0.108117 \\
0.134359609879251   0.171980 \\
0.096923683183123   0.344628 \\
0.068774852532456   0.715253 \\
0.048201665231121   1.537144 \\
0.033765166440742   3.506396 \\
0.024216305406150   8.347125 \\
0.017228947735558   23.23145 \\
0.012388797881768   70.56941 \\
0.008636811860669   312.0441 \\
};

\addplot [only marks, color=green, line width=0.5pt, mark=diamond*,  mark size=4.0pt, mark options={fill=green,draw=black}]
  table[row sep=crcr]{%
0.048201665231121   2.137951 \\
0.033765166440742   4.756351 \\
0.024216305406150   11.45844 \\
0.017228947735558   32.12761 \\
0.012388797881768   96.53730 \\
0.008636811860669   373.6126 \\
};
\end{axis}
\end{tikzpicture}%

%% file: SolvingTimesp2.tex
\definecolor{mycolor2}{rgb}{0.00000,1.00000,1.00000}%
\begin{tikzpicture}
\begin{axis}[%
width=2in,
height=2in,
at={(2.6in,1.099in)},
scale only axis,
xmode=log,
xmin=0.008636811860669,
xmax=0.033765166440742,
xminorticks=true,
xlabel = {$h$ [-]},
ylabel = {Times [s]},
xticklabel={\pgfmathparse{exp(\tick)}\pgfmathprintnumber{\pgfmathresult}},
x tick label style={
/pgf/number format/.cd, fixed, fixed zerofill,
precision=2},
ymode=log,
ymin=1e-2,
ymax=1000,
yminorticks=true,
axis background/.style={fill=white},
title style={font=\bfseries},
xmajorgrids,
xminorgrids,
ymajorgrids,
yminorgrids,
legend style={at={(0.50,0.35)},legend cell align=left, draw=white!15!black}
]
    
\addplot [only marks, color=orange, line width=0.5pt, mark=square*,  mark size=3.0pt, mark options={fill=orange,draw=black}]
  table[row sep=crcr]{%
0.185101829426493   0.002641 \\
0.134359609879251   0.005499 \\
0.096923683183123   0.012752 \\
0.068774852532456   0.032970 \\
0.048201665231121   0.083479 \\
0.033765166440742   0.202548 \\
0.024216305406150   0.495542 \\
0.017228947735558   1.188554 \\
0.012388797881768   2.921672 \\
0.008636811860669   8.496603 \\
};

\addplot [only marks, color=magenta, line width=0.5pt, mark=triangle*,  mark size=4.0pt, mark options={fill=magenta,draw=black}]
  table[row sep=crcr]{%
0.185101829426493   0.004727 \\
0.134359609879251   0.011965 \\
0.096923683183123   0.029622 \\
0.068774852532456   0.082327 \\
0.048201665231121   0.210187 \\
0.033765166440742   0.538146 \\
0.024216305406150   1.438232 \\
0.017228947735558   4.046540 \\
0.012388797881768   10.87280 \\
0.008636811860669   33.78916 \\
};
          
\addplot [only marks, color=cyan, line width=0.5pt, mark=*,  mark size=4.0pt, mark options={fill=cyan,draw=black}]
  table[row sep=crcr]{%
0.185101829426493   0.002380 \\
0.134359609879251   0.005084 \\
0.096923683183123   0.011673 \\
0.068774852532456   0.026637 \\
0.048201665231121   0.067379 \\
0.033765166440742   0.162329 \\
0.024216305406150   0.403908 \\
0.017228947735558   0.982518 \\
0.012388797881768   2.423196 \\
0.008636811860669   6.212517 \\
};

\addplot [only marks, color=green, line width=0.5pt, mark=diamond*,  mark size=4.0pt, mark options={fill=green,draw=black}]
  table[row sep=crcr]{%
0.185101829426493   0.002456 \\
0.134359609879251   0.004715 \\
0.096923683183123   0.011417 \\
0.068774852532456   0.026004 \\
0.048201665231121   0.068056 \\
0.033765166440742   0.163103 \\
0.024216305406150   0.405597 \\
0.017228947735558   0.978113 \\
0.012388797881768   2.422893 \\
0.008636811860669   6.247608 \\
};

\end{axis}
\end{tikzpicture}%

%% file: SolvingTimesp4.tex
\definecolor{mycolor2}{rgb}{0.00000,1.00000,1.00000}%
\begin{tikzpicture}
\begin{axis}[%
width=2in,
height=2in,
at={(2.6in,1.099in)},
scale only axis,
xmode=log,
xmin=0.008636811860669,
xmax=0.033765166440742,
xminorticks=true,
xlabel = {$h$ [-]},
ylabel = {Times [s]},
xticklabel={\pgfmathparse{exp(\tick)}\pgfmathprintnumber{\pgfmathresult}},
x tick label style={
/pgf/number format/.cd, fixed, fixed zerofill,
precision=2},
ymode=log,
ymin=1e-2,
ymax=1000,
yminorticks=true,
axis background/.style={fill=white},
title style={font=\bfseries},
xmajorgrids,
xminorgrids,
ymajorgrids,
yminorgrids,
legend style={at={(0.90,0.85)},legend cell align=left, draw=white!15!black}
]
    
\addplot [only marks, color=orange, line width=0.5pt, mark=square*,  mark size=3.0pt, mark options={fill=orange,draw=black}]
  table[row sep=crcr]{%
0.185101829426493   0.014153 \\
0.134359609879251   0.033516 \\
0.096923683183123   0.078652 \\
0.068774852532456   0.217622 \\
0.048201665231121   0.567182 \\
0.033765166440742   1.378693 \\
0.024216305406150   3.735154 \\
0.017228947735558   9.643622 \\
0.012388797881768   27.40377 \\
0.008636811860669   90.05355 \\
};

\addplot [only marks, color=magenta, line width=0.5pt, mark=triangle*,  mark size=4.0pt, mark options={fill=magenta,draw=black}]
  table[row sep=crcr]{%
0.185101829426493   0.031160 \\
0.134359609879251   0.081497 \\
0.096923683183123   0.215524 \\
0.068774852532456   0.627253 \\
0.048201665231121   1.694275 \\
0.033765166440742   4.826582 \\
0.024216305406150   14.85551 \\
0.017228947735558   45.45046 \\
0.012388797881768   128.6702 \\
0.008636811860669   427.7784 \\
};
          
\addplot [only marks, color=cyan, line width=0.5pt, mark=*,  mark size=4.0pt, mark options={fill=cyan,draw=black}]
  table[row sep=crcr]{%
0.185101829426493   0.013333 \\
0.134359609879251   0.032290 \\
0.096923683183123   0.071576 \\
0.068774852532456   0.183235 \\
0.048201665231121   0.460417 \\
0.033765166440742   1.126231 \\
0.024216305406150   3.065965 \\
0.017228947735558   8.156845 \\
0.012388797881768   22.48380 \\
0.008636811860669   63.28253 \\
};

\addplot [only marks, color=green, line width=0.5pt, mark=diamond*,  mark size=4.0pt, mark options={fill=green,draw=black}]
  table[row sep=crcr]{%
0.185101829426493   0.012699 \\
0.134359609879251   0.032214 \\
0.096923683183123   0.076734 \\
0.068774852532456   0.186529 \\
0.048201665231121   0.456001 \\
0.033765166440742   1.131390 \\
0.024216305406150   3.067231 \\
0.017228947735558   8.185233 \\
0.012388797881768   22.45852 \\
0.008636811860669   63.20816 \\
};

\end{axis}
\end{tikzpicture}%

%% file: MeshComparison_L2_XF.tex
\begin{tikzpicture}
\begin{axis}[%
width=3.00in,
height=2.00in,
at={(2.6in,1.099in)},
scale only axis,
xmode=log,
xmin=0.014142135659262,
xmax=0.198983865061495,
xminorticks=true,
xlabel = {$h$ [-]},
ylabel = {$\|u_h-u\|_{L^2(\Omega)}$},
xticklabel={\pgfmathparse{exp(\tick)}\pgfmathprintnumber{\pgfmathresult}},
x tick label style={
/pgf/number format/.cd, fixed, fixed zerofill,
precision=2},
ymode=log,
ymin=1e-11,
ymax=1e-3,
yminorticks=true,
axis background/.style={fill=white},
title style={font=\bfseries},
xmajorgrids,
xminorgrids,
ymajorgrids,
yminorgrids,
legend style={at={(0.96,0.40)},legend cell align=left, draw=white!15!black}
]

\addplot [color=purple, line width=2.0pt, mark=triangle*]
  table[row sep=crcr]{%
0.176776695296637   1.792707545327093e-05 \\
0.117852073006800   2.416812634074811e-06 \\
0.088388347648318   5.791492921094836e-07 \\
0.070710678118655   1.907402996240128e-07 \\
0.056568542494924   6.272606727959129e-08 \\
0.044194173824159   1.830623601793973e-08 \\
0.035355339059327   6.009115571432744e-09 \\
0.028284271247462   1.971560727815801e-09 \\
0.023570697444073   7.929409917118179e-10 \\
0.018857123640683   2.600186789334765e-10 \\
};

\addplot [color=cyan, line width=2.0pt, mark=square*]
  table[row sep=crcr]{%
0.141421356592613   9.083743710130632e-06 \\
0.070710678296307   2.932599377833020e-07 \\
0.047140452197538   3.915531774071636e-08 \\
0.035355339148153   9.281455655213969e-09 \\
0.028284271318523   3.050482228153148e-09 \\
0.023570226098769   1.232779362429422e-09 \\
0.017677669574077   2.899799298453101e-10 \\
0.014142135659262   9.569855836389577e-11 \\
};

\addplot [color=teal, line width=2.0pt, mark=pentagon*]
  table[row sep=crcr]{%
0.185101829426493   0.000005890611184 \\
0.134359609879251   0.000001019720338 \\
0.096923683183123   0.000000167663469 \\
0.068774852532456   0.000000029634671 \\
0.048201665231121   0.000000004938109 \\
0.033765166440742   0.000000000868204 \\
0.024216305406150   0.000000000151183 \\
0.017228947735558   0.000000000025959 \\
};

\addplot [color=orange, line width=2.0pt, mark=*]
  table[row sep=crcr]{%
0.198983865061495   0.000015217768239 \\
0.102476025836552   0.000000457692829 \\
0.076298886200117   0.000000075777047 \\
0.056914832977051   0.000000012935313 \\
0.043328680110681   0.000000003063512 \\
0.027848260524380   0.000000000339825 \\
};

\addplot [color=black, dotted, line width=1.5pt]
  table[row sep=crcr]{%
0.198983865061495    2.335934745601948e-04\\
0.014142135659262    4.643348069415500e-10\\
};

\end{axis}
\end{tikzpicture}%

%% file: MeshComparison_L2_X1.tex
\begin{tikzpicture}
\begin{axis}[%
width=3.00in,
height=2.00in,
at={(2.6in,1.099in)},
scale only axis,
xmode=log,
xmin=0.014142135659262,
xmax=0.198983865061495,
xminorticks=true,
xlabel = {$h$ [-]},
ylabel = {$\|u_h-u\|_{L^2(\Omega)}$},
xticklabel={\pgfmathparse{exp(\tick)}\pgfmathprintnumber{\pgfmathresult}},
x tick label style={
/pgf/number format/.cd, fixed, fixed zerofill,
precision=2},
ymode=log,
ymin=1e-11,
ymax=1e-3,
yminorticks=true,
axis background/.style={fill=white},
title style={font=\bfseries},
xmajorgrids,
xminorgrids,
ymajorgrids,
yminorgrids,
legend style={at={(0.96,0.40)},legend cell align=left, draw=white!15!black}
]

\addplot [color=purple, line width=2.0pt, mark=triangle*]
  table[row sep=crcr]{%
0.176776695296637   1.337725588696407e-05 \\
0.117852073006800   1.809131103911243e-06 \\
0.088388347648318   4.343103668246878e-07 \\
0.070710678118655   1.432125591938426e-07 \\
0.056568542494924   4.714651112548427e-08 \\
0.044194173824159   1.377330393981996e-08 \\
0.035355339059327   4.524622912766746e-09 \\
0.028284271247462   1.485459848339047e-09 \\
0.023570697444073   5.977022993706005e-10 \\
0.018857123640683   1.960852270818176e-10 \\
};

\addplot [color=cyan, line width=2.0pt, mark=square*]
  table[row sep=crcr]{%
0.141421356592613   7.368525589971084e-06 \\
0.070710678296307   2.361725451143919e-07 \\
0.047140452197538   3.100784332711394e-08 \\
0.035355339148153   7.376017733900127e-09 \\
0.028284271318523   2.417388384111082e-09 \\
0.023570226098769   9.735043690457316e-10 \\
0.017677669574077   2.399799298453101e-10 \\
0.014142135659262   7.578085778946675e-11 \\
};

\addplot [color=teal, line width=2.0pt, mark=pentagon*]
  table[row sep=crcr]{%
0.185101829426493   1.677369734981965e-04 \\
0.134359609879251   4.009842277270836e-06 \\
0.096923683183123   7.081149094582682e-07 \\
0.068774852532456   5.504171252788501e-06 \\
0.048201665231121   7.958096264750936e-08 \\
0.033765166440742   1.699749005508952e-08 \\
0.024216305406150   2.866110326151568e-09 \\
0.017228947735558   5.749423263312488e-10 \\
};

\addplot [color=orange, line width=2.0pt, mark=*]
  table[row sep=crcr]{%
0.198983865061495   0.000488807547108 \\
0.102476025836552   0.000015326506935 \\
0.076298886200117   0.000003762197616 \\
0.056914832977051   0.000000384297190 \\
0.043328680110681   0.000002642103920 \\
0.027848260524380   0.000000022710603 \\
};

\addplot [color=black, dotted, line width=1.5pt]
  table[row sep=crcr]{%
0.198983865061495    2.335934745601948e-05\\
0.014142135659262    4.643348069415500e-11\\
};

\end{axis}
\end{tikzpicture}%

%% file: MeshComparison_L2_X01.tex
\begin{tikzpicture}
\begin{axis}[%
width=3.00in,
height=2.00in,
at={(2.6in,1.099in)},
scale only axis,
xmode=log,
xmin=0.014142135659262,
xmax=0.198983865061495,
xminorticks=true,
xlabel = {$h$ [-]},
ylabel = {$\|u_h-u\|_{L^2(\Omega)}$},
xticklabel={\pgfmathparse{exp(\tick)}\pgfmathprintnumber{\pgfmathresult}},
x tick label style={
/pgf/number format/.cd, fixed, fixed zerofill,
precision=2},
ymode=log,
ymin=1e-11,
ymax=1e-3,
yminorticks=true,
axis background/.style={fill=white},
title style={font=\bfseries},
xmajorgrids,
xminorgrids,
ymajorgrids,
yminorgrids,
legend style={at={(0.96,0.50)},legend cell align=left, draw=white!15!black}
]

\addplot [color=purple, line width=2.0pt, mark=triangle*]
  table[row sep=crcr]{%
0.176776695296637   5.996221880010621e-04 \\
0.117852073006800   6.665698641465187e-05 \\
0.088388347648318   1.405446578558408e-05 \\
0.070710678118655   4.293893192925027e-06 \\
0.056568542494924   1.301597695590333e-06 \\
0.044194173824159   3.424068671646818e-07 \\
0.035355339059327   1.013988536231059e-07 \\
0.028284271247462   2.997399864215090e-08 \\
0.023570697444073   1.111691167799140e-08 \\
0.018857123640683   3.333770308156580e-09 \\
};
\addlegendentry{Triangular}

\addplot [color=cyan, line width=2.0pt, mark=square*]
  table[row sep=crcr]{%
0.141421356592613   1.198731187425877e-04 \\
0.070710678296307   1.854264564126671e-06 \\
0.047140452197538   1.004062339888241e-06 \\
0.035355339148153   1.123673275524957e-07 \\
0.028284271318523   5.036991538753130e-07 \\
0.023570226098769   6.838713050312869e-08 \\
0.017677669574077   1.405877428615309e-08 \\
0.014142135659262   1.955676275329804e-09 \\
};
\addlegendentry{Cartesian}

\addplot [color=teal, line width=2.0pt, mark=pentagon*]
  table[row sep=crcr]{%
0.185101829426493   2.586520571004273e-05 \\
0.134359609879251   9.383992932013325e-06 \\
0.096923683183123   2.177047846688053e-06 \\
0.068774852532456   8.502553910124020e-07 \\
0.048201665231121   1.281672597437476e-07 \\
0.033765166440742   1.107036277459412e-06 \\
0.024216305406150   7.434971480735172e-09 \\
0.017228947735558   1.318219467462726e-08 \\
};
\addlegendentry{Voronoi}

\addplot [color=orange, line width=2.0pt, mark=*]
  table[row sep=crcr]{%
0.198983865061495   0.000154188096732 \\
0.102476025836552   0.000018215397107 \\
0.076298886200117   0.000001256144614 \\
0.056914832977051   0.000001708566425 \\
0.043328680110681   0.000004258144356 \\
0.027848260524380   0.000000038529322 \\
};
\addlegendentry{Agglomerated}

\addplot [color=black, dotted, line width=1.5pt]
  table[row sep=crcr]{%
0.198983865061495    2.335934745601948e-05\\
0.014142135659262    4.643348069415500e-11\\
};
\addlegendentry{$\sim h^5$}

\end{axis}
\end{tikzpicture}%

%% file: MeshComparison_DG_XF.tex
\begin{tikzpicture}
\begin{axis}[%
width=3.00in,
height=2.00in,
at={(2.6in,1.099in)},
scale only axis,
xmode=log,
xmin=0.014142135659262,
xmax=0.198983865061495,
xminorticks=true,
xlabel = {$h$ [-]},
ylabel = {$\Tnorm{u_h-u}{\mathrm{CDG}}$},
xticklabel={\pgfmathparse{exp(\tick)}\pgfmathprintnumber{\pgfmathresult}},
x tick label style={
/pgf/number format/.cd, fixed, fixed zerofill,
precision=2},
ymode=log,
ymin=5e-8,
ymax=1e-1,
yminorticks=true,
axis background/.style={fill=white},
title style={font=\bfseries},
xmajorgrids,
xminorgrids,
ymajorgrids,
yminorgrids,
legend style={at={(0.96,0.40)},legend cell align=left, draw=white!15!black}
]

\addplot [color=purple, line width=2.0pt, mark=triangle*]
  table[row sep=crcr]{%
0.176776695296637   4.909935198048000e-03 \\
0.117852073006800   9.983597048342445e-04 \\
0.088388347648318   3.195670761772958e-04 \\
0.070710678118655   1.316745108905092e-04 \\
0.056568542494924   5.416010910682530e-05 \\
0.044194173824159   2.024168704163343e-05 \\
0.035355339059327   8.308162011151621e-06 \\
0.028284271247462   3.408150244558654e-06 \\
0.023570697444073   1.645119660224887e-06 \\
0.018857123640683   6.744278441004053e-07 \\
};

\addplot [color=cyan, line width=2.0pt, mark=square*]
  table[row sep=crcr]{%
0.141421356592613   1.939980238314000e-03 \\
0.070710678296307   1.232782766545326e-04 \\
0.047140452197538   2.475680131083004e-05 \\
0.035355339148153   7.786963679392058e-06 \\
0.028284271318523   3.197029060014680e-06 \\
0.023570226098769   1.555590811760477e-06 \\
0.017677669574077   4.840553347478324e-07 \\
0.014142135659262   2.005888121411444e-07 \\
};

\addplot [color=teal, line width=2.0pt, mark=pentagon*]
  table[row sep=crcr]{%
0.185101829426493   0.001574516935129 \\
0.134359609879251   0.000398910553734 \\
0.096923683183123   0.000099525195044 \\
0.068774852532456   0.000023807514432 \\
0.048201665231121   0.000006062512964 \\
0.033765166440742   0.000001477182480 \\
0.024216305406150   0.000000368842515 \\
0.017228947735558   0.000000091000063 \\
};

\addplot [color=orange, line width=2.0pt, mark=*]
  table[row sep=crcr]{%
0.198983865061495   0.005403319955988 \\
0.102476025836552   0.000316681033653 \\
0.076298886200117   0.000074024130329 \\ 
0.056914832977051   0.000018408418128 \\
0.043328680110681   0.000006152491736 \\
0.027848260524380   0.000001632359677 \\
};

\addplot [color=black, dotted, line width=1.5pt]
  table[row sep=crcr]{%
0.198983865061495   2.523945606415074e-02\\
0.014142135659262   8.000000080397586e-07\\
};

\end{axis}
\end{tikzpicture}%

%% file: MeshComparison_DG_X1.tex
\begin{tikzpicture}
\begin{axis}[%
width=3.00in,
height=2.00in,
at={(2.6in,1.099in)},
scale only axis,
xmode=log,
xmin=0.014142135659262,
xmax=0.198983865061495,
xminorticks=true,
xlabel = {$h$ [-]},
ylabel = {$\Tnorm{u_h-u}{\mathrm{CDG}}$},
xticklabel={\pgfmathparse{exp(\tick)}\pgfmathprintnumber{\pgfmathresult}},
x tick label style={
/pgf/number format/.cd, fixed, fixed zerofill,
precision=2},
ymode=log,
ymin=5e-8,
ymax=1e-1,
yminorticks=true,
axis background/.style={fill=white},
title style={font=\bfseries},
xmajorgrids,
xminorgrids,
ymajorgrids,
yminorgrids,
legend style={at={(0.96,0.40)},legend cell align=left, draw=white!15!black}
]

\addplot [color=purple, line width=2.0pt, mark=triangle*]
  table[row sep=crcr]{%
0.176776695296637   3.355432793812000e-03 \\
0.117852073006800   6.724847715731633e-04 \\
0.088388347648318   2.137463634603247e-04 \\
0.070710678118655   8.772507503586256e-05 \\
0.056568542494924   3.597707410184936e-05 \\
0.044194173824159   1.341414660229592e-05 \\
0.035355339059327   5.497207600383899e-06 \\
0.028284271247462   2.252420852186843e-06 \\
0.023570697444073   1.086457434376704e-06 \\
0.018857123640683   4.450896908596112e-07 \\
};

\addplot [color=cyan, line width=2.0pt, mark=square*]
  table[row sep=crcr]{%
0.141421356592613   1.310263654805000e-03 \\
0.070710678296307   8.114733963117837e-05 \\
0.047140452197538   1.598664936648679e-05 \\
0.035355339148153   5.111560918802839e-06 \\
0.028284271318523   2.094890649377590e-06 \\
0.023570226098769   1.010006430206892e-06 \\
0.017677669574077   3.187218505761671e-07 \\
0.014142135659262   1.307593539634945e-07 \\
};

\addplot [color=teal, line width=2.0pt, mark=pentagon*]
  table[row sep=crcr]{%
0.185101829426493   3.046750448143400e-02 \\
0.134359609879251   1.059742057217000e-03 \\
0.096923683183123   2.781292654160828e-04 \\
0.068774852532456   2.754679625912000e-03 \\
0.048201665231121   6.537299508074659e-05 \\
0.033765166440742   1.708710531812959e-05 \\
0.024216305406150   4.180981510349747e-06 \\
0.017228947735558   1.200116171861723e-06 \\
};

\addplot [color=orange, line width=2.0pt, mark=*]
  table[row sep=crcr]{%
0.198983865061495   0.078772162629341 \\
0.102476025836552   0.005613884003047 \\
0.076298886200117   0.002181440239716 \\
0.056914832977051   0.000328009218582 \\
0.043328680110681   0.003198722935058 \\
0.027848260524380   0.000042742864614 \\
};

\addplot [color=black, dotted, line width=1.5pt]
  table[row sep=crcr]{%
0.198983865061495   2.523945606415074e-03\\
0.014142135659262   8.000000080397586e-08\\
};

\end{axis}
\end{tikzpicture}%

%% file: MeshComparison_DG_X01.tex
\begin{tikzpicture}
\begin{axis}[%
width=3.00in,
height=2.00in,
at={(2.6in,1.099in)},
scale only axis,
xmode=log,
xmin=0.014142135659262,
xmax=0.198983865061495,
xminorticks=true,
xlabel = {$h$ [-]},
ylabel = {$\Tnorm{u_h-u}{\mathrm{CDG}}$},
xticklabel={\pgfmathparse{exp(\tick)}\pgfmathprintnumber{\pgfmathresult}},
x tick label style={
/pgf/number format/.cd, fixed, fixed zerofill,
precision=2},
ymode=log,
ymin=5e-8,
ymax=1e-1,
yminorticks=true,
axis background/.style={fill=white},
title style={font=\bfseries},
xmajorgrids,
xminorgrids,
ymajorgrids,
yminorgrids,
legend style={at={(0.96,0.50)},legend cell align=left, draw=white!15!black}
]

\addplot [color=purple, line width=2.0pt, mark=triangle*]
  table[row sep=crcr]{%
0.176776695296637   6.452882137367800e-02 \\
0.117852073006800   1.073044012446400e-02 \\
0.088388347648318   3.000194097052000e-03 \\
0.070710678118655   1.143508912688000e-03 \\
0.056568542494924   4.336215570652385e-04 \\
0.044194173824159   1.466744967905345e-04 \\
0.035355339059327   5.470405749813467e-05 \\
0.028284271247462   2.042575166528725e-05 \\
0.023570697444073   9.183599419539226e-06 \\
0.018857123640683   3.489190642111527e-06 \\
};
\addlegendentry{Triangular}

\addplot [color=cyan, line width=2.0pt, mark=square*]
  table[row sep=crcr]{%
0.141421356592613   1.271772916244100e-02 \\
0.070710678296307   3.990177002550811e-04 \\
0.047140452197538   3.238971495370315e-04 \\
0.035355339148153   4.800685236158226e-05 \\
0.028284271318523   2.694258137963442e-04 \\
0.023570226098769   4.408257538989427e-05 \\
0.017677669574077   1.207610809332528e-05 \\
0.014142135659262   2.102093136606353e-06 \\
};
\addlegendentry{Cartesian}

\addplot [color=teal, line width=2.0pt, mark=pentagon*]
  table[row sep=crcr]{%
0.185101829426493   2.879494422400000e-03 \\
0.134359609879251   1.458914071551000e-03 \\
0.096923683183123   4.806835804754187e-04 \\
0.068774852532456   2.616861431703884e-04 \\
0.048201665231121   5.532983218203868e-05 \\
0.033765166440742   6.913362262004837e-04 \\
0.024216305406150   6.583278732373679e-06 \\
0.017228947735558   1.661367299665876e-05 \\
};
\addlegendentry{Voronoi}

\addplot [color=orange, line width=2.0pt, mark=*]
  table[row sep=crcr]{%
0.198983865061495   0.017245384514166 \\
0.102476025836552   0.004149973787618 \\
0.076298886200117   0.000418561492022 \\
0.056914832977051   0.000785130299460 \\
0.043328680110681   0.002833434899106 \\
0.027848260524380   0.000040183319922 \\
};
\addlegendentry{Agglomerated}

\addplot [color=black, dotted, line width=1.5pt]
  table[row sep=crcr]{%
0.198983865061495   2.523945606415074e-03\\
0.014142135659262   8.000000080397586e-08\\
};
\addlegendentry{$\sim h^4$}

\end{axis}
\end{tikzpicture}%

%% file: MinimumEigenvalueChi.tex
\begin{tikzpicture}
\begin{groupplot}[group style={group size=3 by 2, vertical sep=0.25cm, horizontal sep=2cm}]

\nextgroupplot[
    width       = 2.5in,
    height      = 1.5in,
    scale only axis,
    xmode       = log,
    xmin        = 0.01,
    xmax        = 100,
    xminorticks = true,
    xticklabel  = {\textcolor{white}{1}},
    ylabel      = {$\lambda_1(\mathcal{A})$},
    ymode       = log,
    ymin        = 0.03,
    ymax        = 0.10,
    ytick       = {0.03, 0.04, 0.05, 0.06, 0.07, 0.08, 0.09, 0.10},
    yminorticks = true,
    yticklabel={\pgfmathparse{exp(\tick)}\pgfmathprintnumber{\pgfmathresult}},
    y tick label style={/pgf/number format/.cd, fixed, fixed zerofill, precision=2},
    axis background/.style = {fill=white},
    title style = {font=\bfseries},
    title       = {Coarse mesh with $p=1$},
    xmajorgrids,
    xminorgrids,
    ymajorgrids,
    yminorgrids,
    legend style={at={(0.50,0.95)},legend cell align=left, draw=white!15!black}
]

\addplot [only marks, color=teal, line width=0.5pt, mark=pentagon*,  mark size=3.0pt, mark options={fill=teal,draw=black}]
  table[row sep=crcr]{%
   1.0000e+02   6.1630e-02 \\
   8.0000e+01   5.9473e-02 \\
   6.4000e+01   5.7706e-02 \\
   5.1200e+01   5.6256e-02 \\
   4.0960e+01   5.5065e-02 \\
   3.2768e+01   5.4088e-02 \\
   2.6214e+01   5.3286e-02 \\
   2.0972e+01   5.2627e-02 \\
   1.6777e+01   5.2087e-02 \\
   1.3422e+01   5.1643e-02 \\
   1.0737e+01   5.1277e-02 \\
   8.5899e+00   5.0975e-02 \\
   6.8719e+00   5.0723e-02 \\
   5.4976e+00   5.0512e-02 \\
   4.3980e+00   5.0330e-02 \\
   3.5184e+00   5.0170e-02 \\
   2.8147e+00   5.0023e-02 \\
   2.2518e+00   4.9882e-02 \\
   1.8014e+00   4.9735e-02 \\
   1.4412e+00   4.9571e-02 \\
   1.1529e+00   4.9369e-02 \\
   9.2234e-01   4.9097e-02 \\
   7.3787e-01   4.8680e-02 \\
   5.9030e-01   4.6938e-02 \\
};
\addlegendentry{Voronoi}

\addplot [only marks, color=cyan, line width=0.5pt, mark=square*,  mark size=3.0pt, mark options={fill=cyan,draw=black}]
  table[row sep=crcr]{%
   1.0000e+02   8.5118e-02 \\
   8.0000e+01   7.7985e-02 \\
   6.4000e+01   7.2273e-02 \\
   5.1200e+01   6.7699e-02 \\
   4.0960e+01   6.4036e-02 \\
   3.2768e+01   6.1102e-02 \\
   2.6214e+01   5.8751e-02 \\
   2.0972e+01   5.6868e-02 \\
   1.6777e+01   5.5360e-02 \\
   1.3422e+01   5.4150e-02 \\
   1.0737e+01   5.3180e-02 \\
   8.5899e+00   5.2400e-02 \\
   6.8719e+00   5.1773e-02 \\
   5.4976e+00   5.1267e-02 \\
   4.3980e+00   5.0856e-02 \\
   3.5184e+00   5.0521e-02 \\
   2.8147e+00   5.0243e-02 \\
   2.2518e+00   5.0008e-02 \\
   1.8014e+00   4.9803e-02 \\
   1.4412e+00   4.9617e-02 \\
   1.1529e+00   4.9435e-02 \\
   9.2234e-01   4.9243e-02 \\
   7.3787e-01   4.9017e-02 \\
   5.9030e-01   4.8709e-02 \\
};
\addlegendentry{Cartesian}

\addplot [only marks, color=purple, line width=0.5pt, mark=triangle*,  mark size=3.0pt, mark options={fill=purple,draw=black}]
  table[row sep=crcr]{%
   1.0000e+02   3.9903e-02 \\
   8.0000e+01   3.9895e-02 \\
   6.4000e+01   3.9886e-02 \\
   5.1200e+01   3.9874e-02 \\
   4.0960e+01   3.9859e-02 \\
   3.2768e+01   3.9841e-02 \\
   2.6214e+01   3.9818e-02 \\
   2.0972e+01   3.9791e-02 \\
   1.6777e+01   3.9757e-02 \\
   1.3422e+01   3.9716e-02 \\
   1.0737e+01   3.9666e-02 \\
   8.5899e+00   3.9606e-02 \\
   6.8719e+00   3.9534e-02 \\
   5.4976e+00   3.9450e-02 \\
   4.3980e+00   3.9351e-02 \\
   3.5184e+00   3.9237e-02 \\
   2.8147e+00   3.9107e-02 \\
   2.2518e+00   3.8961e-02 \\
   1.8014e+00   3.8797e-02 \\
   1.4412e+00   3.8616e-02 \\
   1.1529e+00   3.8414e-02 \\
   9.2234e-01   3.8183e-02 \\
   7.3787e-01   3.7901e-02 \\
   5.9030e-01   3.7486e-02 \\ 
};
\addlegendentry{Triangular}

\addplot [only marks, color=orange, line width=0.5pt, mark=*,  mark size=3.0pt, mark options={fill=orange,draw=black}]
  table[row sep=crcr]{%
   1.0000e+02   4.3871e-02 \\
   8.0000e+01   4.2902e-02 \\
   6.4000e+01   4.2115e-02 \\
   5.1200e+01   4.1473e-02 \\
   4.0960e+01   4.0948e-02 \\
   3.2768e+01   4.0515e-02 \\
   2.6214e+01   4.0155e-02 \\
   2.0972e+01   3.9850e-02 \\
   1.6777e+01   3.9587e-02 \\
   1.3422e+01   3.9349e-02 \\
   1.0737e+01   3.9121e-02 \\
   8.5899e+00   3.8876e-02 \\
};
\addlegendentry{Agglomerated}

\addplot [color=cyan, line width=1.25pt]
  table[row sep=crcr]{%
   4    0.03 \\
   4    0.10 \\
};

\addplot [color=teal, line width=1.25pt]
  table[row sep=crcr]{%
   5    0.03 \\
   5    0.10 \\
};

\addplot [color=purple, line width=1.25pt]
  table[row sep=crcr]{%
   3    0.03 \\
   3    0.10 \\
};

\addplot [color=orange, line width=1.25pt]
  table[row sep=crcr]{%
   72    0.03 \\
   72    0.10 \\
};

\nextgroupplot[
    width       = 2.5in,
    height      = 1.5in,
    scale only axis,
    xmode       = log,
    xmin        = 0.01,
    xmax        = 100,
    xminorticks = true,
    xticklabel  = {\textcolor{white}{1}},
    ymode       = log,
    ymin        = 0.001,
    ymax        = 0.10,
    ytick       = {0.001, 0.004, 0.01, 0.04, 0.10},
    yminorticks = true,
    yticklabel={\pgfmathparse{exp(\tick)}\pgfmathprintnumber{\pgfmathresult}},
    y tick label style={/pgf/number format/.cd, fixed, fixed zerofill, precision=3},
    axis background/.style = {fill=white},
    title style = {font=\bfseries},
    title       = {Coarse mesh with $p=4$},
    xmajorgrids,
    xminorgrids,
    ymajorgrids,
    yminorgrids,
    legend style={at={(0.55,0.95)},legend cell align=left, draw=white!15!black}
]
      
\addplot [only marks, color=cyan, line width=0.5pt, mark=square*,  mark size=3.0pt, mark options={fill=cyan,draw=black}]
  table[row sep=crcr]{%
   1.0000e+02   4.9348e-02 \\
   8.0000e+01   4.9348e-02 \\
   6.4000e+01   4.9348e-02 \\
   5.1200e+01   4.9348e-02 \\
   4.0960e+01   4.9348e-02 \\
   3.2768e+01   4.9348e-02 \\
   2.6214e+01   4.9348e-02 \\
   2.0972e+01   4.9348e-02 \\
   1.6777e+01   4.9348e-02 \\
   1.3422e+01   4.9348e-02 \\
   1.0737e+01   4.9348e-02 \\
   8.5899e+00   4.9348e-02 \\
   6.8719e+00   4.9348e-02 \\
   5.4976e+00   4.9348e-02 \\
   4.3980e+00   4.9348e-02 \\
   3.5184e+00   4.9348e-02 \\
   2.8147e+00   4.9348e-02 \\
   2.2518e+00   4.9348e-02 \\
   1.8014e+00   4.9348e-02 \\
   1.4412e+00   4.9348e-02 \\
   1.1529e+00   4.9348e-02 \\
   9.2234e-01   4.9348e-02 \\
};
        
\addplot [only marks, color=teal, line width=0.5pt, mark=pentagon*,  mark size=3.0pt, mark options={fill=teal,draw=black}]
  table[row sep=crcr]{%
   1.0000e+02   4.9556e-02 \\
   8.0000e+01   4.9552e-02 \\
   6.4000e+01   4.9547e-02 \\
   5.1200e+01   4.9542e-02 \\
   4.0960e+01   4.9537e-02 \\
   3.2768e+01   4.9532e-02 \\
   2.6214e+01   4.9527e-02 \\
   2.0972e+01   4.9521e-02 \\
   1.6777e+01   4.9516e-02 \\
   1.3422e+01   4.9509e-02 \\
   1.0737e+01   4.9502e-02 \\
   8.5899e+00   4.9494e-02 \\
   6.8719e+00   4.9485e-02 \\
   5.4976e+00   4.9475e-02 \\
   4.3980e+00   4.9461e-02 \\
   3.5184e+00   4.9444e-02 \\
   2.8147e+00   4.9421e-02 \\
   2.2518e+00   4.9389e-02 \\
   1.8014e+00   4.9339e-02 \\
   1.4412e+00   4.9254e-02 \\
   1.1529e+00   4.9067e-02 \\
};

\addplot [only marks, color=purple, line width=0.5pt, mark=triangle*,  mark size=3.0pt, mark options={fill=purple,draw=black}]
  table[row sep=crcr]{
   1.0000e+02   3.1104e-03 \\
   8.0000e+01   3.1062e-03 \\
   6.4000e+01   3.1009e-03 \\
   5.1200e+01   3.0942e-03 \\
   4.0960e+01   3.0859e-03 \\
   3.2768e+01   3.0756e-03 \\
   2.6214e+01   3.0627e-03 \\
   2.0972e+01   3.0467e-03 \\
   1.6777e+01   3.0266e-03 \\
   1.3422e+01   3.0017e-03 \\
   1.0737e+01   2.9707e-03 \\
   8.5899e+00   2.9320e-03 \\
   6.8719e+00   2.8839e-03 \\
   5.4976e+00   2.8167e-03 \\
   4.3980e+00   2.7315e-03 \\
   3.5184e+00   2.6301e-03 \\
   2.8147e+00   2.4370e-03 \\
   2.2518e+00   2.1475e-03 \\
   1.8014e+00   1.8190e-03 \\
   1.4412e+00   1.4421e-03 \\
   1.1529e+00   1.0023e-03 \\
   9.2234e-01   4.7710e-04 \\
};
<%

\addplot [only marks, color=orange, line width=0.5pt, mark=*,  mark size=3.0pt, mark options={fill=orange,draw=black}]
  table[row sep=crcr]{%
   1.0000e+02   3.8522e-02 \\
   8.0000e+01   3.8189e-02 \\
   6.4000e+01   3.5074e-02 \\
   5.1200e+01   3.1276e-02 \\
   4.0960e+01   2.7900e-02 \\
   3.2768e+01   2.4915e-02 \\
   2.6214e+01   2.2263e-02 \\
   2.0972e+01   1.9889e-02 \\
   1.6777e+01   1.7743e-02 \\
   1.3422e+01   1.5776e-02 \\
   1.0737e+01   1.3939e-02 \\
};

\addplot [color=cyan, line width=1.25pt]
  table[row sep=crcr]{%
   4    0.001 \\
   4    0.10 \\
};

\addplot [color=teal, line width=1.25pt]
  table[row sep=crcr]{%
   5    0.001 \\
   5    0.10 \\
};

\addplot [color=purple, line width=1.25pt]
  table[row sep=crcr]{%
   3    0.001 \\
   3    0.10 \\
};

\addplot [color=orange, line width=1.25pt]
  table[row sep=crcr]{%
   72    0.001 \\
   72    0.10 \\
};

\nextgroupplot[
    width       = 2.5in,
    height      = 1.5in,
    scale only axis,
    xmode       = log,
    xmin        = 0.01,
    xmax        = 100,
    xminorticks = true,
    xticklabel  = {\textcolor{white}{1}},
    ymode       = log,
    ymin        = 0.004,
    ymax        = 0.007,
    ytick       = {0.004, 0.005, 0.006, 0.007},
    yminorticks = true,
    yticklabel={\pgfmathparse{exp(\tick)}\pgfmathprintnumber{\pgfmathresult}},
    y tick label style={/pgf/number format/.cd, fixed, fixed zerofill, precision=3},
    axis background/.style = {fill=white},
    title style = {font=\bfseries},
    title       = {Refined mesh with $p=1$},
    xmajorgrids,
    xminorgrids,
    ymajorgrids,
    yminorgrids,
    legend style={at={(0.45,0.95)},legend cell align=left, draw=white!15!black}
]
              
\addplot [only marks, color=teal, line width=0.5pt, mark=pentagon*,  mark size=3.0pt, mark options={fill=teal,draw=black}]
  table[row sep=crcr]{%
   1.0000e+02   6.2271e-03 \\
   8.0000e+01   6.2087e-03 \\
   6.4000e+01   6.1939e-03 \\
   5.1200e+01   6.1819e-03 \\
   4.0960e+01   6.1722e-03 \\
   3.2768e+01   6.1643e-03 \\
   2.6214e+01   6.1580e-03 \\
   2.0972e+01   6.1529e-03 \\
   1.6777e+01   6.1487e-03 \\
   1.3422e+01   6.1453e-03 \\
   1.0737e+01   6.1424e-03 \\
   8.5899e+00   6.1401e-03 \\
   6.8719e+00   6.1381e-03 \\
   5.4976e+00   6.1364e-03 \\
   4.3980e+00   6.1349e-03 \\
   3.5184e+00   6.1334e-03 \\
   2.8147e+00   6.1320e-03 \\
   2.2518e+00   6.1305e-03 \\
   1.8014e+00   6.1287e-03 \\
   1.4412e+00   6.1265e-03 \\
   1.1529e+00   6.1237e-03 \\
   9.2234e-01   6.1196e-03 \\
   7.3787e-01   6.1132e-03 \\
};

\addplot [only marks, color=cyan, line width=0.5pt, mark=square*,  mark size=3.0pt, mark options={fill=cyan,draw=black}]
  table[row sep=crcr]{%
   1.0000e+02   5.9662e-03 \\
   8.0000e+01   5.8696e-03 \\
   6.4000e+01   5.7923e-03 \\
   5.1200e+01   5.7305e-03 \\
   4.0960e+01   5.6810e-03 \\
   3.2768e+01   5.6414e-03 \\
   2.6214e+01   5.6097e-03 \\
   2.0972e+01   5.5843e-03 \\
   1.6777e+01   5.5640e-03 \\
   1.3422e+01   5.5477e-03 \\
   1.0737e+01   5.5347e-03 \\
   8.5899e+00   5.5242e-03 \\
   6.8719e+00   5.5158e-03 \\
   5.4976e+00   5.5090e-03 \\
   4.3980e+00   5.5035e-03 \\
   3.5184e+00   5.4990e-03 \\
   2.8147e+00   5.4953e-03 \\
   2.2518e+00   5.4921e-03 \\
   1.8014e+00   5.4894e-03 \\
   1.4412e+00   5.4870e-03 \\
   1.1529e+00   5.4846e-03 \\
   9.2234e-01   5.4821e-03 \\
   7.3787e-01   5.4792e-03 \\
   5.9030e-01   5.4752e-03 \\
};

\addplot [only marks, color=purple, line width=0.5pt, mark=triangle*,  mark size=3.0pt, mark options={fill=purple,draw=black}]
  table[row sep=crcr]{%
   1.0000e+02   6.2053e-03 \\
   8.0000e+01   6.2051e-03 \\
   6.4000e+01   6.2049e-03 \\
   5.1200e+01   6.2046e-03 \\
   4.0960e+01   6.2042e-03 \\
   3.2768e+01   6.2038e-03 \\
   2.6214e+01   6.2032e-03 \\
   2.0972e+01   6.2025e-03 \\
   1.6777e+01   6.2017e-03 \\
   1.3422e+01   6.2007e-03 \\
   1.0737e+01   6.1994e-03 \\
   8.5899e+00   6.1979e-03 \\
   6.8719e+00   6.1961e-03 \\
   5.4976e+00   6.1939e-03 \\
   4.3980e+00   6.1914e-03 \\
   3.5184e+00   6.1885e-03 \\
   2.8147e+00   6.1852e-03 \\
   2.2518e+00   6.1814e-03 \\
   1.8014e+00   6.1771e-03 \\
   1.4412e+00   6.1724e-03 \\
   1.1529e+00   6.1671e-03 \\
   9.2234e-01   6.1611e-03 \\
   7.3787e-01   6.1539e-03 \\
   5.9030e-01   6.1440e-03 \\
};

\addplot [only marks, color=orange, line width=0.5pt, mark=*,  mark size=3.0pt, mark options={fill=orange,draw=black}]
  table[row sep=crcr]{%
   1.0000e+02   5.1350e-03 \\
   8.0000e+01   5.1008e-03 \\
   6.4000e+01   5.0732e-03 \\
   5.1200e+01   5.0507e-03 \\
   4.0960e+01   5.0324e-03 \\
   3.2768e+01   5.0175e-03 \\
   2.6214e+01   5.0053e-03 \\
   2.0972e+01   4.9953e-03 \\
   1.6777e+01   4.9870e-03 \\
   1.3422e+01   4.9802e-03 \\
   1.0737e+01   4.9744e-03 \\
   8.5899e+00   4.9695e-03 \\
   6.8719e+00   4.9653e-03 \\
   5.4976e+00   4.9615e-03 \\
   4.3980e+00   4.9579e-03 \\
   3.5184e+00   4.9541e-03 \\
   2.8147e+00   4.9481e-03 \\
};

\addplot [color=cyan, line width=1.25pt]
  table[row sep=crcr]{%
   4    0.001 \\
   4    0.10 \\
};

\addplot [color=teal, line width=1.25pt]
  table[row sep=crcr]{%
   5    0.001 \\
   5    0.10 \\
};

\addplot [color=purple, line width=1.25pt]
  table[row sep=crcr]{%
   3    0.001 \\
   3    0.10 \\
};

\addplot [color=orange, line width=1.25pt]
  table[row sep=crcr]{%
   31    0.001 \\
   31    0.10 \\
};

\nextgroupplot[
    width       = 2.5in,
    height      = 1.5in,
    scale only axis,
    xmode       = log,
    xmin        = 0.01,
    xmax        = 100,
    xminorticks = true,
    ylabel      = {$\lambda_1(\mathcal{A})$},
    ymin        = -5,
    ymax        = 0,
    yminorticks = true,
    xlabel      = {$\chi$},
    axis background/.style = {fill=white},
    y tick label style={/pgf/number format/.cd, fixed, fixed zerofill, precision=2},
    axis background/.style = {fill=white},
    xmajorgrids,
    xminorgrids,
    ymajorgrids,
    yminorgrids,
]
              
\addplot [only marks, color=teal, line width=0.5pt, mark=pentagon*,  mark size=3.0pt, mark options={fill=teal,draw=black}]
  table[row sep=crcr]{%
   4.7224e-01  -4.2310e-01 \\
   3.7779e-01  -9.2240e-01 \\
   3.0223e-01  -1.3806e+00 \\
   2.4179e-01  -1.7858e+00 \\
   1.9343e-01  -2.1513e+00 \\
   1.5474e-01  -2.4608e+00 \\
   1.2379e-01  -2.7141e+00 \\
   9.9035e-02  -2.9204e+00 \\
   7.9228e-02  -3.0876e+00 \\
   6.3383e-02  -3.2347e+00 \\
   5.0706e-02  -3.3535e+00 \\
   4.0565e-02  -3.4490e+00 \\
   3.2452e-02  -3.5255e+00 \\
   2.5961e-02  -3.5870e+00 \\
   2.0769e-02  -3.6362e+00 \\
   1.6615e-02  -3.6756e+00 \\
};

\addplot [only marks, color=cyan, line width=0.5pt, mark=square*,  mark size=3.0pt, mark options={fill=cyan,draw=black}]
  table[row sep=crcr]{%
   4.7224e-01  -1.6180e-01 \\
   3.7779e-01  -7.8953e-01 \\
   3.0223e-01  -1.3317e+00 \\
   2.4179e-01  -1.8625e+00 \\
   1.9343e-01  -2.3301e+00 \\
   1.5474e-01  -2.7160e+00 \\
   1.2379e-01  -3.0313e+00 \\
   9.9035e-02  -3.2873e+00 \\
   7.9228e-02  -3.4944e+00 \\
   6.3383e-02  -3.6614e+00 \\
   5.0706e-02  -3.7958e+00 \\
   4.0565e-02  -3.9039e+00 \\
   3.2452e-02  -3.9906e+00 \\
   2.5961e-02  -4.0602e+00 \\
   2.0769e-02  -4.1160e+00 \\
   1.6615e-02  -4.1607e+00 \\
};

\addplot [only marks, color=purple, line width=0.5pt, mark=triangle*,  mark size=3.0pt, mark options={fill=purple,draw=black}]
  table[row sep=crcr]{%
   4.7224e-01  -6.0482e-02 \\  
   3.7779e-01  -4.0105e-01 \\
   3.0223e-01  -7.0432e-01 \\
   2.4179e-01  -9.7881e-01 \\
   1.9343e-01  -1.2373e+00 \\
   1.5474e-01  -1.4882e+00 \\
   1.2379e-01  -1.7287e+00 \\
   9.9035e-02  -1.9504e+00 \\
   7.9228e-02  -2.1469e+00 \\
   6.3383e-02  -2.3163e+00 \\
   5.0706e-02  -2.4595e+00 \\
   4.0565e-02  -2.5785e+00 \\
   3.2452e-02  -2.6761e+00 \\
   2.5961e-02  -2.7554e+00 \\
   2.0769e-02  -2.8195e+00 \\
   1.6615e-02  -2.8711e+00 \\
};
<%

\addplot [only marks, color=orange, line width=0.5pt, mark=*,  mark size=3.0pt, mark options={fill=orange,draw=black}]
  table[row sep=crcr]{%
   6.8719e+00  -1.1877e-01 \\
   5.4976e+00  -8.1059e-01 \\
   4.3980e+00  -1.3972e+00 \\
   3.5184e+00  -1.8865e+00 \\
   2.8147e+00  -2.2928e+00 \\
   2.2518e+00  -2.6667e+00 \\
   1.8014e+00  -3.0335e+00 \\
   1.4412e+00  -3.3312e+00 \\
   1.1529e+00  -3.5715e+00 \\
   9.2234e-01  -3.7649e+00 \\
   7.3787e-01  -3.9203e+00 \\
   5.9030e-01  -4.0451e+00 \\
   4.7224e-01  -4.1452e+00 \\
   3.7779e-01  -4.2254e+00 \\
   3.0223e-01  -4.2897e+00 \\
   2.4179e-01  -4.3412e+00 \\
   1.9343e-01  -4.3825e+00 \\
   1.5474e-01  -4.4155e+00 \\
   1.2379e-01  -4.4419e+00 \\
   9.9035e-02  -4.4631e+00 \\
   7.9228e-02  -4.4800e+00 \\
   6.3383e-02  -4.4936e+00 \\
   5.0706e-02  -4.5044e+00 \\
   4.0565e-02  -4.5131e+00 \\
   3.2452e-02  -4.5200e+00 \\
   2.5961e-02  -4.5256e+00 \\
   2.0769e-02  -4.5300e+00 \\
   1.6615e-02  -4.5336e+00 \\
};

\addplot [color=cyan, line width=1.25pt]
  table[row sep=crcr]{%
   4    0.00 \\
   4   -5.00 \\
};

\addplot [color=teal, line width=1.25pt]
  table[row sep=crcr]{%
   5    0.00 \\
   5   -5.00 \\
};

\addplot [color=purple, line width=1.25pt]
  table[row sep=crcr]{%
   3    0.00 \\
   3   -5.00 \\
};

\addplot [color=orange, line width=1.25pt]
  table[row sep=crcr]{%
   72     0.00 \\
   72    -5.00 \\
};

\nextgroupplot[
    width       = 2.5in,
    height      = 1.5in,
    scale only axis,
    xmode       = log,
    xmin        = 0.01,
    xmax        = 100,
    xminorticks = true,
    ymin        = -300,
    ymax        = 0,
    yminorticks = true,
    xlabel      = {$\chi$},
    axis background/.style = {fill=white},
    y tick label style={/pgf/number format/.cd, fixed, fixed zerofill, precision=0},
    axis background/.style = {fill=white},
    xmajorgrids,
    xminorgrids,
    ymajorgrids,
    yminorgrids,
]
              
\addplot [only marks, color=teal, line width=0.5pt, mark=pentagon*,  mark size=3.0pt, mark options={fill=teal,draw=black}]
  table[row sep=crcr]{%
   9.2234e-01  -1.7725e+00 \\
   7.3787e-01  -1.1375e+01 \\
   5.9030e-01  -2.6991e+01 \\
   4.7224e-01  -4.6610e+01 \\
   3.7779e-01  -6.3777e+01 \\
   3.0223e-01  -7.8251e+01 \\
   2.4179e-01  -9.0238e+01 \\
   1.9343e-01  -1.0007e+02 \\
   1.5474e-01  -1.0808e+02 \\
   1.2379e-01  -1.1457e+02 \\
   9.9035e-02  -1.1983e+02 \\
   7.9228e-02  -1.2407e+02 \\
   6.3383e-02  -1.2749e+02 \\
   5.0706e-02  -1.3024e+02 \\
   4.0565e-02  -1.3245e+02 \\
   3.2452e-02  -1.3422e+02 \\
   2.5961e-02  -1.3564e+02 \\
   2.0769e-02  -1.3678e+02 \\
   1.6615e-02  -1.3770e+02 \\
};

\addplot [only marks, color=cyan, line width=0.5pt, mark=square*,  mark size=3.0pt, mark options={fill=cyan,draw=black}]
  table[row sep=crcr]{%
   7.3787e-01  -7.6986e+00 \\
   5.9030e-01  -3.1072e+01 \\
   4.7224e-01  -5.5231e+01 \\
   3.7779e-01  -7.7541e+01 \\
   3.0223e-01  -9.7053e+01 \\
   2.4179e-01  -1.1359e+02 \\
   1.9343e-01  -1.2734e+02 \\
   1.5474e-01  -1.3864e+02 \\
   1.2379e-01  -1.4786e+02 \\
   9.9035e-02  -1.5534e+02 \\
   7.9228e-02  -1.6139e+02 \\
   6.3383e-02  -1.6626e+02 \\
   5.0706e-02  -1.7019e+02 \\
   4.0565e-02  -1.7335e+02 \\
   3.2452e-02  -1.7588e+02 \\
   2.5961e-02  -1.7792e+02 \\
   2.0769e-02  -1.7955e+02 \\
   1.6615e-02  -1.8085e+02 \\
};

\addplot [only marks, color=purple, line width=0.5pt, mark=triangle*,  mark size=3.0pt, mark options={fill=purple,draw=black}]
  table[row sep=crcr]{%
   7.3787e-01  -3.3367e-01 \\
   5.9030e-01  -9.9977e+00 \\
   4.7224e-01  -2.3577e+01 \\
   3.7779e-01  -3.9209e+01 \\
   3.0223e-01  -5.4764e+01 \\
   2.4179e-01  -6.9228e+01 \\
   1.9343e-01  -8.2574e+01 \\
   1.5474e-01  -9.3861e+01 \\
   1.2379e-01  -1.0319e+02 \\
   9.9035e-02  -1.1082e+02 \\
   7.9228e-02  -1.1703e+02 \\
   6.3383e-02  -1.2206e+02 \\
   5.0706e-02  -1.2613e+02 \\
   4.0565e-02  -1.2940e+02 \\
   3.2452e-02  -1.3203e+02 \\
   2.5961e-02  -1.3415e+02 \\
   2.0769e-02  -1.3584e+02 \\
   1.6615e-02  -1.3721e+02 \\
};

\addplot [only marks, color=orange, line width=0.5pt, mark=*,  mark size=3.0pt, mark options={fill=orange,draw=black}]
  table[row sep=crcr]{%
   8.5899e+00  -7.1364e+00 \\
   6.8719e+00  -3.6056e+01 \\
   5.4976e+00  -6.4684e+01 \\
   4.3980e+00  -9.0372e+01 \\
   3.5184e+00  -1.1269e+02 \\
   2.8147e+00  -1.3176e+02 \\
   2.2518e+00  -1.4784e+02 \\
   1.8014e+00  -1.6126e+02 \\
   1.4412e+00  -1.7238e+02 \\
   1.1529e+00  -1.8152e+02 \\
   9.2234e-01  -1.8899e+02 \\
   7.3787e-01  -1.9508e+02 \\
   5.9030e-01  -2.0001e+02 \\
   4.7224e-01  -2.0400e+02 \\
   3.7779e-01  -2.0722e+02 \\
   3.0223e-01  -2.0981e+02 \\
   2.4179e-01  -2.1190e+02 \\
   1.9343e-01  -2.1357e+02 \\
   1.5474e-01  -2.1492e+02 \\
   1.2379e-01  -2.1600e+02 \\
   9.9035e-02  -2.1686e+02 \\
   7.9228e-02  -2.1756e+02 \\
   6.3383e-02  -2.1811e+02 \\
   5.0706e-02  -2.1856e+02 \\
   4.0565e-02  -2.1891e+02 \\
   3.2452e-02  -2.1920e+02 \\
   2.5961e-02  -2.1943e+02 \\
   2.0769e-02  -2.1961e+02 \\
   1.6615e-02  -2.1976e+02 \\
};

\addplot [color=cyan, line width=1.25pt]
  table[row sep=crcr]{%
   4    0.00 \\
   4   -305.00 \\
};

\addplot [color=teal, line width=1.25pt]
  table[row sep=crcr]{%
   5    0.00 \\
   5   -305.00 \\
};

\addplot [color=purple, line width=1.25pt]
  table[row sep=crcr]{%
   3    0.00 \\
   3   -305.00 \\
};

\addplot [color=orange, line width=1.25pt]
  table[row sep=crcr]{%
   72     0.00 \\
   72    -305.00 \\
};

\nextgroupplot[
    width       = 2.5in,
    height      = 1.5in,
    scale only axis,
    xmode       = log,
    xmin        = 0.01,
    xmax        = 100,
    xminorticks = true,
    ymin        = -5.2,
    ymax        = 0,
    yminorticks = true,
    xlabel      = {$\chi$},
    axis background/.style = {fill=white},
    y tick label style={/pgf/number format/.cd, fixed, fixed zerofill, precision=2},
    axis background/.style = {fill=white},
    xmajorgrids,
    xminorgrids,
    ymajorgrids,
]
              
\addplot [only marks, color=teal, line width=0.5pt, mark=pentagon*,  mark size=3.0pt, mark options={fill=teal,draw=black}]
  table[row sep=crcr]{%
   5.9030e-01  -9.5072e-02 \\
   4.7224e-01  -6.1798e-01 \\
   3.7779e-01  -1.1207e+00 \\
   3.0223e-01  -1.6375e+00 \\
   2.4179e-01  -2.0921e+00 \\
   1.9343e-01  -2.5481e+00 \\
   1.5474e-01  -2.9189e+00 \\
   1.2379e-01  -3.2190e+00 \\
   9.9035e-02  -3.4611e+00 \\
   7.9228e-02  -3.6561e+00 \\
   6.3383e-02  -3.8128e+00 \\
   5.0706e-02  -3.9386e+00 \\
   4.0565e-02  -4.0396e+00 \\
   3.2452e-02  -4.1206e+00 \\
   2.5961e-02  -4.1854e+00 \\
   2.0769e-02  -4.2374e+00 \\
   1.6615e-02  -4.2790e+00 \\
};

\addplot [only marks, color=cyan, line width=0.5pt, mark=square*,  mark size=3.0pt, mark options={fill=cyan,draw=black}]
  table[row sep=crcr]{%
   4.7224e-01  -1.7025e-01 \\
   3.7779e-01  -7.9631e-01 \\
   3.0223e-01  -1.3785e+00 \\
   2.4179e-01  -1.9173e+00 \\
   1.9343e-01  -2.3685e+00 \\
   1.5474e-01  -2.7402e+00 \\
   1.2379e-01  -3.0436e+00 \\
   9.9035e-02  -3.2899e+00 \\
   7.9228e-02  -3.4890e+00 \\
   6.3383e-02  -3.6497e+00 \\
   5.0706e-02  -3.7792e+00 \\
   4.0565e-02  -3.8834e+00 \\
   3.2452e-02  -3.9672e+00 \\
   2.5961e-02  -4.0345e+00 \\
   2.0769e-02  -4.0887e+00 \\
   1.6615e-02  -4.1321e+00 \\
};

\addplot [only marks, color=purple, line width=0.5pt, mark=triangle*,  mark size=3.0pt, mark options={fill=purple,draw=black}]
  table[row sep=crcr]{%
   4.7224e-01  -6.0482e-02 \\
   3.7779e-01  -4.0105e-01 \\
   3.0223e-01  -7.0432e-01 \\
   2.4179e-01  -9.7881e-01 \\
   1.9343e-01  -1.2373e+00 \\
   1.5474e-01  -1.4882e+00 \\
   1.2379e-01  -1.7287e+00 \\
   9.9035e-02  -1.9504e+00 \\
   7.9228e-02  -2.1469e+00 \\
   6.3383e-02  -2.3164e+00 \\
   5.0706e-02  -2.4599e+00 \\
   4.0565e-02  -2.5799e+00 \\
   3.2452e-02  -2.6791e+00 \\
   2.5961e-02  -2.7598e+00 \\
   2.0769e-02  -2.8250e+00 \\
   1.6615e-02  -2.8775e+00 \\
};
<%

\addplot [only marks, color=orange, line width=0.5pt, mark=*,  mark size=3.0pt, mark options={fill=orange,draw=black}]
  table[row sep=crcr]{%
   2.2518e+00  -6.8847e-01 \\
   1.8014e+00  -1.4546e+00 \\
   1.4412e+00  -2.1537e+00 \\
   1.1529e+00  -2.7305e+00 \\ 
   9.2234e-01  -3.2015e+00 \\
   7.3787e-01  -3.5836e+00 \\
   5.9030e-01  -3.8923e+00 \\
   4.7224e-01  -4.1411e+00 \\
   3.7779e-01  -4.3413e+00 \\
   3.0223e-01  -4.5020e+00 \\ 
   2.4179e-01  -4.6311e+00 \\
   1.9343e-01  -4.7345e+00 \\
   1.5474e-01  -4.8175e+00 \\
   1.2379e-01  -4.8839e+00 \\
   9.9035e-02  -4.9371e+00 \\
   7.9228e-02  -4.9797e+00 \\
   6.3383e-02  -5.0139e+00 \\
   5.0706e-02  -5.0412e+00 \\
   4.0565e-02  -5.0630e+00 \\
   3.2452e-02  -5.0805e+00 \\
   2.5961e-02  -5.0945e+00 \\
   2.0769e-02  -5.1057e+00 \\
   1.6615e-02  -5.1147e+00 \\
};

\addplot [color=cyan, line width=1.25pt]
  table[row sep=crcr]{%
   4    0.00 \\
   4   -6.00 \\
};

\addplot [color=teal, line width=1.25pt]
  table[row sep=crcr]{%
   5    0.00 \\
   5   -6.00 \\
};

\addplot [color=purple, line width=1.25pt]
  table[row sep=crcr]{%
   3    0.00 \\
   3   -6.00 \\
};

\addplot [color=orange, line width=1.25pt]
  table[row sep=crcr]{%
   31     0.00 \\
   31    -6.00 \\
};

\end{groupplot}
\end{tikzpicture}

%% file: Errors_p_Tria.tex
\definecolor{mycolor2}{rgb}{0.00000,1.00000,1.00000}%
\begin{tikzpicture}
\begin{axis}[%
width=3in,
height=2.8in,
at={(2.6in,1.099in)},
scale only axis,
xmode=log,
xmin=1,
xmax=12,
xminorticks=true,
xlabel = {$p$ [-]},
ylabel = {},
xticklabel={\pgfmathparse{exp(\tick)}\pgfmathprintnumber{\pgfmathresult}},
x tick label style={
/pgf/number format/.cd, fixed, fixed zerofill,
precision=2},
ymode=log,
ymin=5e-5,
ymax=5e+0,
yminorticks=true,
axis background/.style={fill=white},
title style={font=\bfseries},
xmajorgrids,
xminorgrids,
ymajorgrids,
yminorgrids,
legend style={at={(0.96,0.96)},legend cell align=left, draw=white!15!black}
]
              
\addplot [color=blue, line width=2.0pt, mark=square*]
  table[row sep=crcr]{%
   1.0000e+00   2.5982e+00 \\
   2.0000e+00   1.7066e-01 \\
   3.0000e+00   1.5365e-02 \\
   4.0000e+00   2.6592e-03 \\
   5.0000e+00   1.5357e-03 \\
   6.0000e+00   7.2925e-04 \\
   7.0000e+00   5.1737e-04 \\
   8.0000e+00   3.2251e-04 \\
   9.0000e+00   2.3942e-04 \\
   1.0000e+01   1.7198e-04 \\
   1.1000e+01   1.3136e-04 \\
   1.2000e+01   1.0268e-04 \\ 
};
\addlegendentry{$\Tnorm{u-u_h}{\mathrm{CDG}}$}

\addplot [color=orange, line width=2.0pt, mark=*]
  table[row sep=crcr]{%
   1.0000e+00   2.5153e+00 \\
   2.0000e+00   2.2629e-01 \\
   3.0000e+00   2.1461e-02 \\
   4.0000e+00   4.5960e-03 \\
   5.0000e+00   2.2391e-03 \\
   6.0000e+00   1.4733e-03 \\
   7.0000e+00   6.6344e-04 \\
   8.0000e+00   6.5695e-04 \\
   9.0000e+00   2.8102e-04 \\
   1.0000e+01   3.4516e-04 \\
   1.1000e+01   1.4545e-04 \\
   1.2000e+01   2.0178e-04 \\
};
\addlegendentry{$\Tnorm{u-u_h}{\mathrm{IPDG}}$}

\addplot [color=black, line width=1.5pt, dashed]
  table[row sep=crcr]{%
   1.0000e+00   1.000e-01 \\
   1.2000e+01   2.004e-04 \\
};
\addlegendentry{$p^{-5/2}$}

\addplot [color=black, line width=1.5pt, dotted]
  table[row sep=crcr]{%
   1.0000e+00   0.1250e+00 \\
   1.2000e+01   0.7233e-04 \\
};
\addlegendentry{$p^{-3}$}

\end{axis}
\end{tikzpicture}%

%% file: Errors_p_Voro.tex
\definecolor{mycolor2}{rgb}{0.00000,1.00000,1.00000}%
\begin{tikzpicture}
\begin{axis}[%
width=3in,
height=2.8in,
at={(2.6in,1.099in)},
scale only axis,
xmode=log,
xmin=1,
xmax=12,
xminorticks=true,
xlabel = {$p$ [-]},
ylabel = {},
xticklabel={\pgfmathparse{exp(\tick)}\pgfmathprintnumber{\pgfmathresult}},
x tick label style={
/pgf/number format/.cd, fixed, fixed zerofill,
precision=2},
ymode=log,
ymin=5e-5,
ymax=5e+0,
yminorticks=true,
axis background/.style={fill=white},
title style={font=\bfseries},
xmajorgrids,
xminorgrids,
ymajorgrids,
yminorgrids,
legend style={at={(0.96,0.96)},legend cell align=left, draw=white!15!black}
]
              
\addplot [color=blue, line width=2.0pt, mark=square*]
  table[row sep=crcr]{%
   1.0000e+00   3.4098e+00 \\
   2.0000e+00   2.9460e-01 \\
   3.0000e+00   2.8485e-02 \\
   4.0000e+00   7.9469e-03 \\
   5.0000e+00   2.6878e-03 \\
   6.0000e+00   1.5475e-03 \\
   7.0000e+00   9.2106e-04 \\
   8.0000e+00   5.6632e-04 \\
   9.0000e+00   4.0511e-04 \\
   1.0000e+01   2.7211e-04 \\
   1.1000e+01   2.1084e-04 \\
   1.2000e+01   1.5037e-04 \\
};
\addlegendentry{$\Tnorm{u-u_h}{\mathrm{CDG}}$}

\addplot [color=orange, line width=2.0pt, mark=*]
  table[row sep=crcr]{%
   1.0000e+00   4.3790e+00 \\
   2.0000e+00   5.1203e-01 \\
   3.0000e+00   4.9584e-02 \\
   4.0000e+00   1.8021e-02 \\
   5.0000e+00   3.8828e-03 \\
   6.0000e+00   4.8899e-03 \\   
   7.0000e+00   1.2720e-03 \\
   8.0000e+00   2.1185e-03 \\
   9.0000e+00   5.7158e-04 \\
   1.0000e+01   1.1263e-03 \\
   1.1000e+01   2.9602e-04 \\
   1.2000e+01   6.7761e-04 \\
};
\addlegendentry{$\Tnorm{u-u_h}{\mathrm{IPDG}}$}

\addplot [color=black, line width=1.5pt, dashed]
  table[row sep=crcr]{%
   1.0000e+00   0.2500e+00 \\
   1.2000e+01   0.5012e-03 \\
};
\addlegendentry{$p^{-5/2}$}

\addplot [color=black, line width=1.5pt, dotted]
  table[row sep=crcr]{%
   1.0000e+00   0.1250e+00 \\
   1.2000e+01   0.7233e-04 \\
};
\addlegendentry{$p^{-3}$}

\end{axis}
\end{tikzpicture}%

%% file: references.bib
@article {Periaire_Persson:2008,
    AUTHOR = {Peraire, J. and Persson, P.-O.},
     TITLE = {The compact discontinuous {G}alerkin ({CDG}) method for
              elliptic problems},
   JOURNAL = {SIAM J. Sci. Comput.},
  FJOURNAL = {SIAM Journal on Scientific Computing},
    VOLUME = {30},
      YEAR = {2008},
    NUMBER = {4},
     PAGES = {1806--1824},
      ISSN = {1064-8275,1095-7197},
   MRCLASS = {65N30},
  MRNUMBER = {2407142},
       }

@article {Cockburn_Shu:1998,
    AUTHOR = {Cockburn, B. and Shu, C.-W.},
     TITLE = {The local discontinuous {G}alerkin method for time-dependent
              convection-diffusion systems},
   JOURNAL = {SIAM J. Numer. Anal.},
  FJOURNAL = {SIAM Journal on Numerical Analysis},
    VOLUME = {35},
      YEAR = {1998},
    NUMBER = {6},
     PAGES = {2440--2463},
      ISSN = {0036-1429,1095-7170},
   MRCLASS = {65M60 (76M25)},
  MRNUMBER = {1655854},
MRREVIEWER = {Leonid\ K.\ Antanovski\u i},
       }

@article {Arnold:1982,
    AUTHOR = {Arnold, D. N.},
     TITLE = {An interior penalty finite element method with discontinuous
              elements},
   JOURNAL = {SIAM J. Numer. Anal.},
  FJOURNAL = {SIAM Journal on Numerical Analysis},
    VOLUME = {19},
      YEAR = {1982},
    NUMBER = {4},
     PAGES = {742--760},
      ISSN = {0036-1429},
   MRCLASS = {65N30},
  MRNUMBER = {664882},
       }

@article {antonietti_magnet_2025,
    AUTHOR = {Antonietti, P. F. and Caldana, M. and Mazzieri, I. and Re Fraschini, A.},
     TITLE = {MAGNET: an open-source library for mesh agglomeration by graph neural networks},
   JOURNAL = {Eng. Comput.},
  FJOURNAL = {Engineering with Computers},
    VOLUME = {41},
      YEAR = {2025},
    NUMBER = {6},
     PAGES = {4825--4850},
      ISSN = {1435-5663},
       }

@article {Liu_Yan:2008,
    AUTHOR = {Liu, H. and Yan, J.},
     TITLE = {The direct discontinuous {G}alerkin ({DDG}) methods for
              diffusion problems},
   JOURNAL = {SIAM J. Numer. Anal.},
  FJOURNAL = {SIAM Journal on Numerical Analysis},
    VOLUME = {47},
      YEAR = {2008/09},
    NUMBER = {1},
     PAGES = {675--698},
      ISSN = {0036-1429,1095-7170},
   MRCLASS = {65M60},
  MRNUMBER = {2475957},
MRREVIEWER = {Lucia\ Gastaldi},
       }

@article {Beirao_etal:2013,
    AUTHOR = {Beir\~ao da Veiga, L. and Brezzi, F. and Cangiani, A. and
              Manzini, G. and Marini, L. D. and Russo, A.},
     TITLE = {Basic principles of virtual element methods},
   JOURNAL = {Math. Models Methods Appl. Sci.},
  FJOURNAL = {Mathematical Models and Methods in Applied Sciences},
    VOLUME = {23},
      YEAR = {2013},
    NUMBER = {1},
     PAGES = {199--214},
      ISSN = {0218-2025,1793-6314},
   MRCLASS = {65N06},
  MRNUMBER = {2997471},
MRREVIEWER = {Bo\v sko\ S.\ Jovanovi\'c},
       }

@article {DiPietro_Ern_Lemaire:2014,
    AUTHOR = {Di Pietro, D. A. and Ern, A. and Lemaire, S.},
     TITLE = {An arbitrary-order and compact-stencil discretization of
              diffusion on general meshes based on local reconstruction
              operators},
   JOURNAL = {Comput. Methods Appl. Math.},
  FJOURNAL = {Computational Methods in Applied Mathematics},
    VOLUME = {14},
      YEAR = {2014},
    NUMBER = {4},
     PAGES = {461--472},
      ISSN = {1609-4840,1609-9389},
   MRCLASS = {65N08 (65N12)},
  MRNUMBER = {3259024},
MRREVIEWER = {Andrei\ I.\ Tolstykh},
       }

@article {Cangiani_Georgoulis_Houston:2014,
    AUTHOR = {Cangiani, A. and Georgoulis, E. H. and Houston,
              P.},
     TITLE = {{$hp$}-version discontinuous {G}alerkin methods on polygonal
              and polyhedral meshes},
   JOURNAL = {Math. Models Methods Appl. Sci.},
  FJOURNAL = {Mathematical Models and Methods in Applied Sciences},
    VOLUME = {24},
      YEAR = {2014},
    NUMBER = {10},
     PAGES = {2009--2041},
      ISSN = {0218-2025,1793-6314},
   MRCLASS = {65N30 (65N50 65N55)},
  MRNUMBER = {3211116},
MRREVIEWER = {C.\ Ilioi},
       }

@book{Cangiani_Dong_Georgoulis_Houston:2017,
 author = {Cangiani, A. and Dong, Z. and Georgoulis, E. H. and Houston, P.},
 title = {{{\(hp\)}}-version discontinuous {Galerkin} methods on polygonal and polyhedral meshes},
 fseries = {SpringerBriefs in Mathematics},
 series = {SpringerBriefs Math.},
 issn = {2191-8198},
 isbn = {978-3-319-67671-5; 978-3-319-67673-9},
 year = {2017},
 publisher = {Cham: Springer},
 language = {English},
 }

@article {Castillo:2002,
    AUTHOR = {Castillo, P.},
     TITLE = {Performance of discontinuous {G}alerkin methods for elliptic
              {PDE}s},
   JOURNAL = {SIAM J. Sci. Comput.},
  FJOURNAL = {SIAM Journal on Scientific Computing},
    VOLUME = {24},
      YEAR = {2002},
    NUMBER = {2},
     PAGES = {524--547},
      ISSN = {1064-8275,1095-7197},
   MRCLASS = {65N30},
  MRNUMBER = {1951054},
       }

@article{Perugia_Schotzau:2002,
 author = {Perugia, I. and Sch{\"o}tzau, D.},
 title = {An {{\(hp\)}}-analysis of the local discontinuous {Galerkin} method for diffusion problems},
 fjournal = {Journal of Scientific Computing},
 journal = {J. Sci. Comput.},
 issn = {0885-7474},
 volume = {17},
 number = {1-4},
 pages = {561--571},
 year = {2002},
 language = {English},
 }

@article {Arnold_Brezzi_Cockburn_Marini:2001,
    AUTHOR = {Arnold, D. N. and Brezzi, F. and Cockburn, B.
              and Marini, L. D.},
     TITLE = {Unified analysis of discontinuous {G}alerkin methods for
              elliptic problems},
   JOURNAL = {SIAM J. Numer. Anal.},
  FJOURNAL = {SIAM Journal on Numerical Analysis},
    VOLUME = {39},
      YEAR = {2001/02},
    NUMBER = {5},
     PAGES = {1749--1779},
      ISSN = {0036-1429,1095-7170},
   MRCLASS = {65N30},
  MRNUMBER = {1885715},
       }

@article {Castillo:2010,
    AUTHOR = {Castillo, P.},
     TITLE = {Stencil reduction algorithms for the local discontinuous
              {G}alerkin method},
   JOURNAL = {Internat. J. Numer. Methods Engrg.},
  FJOURNAL = {International Journal for Numerical Methods in Engineering},
    VOLUME = {81},
      YEAR = {2010},
    NUMBER = {12},
     PAGES = {1475--1491},
      ISSN = {0029-5981,1097-0207},
   MRCLASS = {65N30},
  MRNUMBER = {2642817},
       }

@phdthesis{Wihler:2002,
  title={Discontinuous {G}alerkin {FEM} for elliptic problems in polygonal domains},
  author={Wihler, T. P.},
  year={2002},
  school={ETH Zurich}
}

@article {Sherwin_etal:2006,
    AUTHOR = {Sherwin, S. J. and Kirby, R. M. and Peir\'o, J. and Taylor, R.
              L. and Zienkiewicz, O. C.},
     TITLE = {On 2{D} elliptic discontinuous {G}alerkin methods},
   JOURNAL = {Internat. J. Numer. Methods Engrg.},
  FJOURNAL = {International Journal for Numerical Methods in Engineering},
    VOLUME = {65},
      YEAR = {2006},
    NUMBER = {5},
     PAGES = {752--784},
      ISSN = {0029-5981,1097-0207},
   MRCLASS = {65N30 (35J25)},
  MRNUMBER = {2195978},
       }

@article {Brdar_Dedner_Klofkorn:2012,
    AUTHOR = {Brdar, S. and Dedner, A. and Kl\"ofkorn, R.},
     TITLE = {Compact and stable discontinuous {G}alerkin methods for
              convection-diffusion problems},
   JOURNAL = {SIAM J. Sci. Comput.},
  FJOURNAL = {SIAM Journal on Scientific Computing},
    VOLUME = {34},
      YEAR = {2012},
    NUMBER = {1},
     PAGES = {A263--A282},
      ISSN = {1064-8275,1095-7197},
   MRCLASS = {65M60},
  MRNUMBER = {2890266},
       }

@article {Huang_Huang:2013,
    AUTHOR = {Huang, X. and Huang, J.},
     TITLE = {The compact discontinuous {G}alerkin method for nearly
              incompressible linear elasticity},
   JOURNAL = {J. Sci. Comput.},
  FJOURNAL = {Journal of Scientific Computing},
    VOLUME = {56},
      YEAR = {2013},
    NUMBER = {2},
     PAGES = {291--318},
      ISSN = {0885-7474,1573-7691},
   MRCLASS = {65N30 (74B05 74S05)},
  MRNUMBER = {3071177},
MRREVIEWER = {Bj\"orn\ Stinner},
       }

@inproceedings{Bassi_Rebay:1997,
  title={A high-order accurate discontinuous finite element method for inviscid and viscous turbomachinery flows},
  author={Bassi, F. and Rebay, S. and Savini, M. and Mariotti, G. and Pedinotti},
  booktitle={2-nd European Conference on Turbomachinery Fluid Dynamics and Thermodynamics},
  pages={99--108},
  year={1997},
  organization={Technologisch Instituut, Antwerpen, Belgium}
}

@article {Brezzi_etal:2000,
    AUTHOR = {Brezzi, F. and Manzini, G. and Marini, D. and Pietra, P. and
              Russo, A.},
     TITLE = {Discontinuous {G}alerkin approximations for elliptic problems},
   JOURNAL = {Numer. Methods Partial Differential Equations},
  FJOURNAL = {Numerical Methods for Partial Differential Equations. An
              International Journal},
    VOLUME = {16},
      YEAR = {2000},
    NUMBER = {4},
     PAGES = {365--378},
      ISSN = {0749-159X,1098-2426},
   MRCLASS = {65N30 (65N12)},
  MRNUMBER = {1765651},
MRREVIEWER = {K.\ Najzar},
       }

@article {Bassi_etal:2012,
    AUTHOR = {Bassi, F. and Botti, L. and Colombo, A. and Di Pietro, D. A.
              and Tesini, P.},
     TITLE = {On the flexibility of agglomeration based physical space
              discontinuous {G}alerkin discretizations},
   JOURNAL = {J. Comput. Phys.},
  FJOURNAL = {Journal of Computational Physics},
    VOLUME = {231},
      YEAR = {2012},
    NUMBER = {1},
     PAGES = {45--65},
      ISSN = {0021-9991,1090-2716},
   MRCLASS = {76M10},
  MRNUMBER = {2846986},
       }

@article {Ye_Zhang_Zhu:2022,
    AUTHOR = {Ye, X. and Zhang, S. and Zhu, P.},
     TITLE = {Development of a {LDG} method on polytopal mesh with optimal
              order of convergence},
   JOURNAL = {J. Comput. Appl. Math.},
  FJOURNAL = {Journal of Computational and Applied Mathematics},
    VOLUME = {410},
      YEAR = {2022},
     PAGES = {Paper No. 114179, 10},
      ISSN = {0377-0427,1879-1778},
   MRCLASS = {65N30 (65N12 65N15)},
  MRNUMBER = {4390956},
MRREVIEWER = {Sebastian\ Franz},
       }

@article {Gomez-Perinati-Stocker:2026,
    AUTHOR = {G\'omez, S. and Perinati, C. and Stocker, P.},
          TITLE = {Inf-sup stable space--time {L}ocal {D}iscontinuous {G}alerkin method for the heat equation}, 
  JOURNAL = {J. Sci. Comput.},
    VOLUME = {106},
      YEAR = {2026},
     PAGES = {Paper No. 22},
       }

@article {Castillo_Sequeira:2013,
    AUTHOR = {Castillo, P. and Sequeira, F. A.},
     TITLE = {Computational aspects of the local discontinuous {G}alerkin
              method on unstructured grids in three dimensions},
   JOURNAL = {Math. Comput. Modelling},
  FJOURNAL = {Mathematical and Computer Modelling},
    VOLUME = {57},
      YEAR = {2013},
    NUMBER = {9-10},
     PAGES = {2279--2288},
      ISSN = {0895-7177},
   MRCLASS = {65N30 (35J25)},
  MRNUMBER = {3068721},
       }

@article {Pan_Persson:2022,
    AUTHOR = {Pan, Y. and Persson, P.-O.},
     TITLE = {Agglomeration-based geometric multigrid solvers for compact
              discontinuous {G}alerkin discretizations on unstructured
              meshes},
   JOURNAL = {J. Comput. Phys.},
  FJOURNAL = {Journal of Computational Physics},
    VOLUME = {449},
      YEAR = {2022},
     PAGES = {Paper No. 110775, 12},
      ISSN = {0021-9991,1090-2716},
   MRCLASS = {65N22 (65F10 65N30)},
  MRNUMBER = {4338996},
       }

@article{antonietti_lymph_2024,
    AUTHOR = {Antonietti, P. F. and Bonetti, S. and Botti, M. and Corti, M. and Fumagalli, I. and Mazzieri, I.},
     TITLE = {\texttt{lymph}: discontinuous poL{Y}topal methods for {M}ulti-{PH}ysics differential problems},
      YEAR = {2025},
   JOURNAL = {ACM Trans. Math. Softw.},
  FJOURNAL = {ACM Transactions on Mathematical Software},
    VOLUME = {51},
    NUMBER = {1},
     PAGES = {3:1--3:22},
      ISSN = {0098-3500,1557-7295},
   MRCLASS = {65M60},
  MRNUMBER = {4894504},
       }

@article{antonietti_polytopal_2026,
	AUTHOR = {Antonietti, P. F. and Corti, M. and Martinelli, G.},
     TITLE = {Polytopal mesh agglomeration via geometrical deep learning for three-dimensional heterogeneous domains},
      YEAR = {2026},
   JOURNAL = {Math. Comput. Simul.},
  FJOURNAL = {Mathematics and Computers in Simulation},
	VOLUME = {241},
     PAGES = {335--353},
      ISSN = {0378-4754},
       }

@article{talischi_polymesher_2012,
    AUTHOR = {Talischi, C. and Paulino, G. H. and P.,
              Anderson and Menezes, Ivan F. M.},
     TITLE = {{\tt {P}oly{M}esher}: a general-purpose mesh generator for polygonal elements written in {M}atlab},
   JOURNAL = {Struct. Multidiscip. Optim.},
  FJOURNAL = {Structural and Multidisciplinary Optimization},
    VOLUME = {45},
      YEAR = {2012},
    NUMBER = {3},
     PAGES = {309--328},
      ISSN = {1615-147X,1615-1488},
   MRCLASS = {74P15 (65-04 65N30 74-04)},
  MRNUMBER = {2897115},
       }

@article {Warburton_Hesthaven:2003,
    AUTHOR = {Warburton, T. and Hesthaven, J. S.},
     TITLE = {On the constants in {$hp$}-finite element trace inverse
              inequalities},
   JOURNAL = {Comput. Methods Appl. Mech. Engrg.},
  FJOURNAL = {Computer Methods in Applied Mechanics and Engineering},
    VOLUME = {192},
      YEAR = {2003},
    NUMBER = {25},
     PAGES = {2765--2773},
      ISSN = {0045-7825,1879-2138},
   MRCLASS = {65N30},
  MRNUMBER = {1986022},
       }

@book {DiPietro_Droniou:2020,
 author = {Di Pietro, D. A. and Droniou, J.},
 title = {The hybrid high-order method for polytopal meshes. {Design}, analysis, and applications},
 fseries = {MS\&A. Modeling, Simulation and Applications},
 series = {MS\&A, Model. Simul. Appl.},
 issn = {2037-5255},
 volume = {19},
 isbn = {978-3-030-37202-6; 978-3-030-37205-7; 978-3-030-37203-3},
 year = {2020},
 publisher = {Cham: Springer},
 language = {English},
 }

@misc{Hewett:2025,
 author = {Hewett, D. P.},
 title = {Piecewise polynomial approximation on non-{Lipschitz} domains},
 year = {2025},
 howpublished = {\href{https://doi.org/10.48550/arXiv.2511.22628}{arXiv:2511.22628}},
}

@article {Cangiani_Dong_Georgoulis:2021,
    AUTHOR = {Cangiani, A. and Dong, Z. and Georgoulis, E.
              H.},
     TITLE = {{$hp$}-version discontinuous {G}alerkin methods on essentially
              arbitrarily-shaped elements},
   JOURNAL = {Math. Comp.},
  FJOURNAL = {Mathematics of Computation},
    VOLUME = {91},
      YEAR = {2021},
    NUMBER = {333},
     PAGES = {1--35},
      ISSN = {0025-5718,1088-6842},
   MRCLASS = {65N30 (65J10 65N15)},
  MRNUMBER = {4350531},
       }

@article {Pazner:2020,
    AUTHOR = {Pazner, W.},
     TITLE = {Efficient low-order refined preconditioners for high-order
              matrix-free continuous and discontinuous {G}alerkin methods},
   JOURNAL = {SIAM J. Sci. Comput.},
  FJOURNAL = {SIAM Journal on Scientific Computing},
    VOLUME = {42},
      YEAR = {2020},
    NUMBER = {5},
     PAGES = {A3055--A3083},
      ISSN = {1064-8275,1095-7197},
   MRCLASS = {65F08 (65N30 65Y10)},
  MRNUMBER = {4157562},
MRREVIEWER = {Nasserdine\ Kechkar},
       }

@article {Schotzau_Schwab_Toselli:2002,
    AUTHOR = {Sch\"otzau, D. and Schwab, C. and Toselli, A.},
     TITLE = {Mixed {$hp$}-{DGFEM} for incompressible flows},
   JOURNAL = {SIAM J. Numer. Anal.},
  FJOURNAL = {SIAM Journal on Numerical Analysis},
    VOLUME = {40},
      YEAR = {2002},
    NUMBER = {6},
     PAGES = {2171--2194},
      ISSN = {0036-1429,1095-7170},
   MRCLASS = {65M30 (76D07 76M10)},
  MRNUMBER = {1974180},
       }

@book{Stein,
 author = {Stein, E. M.},
 title = {Singular integrals and differentiability properties of functions},
 fseries = {Princeton Mathematical Series},
 series = {Princeton Math. Ser.},
 volume = {30},
 year = {1970},
 publisher = {Princeton University Press, Princeton, NJ},
 language = {English},
 zbMATH = {3329342},
 Zbl = {0207.13501}
}

@article {Castillo_Cockburn_Perugia_SChotzau:2000,
    AUTHOR = {Castillo, P. and Cockburn, B. and Perugia, I. and
              Sch\"otzau, D.},
     TITLE = {An a priori error analysis of the local discontinuous
              {G}alerkin method for elliptic problems},
   JOURNAL = {SIAM J. Numer. Anal.},
  FJOURNAL = {SIAM Journal on Numerical Analysis},
    VOLUME = {38},
      YEAR = {2000},
    NUMBER = {5},
     PAGES = {1676--1706},
      ISSN = {0036-1429,1095-7170},
   MRCLASS = {65N15 (65N30)},
  MRNUMBER = {1813251},
       }

@article {Cangiani_Dong_Georgoulis:2017,
    AUTHOR = {Cangiani, A. and Dong, Z. and Georgoulis, E.
              H.},
     TITLE = {{$hp$}-version space-time discontinuous {G}alerkin methods for
              parabolic problems on prismatic meshes},
   JOURNAL = {SIAM J. Sci. Comput.},
  FJOURNAL = {SIAM Journal on Scientific Computing},
    VOLUME = {39},
      YEAR = {2017},
    NUMBER = {4},
     PAGES = {A1251--A1279},
      ISSN = {1064-8275,1095-7197},
   MRCLASS = {65M60 (65M15)},
  MRNUMBER = {3672375},
MRREVIEWER = {Jos\'e\ R.\ Fern\'andez},
}

@article {georgoulis_suboptimality_2010,
    AUTHOR = {Georgoulis, E. H. and Hall, E. and Melenk, J.
              M.},
     TITLE = {On the suboptimality of the {$p$}-version interior penalty
              discontinuous {G}alerkin method},
   JOURNAL = {J. Sci. Comput.},
  FJOURNAL = {Journal of Scientific Computing},
    VOLUME = {42},
      YEAR = {2010},
    NUMBER = {1},
     PAGES = {54--67},
      ISSN = {0885-7474,1573-7691},
   MRCLASS = {65N30},
  MRNUMBER = {2576364},
MRREVIEWER = {Christian\ Wieners},
}
